\documentclass[10pt, leqno]{article}

\usepackage{mathrsfs}
\usepackage{amsmath}
\usepackage{amssymb}
\usepackage{amsthm}
\usepackage{comment}
\usepackage[dvipdfmx]{hyperref}

\usepackage{accents}

\newtheorem{definition}{Definition}[section]
\newtheorem{them}{Theorem}[section]
\newtheorem{prop}{Proposition}[section]
\newtheorem{lem}{Lemma}[section]
\newtheorem{cor}{Corollary}[section]
\newtheorem{remark}{Remark}[section]

\numberwithin{equation}{section}

\begin{document}

\title{
Cauchy problem for operators with triple effectively hyperbolic characteristics\\ --\;Ivrii's conjecture\;--}

\author{Tatsuo Nishitani\footnote{Department of Mathematics, Osaka University:  
nishitani@math.sci.osaka-u.ac.jp
}}

\date{}
\maketitle

\def\dif{\partial}
\def\al{\alpha}
\def\be{\beta}
\def\ga{\gamma}
\def\om{\omega}
\def\lam{\lambda}
\def\ep{\epsilon}
\def\varep{\varepsilon}
\def\R{{\mathbb R}}
\def\N{{\mathbb N}}
\def\C{{\mathbb C}}
\def\Q{{\mathbb Q}}
\def\La{\varLambda}
\def\lr#1{\langle{#1}\rangle_{\gamma}}
\def\olr#1{\langle{#1}\rangle}
\def\lrD{\langle{D}\rangle}
\def\co{{\mathcal C}}
\def\op#1{{\rm op}({#1})}
\def\Xmaru{\overset{\circ}{X}}

\begin{abstract}
Ivrii's conjecture asserts that the Cauchy problem is $C^{\infty}$ well-posed for any lower order term if every singular point of the characteristic variety is effectively hyperbolic. Effectively hyperbolic singular point is at most triple characteristic. If every characteristic is at most double this conjecture has been proved in 1980'. In this paper we prove the conjecture for the remaining cases, that is for operators with triple effectively hyperbolic characteristics.
\end{abstract}

\smallskip
 {\footnotesize Keywords: Ivrii's conjecture, B\'ezout matrix, triple characteristic, effectively hyperbolic, Cauchy problem, Weyl calculus, Tricomi operator.}
 
 \smallskip
 {\footnotesize Mathematics Subject Classification 2010: Primary 35L30, Secondary 35G10}


\section{Introduction}
\label{sec:Intro}

This paper is devoted to the Cauchy problem 
\begin{equation}
\label{eq:CPm}
\left\{\begin{array}{ll}
Pu=D_t^mu+\sum_{j=0}^{m-1}\sum_{|\alpha|+j\leq m}a_{j,\alpha}(t,x)D_x^{\alpha}D_t^ju=0,\\[8pt]
D_t^ju(0,x)=u_j(x),\quad j=0,\ldots, m-1
\end{array}\right.
\end{equation}
where $t\geq 0$, $x\in\R^d$ and the coefficients $a_{j,\alpha}(t, x)$ are real valued $C^{\infty}$ functions in a neighborhood of the origin of $\R^{1+d}$ and $D_x=(D_{x_1},\ldots,D_{x_d})=D$, $D_{x_j}=(1/i)(\dif/\dif x_j)$ and $D_t=(1/i)(\dif/\dif t)$. The Cauchy problem \eqref{eq:CPm} is $C^{\infty}$ well-posed at the origin {\it for $t\geq 0$} if one can find a $\delta>0$ and a neighborhood $U$ of the origin of $\R^d$ such that \eqref{eq:CPm} has a unique solution $u\in C^{\infty}([0,\delta)\times U)$ for any $u_j(x)\in C^{\infty}(\R^d)$. We assume that the principal symbol of $P$
\[
p(t, x,\tau,\xi)=\tau^m+\sum_{j=0}^{m-1}\sum_{|\alpha|+j=m}a_{j,\alpha}(t, x)\xi^{\alpha}\tau^j
\]
 is hyperbolic {\it for $t\geq 0$}, that is there exist $\delta'>0$ and a neighborhood $U'$ of the origin such that
 \begin{equation}
 \label{eq:cond:hyp}
 \text{$p=0$ has only real roots in $\tau$ for  $(t, x)\in [0,\delta')\times U'$ and $\xi\in \R^d$}
 \end{equation}
which is indeed necessary in order that the Cauchy problem \eqref{eq:CPm} is $C^{\infty}$ well-posed near the origin for $t\geq 0$ (\cite{Lax}, \cite{Miz}). 

In \cite{IP}, Ivrii and Petkov proved that if the Cauchy problem \eqref{eq:CPm} is $C^{\infty}$ well posed for any lower order term near the origin for $t\geq 0$ (such operator $P$ is called strongly hyperbolic) then the Hamilton map $F_p$ has a pair of  non-zero real eigenvalues at every singular point of $p=0$ located in $[0,\delta'')\times U''\times (\R^{d+1}\setminus 0)$ (\cite[Theorem 3]{IP}) with some $\delta''>0$ and a neighborhood $U''$ of $x=0$. With $X=(t, x)$, $\Xi=(\tau,\xi)$ the Hamilton map $F_p$ is defined by
\[
F_p(X,\Xi)=\begin{bmatrix}\displaystyle{\frac{\dif^2 p}{\dif X\dif\Xi}}&\displaystyle{\frac{\dif^2 p}{\dif \Xi\dif\Xi}}\\[10pt]
\displaystyle{-\frac{\dif^2 p}{\dif X\dif X}}&\displaystyle{-\frac{\dif^2 p}{\dif \Xi\dif X}}
\end{bmatrix}.
\]
A singular point of $p=0$ where the Hamilton map $F_p$ has a pair of non-zero real eigenvalues  is called {\it effectively hyperbolic} (\cite{Ho1}, \cite{Ibook}).  
He has conjectured that the converse would be also true, that is if every  singular point  of $p=0$ is effectively hyperbolic then the Cauchy problem is $C^{\infty}$ well posed for any lower order term.  If a singular point $(t, x, \tau, \xi)$ is effectively hyperbolic then $\tau$ is a characteristic root of multiplicity at most $2$ if $t>0$ and at most $3$ if $t=0$ (\cite[Lemma 8.1]{IP}).  When every multiple characteristic root is at most double, the conjecture has been proved for some special class in \cite{I}, \cite{Mel} and for the general case in \cite{Iwa1, Iwa2, Iwa3}, \cite{Ni1, Ni2}.

For the case when we have triple effectively hyperbolic characteristic, Ivrii has also proved  in \cite{I} that the conjecture is true if   
$p$ admits a factorization $p=q_1q_2$ near singular points  with real smooth symbols  $q_i$, transforming the original $P$, by means of operator powers of evolution generators, to an operator for which a parametrix can be constructed.     In this case a singular point is effectively hyperbolic  if and only if the Poisson bracket $\{q_1,q_2\}$ does not vanish there. If $m=3$ it is clear that, for such a factorization to exist, it is necessary that the equation $p=0$ has a $C^{\infty}$ real root $\tau=\tau(t, x, \xi)$ near a conic neighborhood of singular points. 
A typical example is
\[
p=q_1q_2,\quad q_1=\tau^2-t|\xi|^2,\;q_2=\tau
\]
where $q_1$ is the Tricomi operator (symbol). Note that $p$ has a complex characteristic root if $t<0$. This is a common feature. In fact if $p(0, 0, \tau, \xi)=0$ has a triple characteristic root and $F_p(0,0,\tau,\xi)\neq O$  then $p$ has necessarily {\it non-real} characteristic roots in the $t<0$ side near $(0,0,\xi)$ (\cite[Lemma 8.1]{IP}).

When $m=3$, without restrictions we can assume that $p$ has the form 
\[
p=\tau^3-a(t, x, \xi)|\xi|^2\tau+b(t, x, \xi)|\xi|^3
\]
hence the condition \eqref{eq:cond:hyp} is reduced to $\Delta=4\,a^3-27\,b^3\geq 0$ for $(t, x)\in [0,\delta')\times U'$, $|\xi|=1$.  Also note that $(0, 0, 0, {\bar\xi})$ is a  triple characteristic then $(0, 0, 0, {\bar\xi})$ is effectively hyperbolic if and only if $\dif_ta(0, 0, {\bar \xi})\neq 0$ hence $a(t, x, \xi)>0$ microlocally for small $t>0$.  
In \cite{BBP}, Bernardi, Bove, Petkov investigated the case that $p$ has a triple effectively hyperbolic characteristic but $p$ may not be be factorized. They studied $P$ with the principal symbol
\begin{gather*}
p=\tau^3-(ta_2(t, x, \xi)+\alpha(x,\xi))|\xi|^2\tau+t^2b_3(t, x, \xi)|\xi|^3
\end{gather*}
where $a_2(t, x,\xi)\geq c'>0$, $\alpha(x, \xi)\geq 0$ and proved the conjecture for such $P$, deriving weighted energy estimates by separating (multiplier) operator method. Note that $\Delta\geq c\,a^3$ holds with some $c>0$ for this $p$. They also proved that if $b_3(0, 0, {\bar\xi})\neq 0$ then a smooth factorization $p=q_1q_2$ is possible only if $\al(x, {\bar\xi})=0$ for all $x$ near $x=0$. This result was extended in \cite{NP, NP2} such that the conjecture is true if $\Delta\geq c\,ta^2$ or if $\Delta\geq c\,t^2 a$ with some additional conditions, where after reducing the original equation to a first order $3\times 3$ system, a symmetrizer $S$ is constructed and used  to get weighted energy estimates. These results are concerned with the case that $p$ is strictly hyperbolic in $t>0$, while in general case, double effectively hyperbolic characteristics in $t>0$ (where $\Delta=0$) approaching to  a triple effectively  hyperbolic characteristic on $t=0$ might exist and we must handle them.   
Moreover the following example (\cite{NP}) 
\begin{equation}
\label{eq:examp:3}
p(t, x, \tau, \xi) = \tau^3 -(t + \alpha(x))\xi^2\tau  + (t^{m}/2 - t )\sqrt{\alpha(x)}\,\xi^3,\;\;x, \xi\in\R
\end{equation}
where $\alpha(x) \geq 0$ and $\sqrt{\alpha(x)}$ is smooth, suggests that it is not enough just to study the zeros themselves of $\Delta$. Indeed since $
\Delta =  ( t - 2\alpha)^2 (4 t + \alpha) + 27 t^{m+1} \alpha( 1-  t^{m-1}/4)$ 
so that $\Delta >0$ for small $t>0$, while $\Delta = 27\cdot 2^{m+1} \alpha^{m+2}(1 -  2^{m-3}\alpha^{m-1})$ if $ t = 2 \alpha$ hence one has no estimate such as $
\Delta \geq c\, t^k( t + \alpha)^q$ with $c > 0$
for small $ \alpha> 0$ if $m > k + q - 2$.

In \cite{Ni4} we employed a new idea which is to diagonalize the symmetrizer $S$ mentioned above so that the system is transformed to  a system with a diagonal symmetrizer. We see that three diagonal entries (the eigenvalues of $S$) are bounded from below by $\Delta/a$, $a$, $1$ respectively and  we recognize here a close relation between  the diagonal symmetrizer and two discriminants $\Delta$ of $p=0$ and $\Delta'(=a)$ of $\dif_{\tau}p=0$. In example \eqref{eq:examp:3} we see $\Delta/a\geq (t-2\al)^2$ which looks like $\tau^2-(\Delta/a)|\xi|^2$  has double effectively hyperbolic characteristics on $t=2\al$ though $\Delta\neq0 $ there (see \cite{Ni2, Ni:book}). When the coefficients of $p$ depend only on $t$, the behavior of $\Delta(t, \xi)/a(t,\xi)$ can be analyzed  relatively easily. Writing $\Delta(t, \xi)=e\prod(t-\nu_j(\xi))$ and dividing $[0, T]$ into subintervals with the end points ${\mathsf{Re}}\,\nu_j(\xi)$ we can obtain suitable estimates of $\Delta/a$ from below in each subinterval. In particular, in this way, we have proved the conjecture for the $t$ dependent case \cite[Theorem 4.1]{Ni4}. In this paper we extend  this idea and apply it to the general case of which outline is given in the next section.

For this class of operators  we have always a loss of regularity  
then the way to obtain microlocal energy estimates for operators of order $m$ from that of order $2$ or $3$ and the way to prove local $C^{\infty}$ well-posed results  from such obtained microlocal energy estimates with loss of regularity is not so straightforward.

Finally we note that if there is a triple characteristic which is not effectively hyperbolic the Cauchy problem is not well-posed in any Gevrey class of order $s>2$  in general,  even though the subprincipal symbol vanishes identically (\cite{BeNi}).

In this paper we prove
\begin{them}
\label{thm:main}Assume \eqref{eq:cond:hyp}. If every  singular point $(0,0,\tau,\xi)$,  $|(\tau, \xi)|\neq 0$ of $p=0$ is effectively hyperbolic then for any $a_{j,\alpha}(t, x)$ with $j+|\alpha|\leq m-1$, which are $C^{\infty}$ in a neighborhood of $(0,0)$, there exist $\delta>0$, a neighborhood $U$  of the origin and $n> 0$, $\ell>0$ such that  for any $s\in\R$ and any $f$ with $t^{-n+1/2}\langle{D}\rangle^{s}f\in L^2((0,\delta)\times \R^d)$ there exists $u$ with $t^{-n-1/2}\langle{D}\rangle^{-\ell+s+m-j}D_t^ju\in L^2((0,\delta)\times \R^d)$, $j=0,1,\ldots,m-1$,  satisfying $Pu=f$ in $(0,\delta)\times U$.
\end{them}
Here $\olr{D}$ stands for$\sqrt{1+|D|^2}$ and $n$, $\ell$ are given
by
\[
n=12\sqrt{2}\,\sup\frac{|P_{sub}(0, 0, \tau, \xi)|}{e(0, 0, \tau, \xi)}+{\bar C^*},\quad \ell=k\,(n+2)
\]
where $P_{sub}$ denotes the subprincipal symbol and $e(0, 0, \tau, \xi)$ is the positive real eigenvalue of $F_p(0, 0, \tau, \xi)$ and the supremum is taken over all singular points $(0, 0, \tau, \xi)$ with $|(\tau, \xi)|=1$ of $p=0$ and ${\bar C^*}$ is a constant depending only on the principal symbol $p$. Here $k$ is the maximal number of singular points $(0, 0, \tau, \xi)$ of $p=0$ with $|(\tau,\xi)|=1$ hence $k\leq [m/2]$. For more detailed estimate of ${\bar C^*}$ see \eqref{eq:pert:Dis} and \eqref{eq:n:P:sub} below. The constant $12\sqrt{2}$ may not be the best.
\begin{them}
\label{thm:main:bis} Under the same assumption as in Theorem \ref{thm:main},  for any $a_{j,\alpha}(t, x)$ with $j+|\alpha|\leq m-1$, which are $C^{\infty}$ in a neighborhood of $(0,0)$, there exist $\delta>0$ and a neighborhood $U$  of the origin such that for any $u_j(x)\in C^{\infty}_0(\R^d)$, $j=0,1,\ldots, m-1$, there exists  $u(t,x)\in C^{\infty}([0,\delta)\times U)$ satisfying \eqref{eq:CPm} in $[0,\delta)\times U$. If $u(t,x)\in C^{\infty}([0,\delta)\times U)$ with $\dif_t^j u(0, x)=0$, $j=0,1,\ldots, m-1$, satisfies $Pu=0$ in $[0,\delta)\times U$ then $u=0$ in a neighborhood of $(0,0)$.
\end{them}
\begin{proof} Compute $u_j(x)=D_t^ju(0,x)$ for $j=m,m+1,\ldots$ from $u_j(x)$, $j=0,1,\ldots, m-1$ and the equation $Pu=0$. By a Borel's lemma there is $w(t, x)\in C_0^{\infty}(\R^{1+d})$ such that $D_t^j w(0, x)=u_j(x)$ for all $j\in \N$.  Since $(D_t^jPw)(0,x)=0$ for all $j\in\N$  it is clear that $t^{-n+1/2}\langle{D}\rangle^s Pw\in L^2((0,\delta)\times \R^d)$ for any $s$. Thanks to Theorem \ref{thm:main} there exists $v$ with $t^{-n-1/2}\langle{D}\rangle^{-\ell+s+m-j}D_t^j v\in L^2((0,\delta)\times \R^d)$, $j=0,1,\ldots, m-1$ satisfying $Pv=-Pw$ in $(0,\delta)\times U$. Since $D_t^k v\in L^2((0,\delta)\times \R^d)$ for any $k$ hence $D_t^j v(0,x)=0$, $j=0,1,\ldots, m-1$ we conclude that $u=v+w$ is a desired solution. Local uniqueness follows from Theorem \ref{thm:local:itii} below because $\dif_t^ku(0, x)=0$ for any $k\in \N$ by $Pu=0$.
\end{proof}
%
%

\section{Outline of the proof of Theorem \ref{thm:main}}

As noted in Introduction, if a singular point $(t, x,\tau, \xi)$ of $p=0$ is effectively hyperbolic then $\tau$ is a characteristic root of multiplicity at most $3$. This implies that 
it is essential to study   third order operator $P$; 
\begin{equation}
\label{eq:takata}
P=D_t^3+\sum_{j=1}^3a_j(t, x, D)\langle{D}\rangle^jD_t^{3-j}
\end{equation}
which is a differential operator in $t$ with coefficients $a_{j}\in S^0$, classical pseudodifferential operators of order $0$, where $\langle{D}\rangle = \op{(1 + |\xi|^2)^{1/2}}$. One can reduce $P$ to the case with $a_1(t, x, D)=0$  and hence the principal symbol  is
\begin{equation}
\label{eq:sokyoku}
p(t, x, \tau, \xi) = \tau^3 -a(t, x, \xi)\olr{\xi}^2\tau
- b(t, x, \xi)\olr{\xi}^3. 
\end{equation}
All characteristic roots are real for $t\geq 0$ implies that
\begin{equation}
\label{eq:hanbetu}
\Delta=4\,a(t, x,\xi)^3-27\,b(t, x,\xi)^2\geq 0, \quad (t, x,\xi)\in [0,T)\times U\times \R^d.
\end{equation}
Assume that $p(0,0,\tau,{\bar\xi})=0$ has a triple characteristic root $\tau={\bar\tau}$, necessarily ${\bar\tau}=0$. The singular point $(0,0,{\bar\tau},{\bar\xi})$ is effectively hyperbolic if and only if
\begin{equation}
\label{eq:outl:b}
\dif_ta(0,0,{\bar\xi})\neq 0
\end{equation}
hence one can write $a=e(t+\al(x,\xi))$ where $e> 0$ and $\alpha\geq 0$. From the conditions \eqref{eq:hanbetu} and \eqref{eq:outl:b} the discriminant $\Delta$ is essentially a third order polynomial in $t$. In Section \ref{sec:const:psi}, for regularized $a_{\ep}=e(t+\al+\ep^2)$ and the corresponding discriminant $\Delta_{\ep}$ we construct a smooth $\psi_{\ep}$ such that
\begin{equation}
\label{eq:outl:c:ep}
{\Delta_{\ep}}\geq { c}\min{\big\{t^2, (t-\psi_{\ep})^2\big\}}\,(t+\rho_{\ep}),\quad t\geq 0
\end{equation}
where $\rho_{\ep}=\al+\ep^2$. In Section \ref{sec:kakutyo}, introducing a large parameter $M$, we localize the coefficients  near reference point $(0, {\bar\xi})$ replacing $(x, \xi)$ by  localized coordinates $\chi(M^2x)x$, $\chi(M^2(\xi/\olr{\xi}-{\bar\xi}))(\xi-{\bar\xi}\olr{\xi})+{\bar\xi}\olr{\xi}$ and taking $\ep=M^{1/2}\olr{\xi}^{-1/2}$ where $\chi\in C_0^{\infty}$ is $1$ near $0$. At this point related symbols are localized near $(0,{\bar\xi})$ (defined in $\R^d\times\R^d$) and \eqref{eq:outl:c:ep} yields
\begin{equation}
\label{eq:outl:c}
\Delta/a\geq c\min{\big\{t^2, (t-\psi)^2+M\rho\langle{\xi}\rangle^{-1}\big\}},\quad t\geq 0.
\end{equation}
We also estimate such localized symbols in terms of  the localized $\rho$. In particular we show that
\[
\big|\dif_x^{\al}\dif_{\xi}^{\be}\psi(x,\xi)\big|\precsim \rho(x,\xi)^{1-|\alpha+\be|/2}\olr{\xi}^{-|\be|}
\]
where, and from now on  $A \precsim B$ means that $A$ is bounded by constant, independent of parameters, times $B$.

One of main arguments in   the paper is to reduce the original equation to a first order $3\times 3$  system with diagonal symmetrizer. 
With $U={^t}(D_t^2u,\olr{D} D_tu,\olr{D}^2u)$ the equation $P u=f$ is reduced to $
D_tU=A(t, x, D)\olr{D} U+B(t, x, D)U+F$ 
where $A, B\in S^0$, $F={^t}(f,0,0)$ and
\[
 A(t,x,\xi)=
\begin{bmatrix}0&a&b\\
1&0 &0\\
0&1&0
\end{bmatrix}.
\]
Let $S$ be  the B\'ezout matrix  of $p$ and $\dif p/\dif \tau$, that is
\[
S(t,x,\xi)=
\begin{bmatrix}3&0&-a\\
0&2a&3b\\
-a&3b&a^2
\end{bmatrix}
\]
then $S$ is positive semidefinite and symmetrizes $A$, that is $SA$ is symmetric.  We now diagonalize $S$ by an orthogonal matrix $T$  so that $T^{-1}ST=\varLambda$. Then with $V=\op{T^{-1}}U$ the system is reduced to, roughly
\begin{equation}
\label{eq:daitai}
D_tV=A^T(t, x, D)\langle{D}\rangle V+{\tilde B}(t, x, D)V
\end{equation}
where $\varLambda$  is diagonal and symmetrizes $A^T(t, x, \xi)=T^{-1}AT$. This reduction  is carried out in Section \ref{sec:henkan:jiko} after examining  $T(t, x, \xi)$ carefully   in Section \ref{sec:Bezout}.

Note that $\varLambda={\rm diag}\,(\lambda_1,\lambda_2,\lambda_3)$  where $0< \lambda_1< \lambda_2< \lambda_3$ are the eigenvalues of $S$.  
As mentioned in Introduction a significant feature of $\lambda_j$ are
\begin{equation}
\label{eq:koyuti:sita:ue}
\Delta/a\precsim \lambda_1\precsim a^2,\quad \lambda_2\simeq a,\quad \lambda_3\simeq 1,\quad t\geq 0.
\end{equation}
Section \ref{sec:Bezout}  is devoted mainly to estimate derivatives of $\lambda_j$ and we prove 
\[
\big|\dif_x^{\al}\dif_{\xi}^{\be}\lambda_j\big|\precsim a^{3-j-|\al+\be|/2}\langle{\xi}\rangle^{-|\be|},\quad  t\geq 0, \quad j=1, 2, 3
\]
which also gives detailed information on the derivatives of $T$. 
Since \eqref{eq:daitai} is a system with diagonal symmetrizer $\varLambda$, a natural energy would be $
\big(\op{\varLambda}V,V\big)=\sum_{j=1}^3\big(\op{\lambda_j}V_j,V_j\big)$
and \eqref{eq:outl:c} and \eqref{eq:koyuti:sita:ue} suggest that  a weighted energy with a scalar pseudodifferential weight  $\op{t^{-n}\phi^{-n}}$  
\[
\phi=\omega+t-\psi,\quad \omega=\sqrt{(t-\psi)^2+M\rho\langle{\xi}\rangle^{-1}}
\]
would work, where $\op{\phi^{-n}}$ is chosen after the weight  employed for studying double effectively hyperbolic characteristics in \cite{Ni2} (see also  \cite{Ni:book}), and satisfies
\[
\dif_t(t\phi)=\kappa\, (t\phi),\quad \kappa=1/t+1/\omega.
\]
In Section \ref{sec:kyori}, to treat these weight functions, we introduce a $\sigma$ temperate (uniformly in $M$) metric 
\[
g=M^{-1}(\olr{\xi}|dx|^2+\olr{\xi}^{-1}|dx|^2)
\]
and prove that $\omega^s\in S(\omega^s, g)$, $\phi^s\in S(\phi^s, g)$ with $s\in\R$,  uniformly in $t\geq 0$, estimating derivatives of $\omega$, $\phi$. In Section \ref{sec:metric} we prove that $\omega$, $\phi$ and $\lambda_j$ are $\sigma, g$ temperate uniformly in $t\in[0, M^{-4}]$ (in this paper such functions are called admissible weights for $g$, while  $\sigma$ is reserved for denoting a certain function). This fact enables us to apply the Weyl calculus of pseudodifferential operators (see \cite[Chapter 18]{Hobook}) to  $\op{\phi^{s}}$, $\op{\omega^s}$ and $\op{\lambda_j^s}$ with $s\in\R$, for example we have $\op{\phi^{s_1}}\op{\phi^{s_2}}=\op{\phi^{s_1}\#\phi^{s_2}}$ where $\phi^{s_1}\#\phi^{s_2}\in S(\phi^{s_1+s_2}, g)$.  In Section \ref{sec:bound:lam} we  prove some basic facts on  inverses and $L^2$ bounds of pseudodifferential operators associated to the metric $g$ which enables us,  for example, to write $\op{\phi^{s_1}}\op{\phi^{s_2}}=\op{1+r}\op{\phi^{s_1+s_2}}$ with $r\in S(M^{-1}, g)$. We also give  lower bounds of $\op{\lam_j}$ here. 

In Section \ref{sec:mitibiku:ene} applying the Weyl calculus of pseudodifferential operators we  estimate  the weighted energy 
\[
{\mathsf{Re}}\,e^{-\theta t}(\op{\varLambda}\op{t^{-n}\phi^{-n}}V, \op{t^{-n}\phi^{-n}}V)
\]
and derive energy estimates for any lower order term (including the term $M\olr{D}D_t$ because  we have added this to the original operator at the beginning) (Proposition \ref{pro:fu:matome}). In Section \ref{sec:pre:sonzai}, using energy estimates for the system coming from the adjoint operator of $P$, which is obtained repeating exactly the same arguments,  we prove an existence result of the Cauchy problem for $P_{\bar \xi}$, which is the localized operator near $(0, {\bar\xi})$ of the original $P$ (Theorem\ref{thm:pre:sonzai:s}). In Section \ref{sec:kyokusho}, in order to sum up such obtained solutions  (which might be considered as a microlocal solution to the Cauchy problem near $(0,{\bar\xi})$),  we prove that the wave front sets of such solutions  propagate with finite speed (Proposition \ref{pro:yugen:denpa}). A more precise picture of the propagation of wave front set of solutions  is also proved  applying the same arguments (Theorem \ref{thm:pro:WF}). Finally, in Section \ref{sec:syuteiri}, using the propagation results in Section \ref{sec:kyokusho} we complete the proof of Theorem \ref{thm:main}.


\section{Construction of $\psi(x,\xi)$}
\label{sec:const:psi}

Study  third order 
operators $P$ of the form \eqref{eq:takata} with $a_1(t, x, D)=0$, 
hence the principal symbol  has the form \eqref{eq:sokyoku} 
where $a(t, x,\xi)$ and $b(t, x,\xi)$ are homogeneous of degree $0$ in $\xi$ and assumed to satisfy \eqref{eq:hanbetu} 
with some $T>0$ and some neighborhood $U$ of the origin of $\R^d$. 
Assume that $p(t, x,\tau,\xi)$ has a triple characteristic root  $\tau=0$ at $(0, 0, {\bar\xi})$, $|{\bar\xi}|=1$ and $(0, 0, 0,{\bar\xi})$ is effectively hyperbolic.  It is clear that $a(0,0,{\bar\xi})=0$ and $b(0,0,{\bar\xi})=0$. 
Since $\dif_x^{\al}\dif_{\xi}^{\be}a(0,0,{\bar\xi})=0$ for $|\al+\be|=1$  and $\dif_x^{\al}\dif_{\xi}^{\be}b(0,0,{\bar\xi})=0$ for $|\al+\be|\leq 2$ by \eqref{eq:hanbetu} 
(see Lemma \ref{lem:a:bound:rho} below) it is easy to see that
\begin{equation}
\label{eq:koyuti}
{\rm det}\,\big(\lambda-F_p(0,0,0,{\bar\xi})\big)=\lambda^{2d}\big(\lambda^2-\{\dif_t a(0,0,{\bar\xi})\}^2\big)
\end{equation}
hence $(0,0,0,{\bar\xi})$ is effectively hyperbolic if and only if $
\dif_t a(0,0,{\bar\xi})\neq 0$. Then  there is a neighborhood ${\mathcal U}$ of $(0,0,{\bar\xi})$ in which one can write
\[
a(t, x,\xi)=e(t, x,\xi)(t+\alpha(x,\xi))
\]
where $e>0$ in ${\mathcal U}$. Note that $\alpha(x,\xi)\geq 0$ near ${\bar\xi}$ because $a(t, x,\xi)\geq 0$ in $[0,T)\times U\times \R^d$.
%

\subsection{A perturbed discriminant}
\label{subsec:coeff:dis}

Introducing a small  parameter $\ep$ we consider
\begin{equation}
\label{eq:ep+:Dis}
\begin{split}
\tau^3-e(t, x,\xi)(t+\al(x,\xi)+\ep^2)|\xi|^2\tau-b(t, x,\xi)\,|\xi|^3\\
=\tau^3-a(t, x,\xi,\ep)|\xi|^2-b(t, x,\xi)|\xi|^3.
\end{split}
\end{equation}
From now on we write $b(X)$ or $a(X, \ep)$  and so on to make clearer that these symbols are defined  in some conic (in $\xi$) neighborhood of ${\bar X}=(0,{\bar\xi})$  or $({\bar X}, 0)$. Consider the discriminant of \eqref{eq:ep+:Dis}; $
\Delta(t,X,\ep)=4\,a^3(t,X,\ep)-27\,b^2(t,X)$.
\begin{lem}
\label{lem:DMal}One can write
\[
\Delta={\tilde e}(t,X,\ep)\big(t^3+a_1(X,\ep)t^2+a_2(X,\ep)t+a_3(X,\ep)\big)
\]
in a neighborhood of $(0,{\bar X},0)$ where $a_j({\bar X},0)=0$, $j=1,2,3$ and ${\tilde e}>0$. 
\end{lem}
\begin{proof}
It is clear that $\dif_t^k a^3(0,{\bar X},0)=0$ for $k=0,1,2$ and $\dif_t^3 a^3(0,{\bar X},0)\neq 0$. Show $\dif_t b(0,{\bar X},0)=0$. Suppose the contrary and hence $
b(t,{\bar X}, 0)=t(b_1+tb_2(t))$ 
with $b_1\neq 0$. Since $a(t,{\bar X},0)=c\,t$ with $c>0$ then $\Delta(t,{\bar X},0)=4\,c^3\,t^3-27\,b(t,{\bar X},0)^2\geq 0$ leads to a contradiction. Thus $\dif_t^k\Delta(0,{\bar X},0)=0$ for $k=0,1,2$ and $\dif_t^3\Delta(0,{\bar X},0)\neq 0$. Then from the Malgrange preparation theorem (e.g. \cite[Theorem 7.5.5]{Hobook:1}) one can conclude the assertion.
\end{proof}
Introducing 
\begin{equation}
\label{eq:rho:teigi}
\rho(X,\ep)=\al(X)+\ep^2
\end{equation}
 one can also write 
\begin{align*}
\Delta=4e^3(t+\rho)^3-27b^2=4e^3\big\{(t+\rho)^3-27b^2/(4e^3)\big\}=4e^3\big\{(t+\rho)^3-{\hat b}^2\big\}
\end{align*}
with ${\hat b}=3\sqrt{3}\,b/2e^{3/2}$. Denoting $
{\hat b}(t,X)=\sum_{j=0}^2{\hat b}_j(X)t^j+{\hat b}_3(t,X)t^3$ 
where ${\hat b}_0({\bar X})={\hat b}_1({\bar X})=0$ which is clear from the proof of Lemma \ref{lem:DMal}, one can write
\begin{equation}
\label{eq:Del:two}
\begin{split}
\Delta/{\tilde e}={\bar \Delta}=t^3+a_1(X,\ep)t^2+a_2(X,\ep)t+a_3(X,\ep)\\
=E\,\big\{(t+\rho)^3-\big(\sum_{j=0}^2{\hat b}_j(X)t^j+{\hat b}_3(t,X)t^3\big)^2\big\}
\end{split}
\end{equation}
with $E(t,X,\ep)=4e^3/{\tilde e}$. Here note that $
E(0,{\bar X}, 0)=1$. 
\begin{lem}
\label{lem:b_1:alpha} There is a neighborhood $V$ of ${\bar X}$ such that $
\big|{\hat b}_1(X)\big|\leq 4\,\al^{1/2}(X)$ for $ X\in V$.
\end{lem}
\begin{proof}It is clear that $|{\hat b}_0(X)|\leq \al^{3/2}(X)$. If $\al(X)=0$ then the assertion is obvious.  Assume $\al(X)\neq 0$.
Since 
\begin{equation}
\label{eq:b_1:alpha}
(t+\al(X))^3\geq \big(\sum_{j=0}^2{\hat b}_j(X)t^j+{\hat b}_3(t,X)t^3\big)^2, \quad 0\leq t\leq T
\end{equation}
choosing $t=3\al(X)\leq T$ and writing $\al=\al(X)$ it follows from \eqref{eq:b_1:alpha} that
\begin{align*}
8\al^{3/2}\geq \big|{\hat b}_0(X)+3{\hat b}_1(X)\al\big|-C\al^2
\geq 3|{\hat b}_1(X)|\al-C\al^2-\al^{3/2}
\end{align*}
hence the assertion is clear because $\al({\bar X})=0$.
\end{proof}
\begin{lem}
\label{lem:a_j:al}In a neighborhood of $({\bar X},0)$ we have $a_j(X,\ep)=O\big(\rho(X,\ep)^j\big)$ for $ j=1,2,3$. More precisely
\begin{gather*}
a_1(X,\ep)=E(0,X,\ep)\big(3\rho(X,\ep)-{\hat b}_1^2(X)\big)+O(\rho^{3/2}),\\
a_2(X,\ep)=E(0,X,\ep)\big(3\rho^2(X,\ep)-2{\hat b}_0(X){\hat b}_1(X)\big)+O(\rho^{3/2}),\\
a_3(X,\ep)=E(0,X,\ep)\big(\rho^3(X,\ep)-{\hat b}^2_0(X)\big).
\end{gather*}
\end{lem}
\begin{proof} Since ${\bar \Delta}(0,X,\ep)\geq 0$ it follows from \eqref{eq:Del:two} that $a_3(X,\ep)=E(0,X,\ep)\big(\rho(X,\ep)^3-{\hat b}_0(X)^2\big)\geq 0$ hence ${\hat b}_0=O(\rho^{3/2})$ and consequently $a_3(X,\ep)=O(\rho^3)$. From \eqref{eq:Del:two}
\begin{align*}
a_2(X,\ep)=\dif_t E(0,X,\ep)a_3(X,\ep)
+E(0,X,\ep)\big(3\rho^2(X,\ep)-2{\hat b}_0(X){\hat b}_1(X)\big).
\end{align*}
Since ${\hat b}_0(X){\hat b}_1(X)=O(\rho^2)$ by Lemma  \ref{lem:b_1:alpha}  hence the above equality shows the assertion for $a_2(X,\ep)$. Finally from \eqref{eq:Del:two} again
\begin{gather*}
2a_1(X,\ep)=\dif_t^2E(0,X,\ep)a_3(X,\ep)
+2\dif_t E(0,X,\ep)\big(3\rho^2(X,\ep)-2{\hat b}_0(X){\hat b}_1(X)\big)
\\+2E(0,X,\ep)\big(3\rho(X,\ep)-{\hat b}_1(X)^2-2{\hat b}_0(X){\hat b}_2(X)\big)
\end{gather*}
and Lemma \ref{lem:b_1:alpha} one concludes the assertion for $a_1(X,\ep)$.
\end{proof}
%

\subsection{Lower bound of perturbed discriminant}
\label{subsec:bound:dis}

Denote
\begin{equation}
\label{eq:defnu}
\nu(X,\ep)=\inf\{t\mid {\bar \Delta}(t,X,\ep)>0\}
\end{equation}
and hence ${\bar \Delta}(\nu, X,\ep)=0$. First check that $\nu(X,\ep)\leq 0$. Suppose the contrary $\nu(X,\ep)=\nu>0$. Since ${\bar\Delta}(t,X,\ep)\geq 0$ for $t\geq 0$ one can write ${\bar \Delta}(t)=(t-\nu)^2(t-{\tilde \nu})$ with a real ${\tilde \nu}$ where ${\tilde \nu}\neq \nu$   and ${\tilde \nu}\leq 0$. Therefore we have ${\bar\Delta}(t)>0$ in ${\tilde \nu}<t<\nu$ which is incompatible with the definition of $\nu$. Write
\[
{\bar \Delta}(t, X,\ep)=(t-\nu(X,\ep))(t^2+A_1(X,\ep)t+A_2(X,\ep))
\]
where $A_1=\nu+a_1$. Here we prepare following lemma.
\begin{lem}
\label{lem:nu_j:al}
One can find a neighborhood ${\mathcal U}$ of $({\bar X},0)$ such that for any $(X,\ep)\in {\mathcal U}$ there is $j\in\{1,2,3\}$ such that $
|\nu_j(X,\ep)|\geq \rho(X,\ep)/9$ 
where ${\bar \Delta}(t,X,\ep)=\prod_{j=1}^3(t-\nu_j(X,\ep))$.
\end{lem}
\begin{proof}
First show that there is $1/3<\delta<1/2$ such that
\begin{equation}
\label{eq:3ko:hikaku}
\max{\big\{\big|\rho^3-{\hat b}_0^2\big|^{1/3}, \;\;\big|\rho^2-2{\hat b}_0{\hat b}_1/3\big|^{1/2}, \;\;\big|\rho-{\hat b}_1^2/3\big|\big\}}\geq \delta^2\rho.
\end{equation}
In fact denoting $f(\delta)=2(1-\delta^6)^{1/2}(1-\delta^2)^{1/2}/\sqrt{3}-1-\delta^4$ it is easy to check that $f(1/3)>0$ and $f(1/2)<0$.  Take $1/3<\delta<1/2$ such that $f(\delta)=0$. If $|\rho^3-{\hat b}_0^2|^{1/3}<\delta^2\rho$ and $|\rho-{\hat b}_1^2/3|<\delta^2\rho$ then $|{\hat b}_0|\geq (1-\delta^6)^{1/2}\rho^{3/2}$ and $|{\hat b}_1|\geq \sqrt{3}(1-\delta^2)^{1/2}\rho^{1/2}$ hence
\begin{align*}
|\rho^2-2{\hat b}_0{\hat b}_1/3|\geq 2|{\hat b}_0{\hat b}_1|/3-\rho^2\geq \big(f(\delta)+\delta^4\big)\rho^2=\delta^4\rho^2.
\end{align*}
Thus \eqref{eq:3ko:hikaku} is proved. Thanks to Lemma \ref{lem:a_j:al}, taking $E(0, {\bar X}, 0)=1$ and $1/3<\delta$ into account, one can find a neighborhood ${\mathcal U}$ of $({\bar X}, 0)$ such that  
\[
\max{\big\{3|a_1(X, \ep)|, \; (3^3|a_2(X,\ep)|)^{1/2},\; (3^6|a_3(X,\ep)|)^{1/3}\big\}}\geq \rho,\quad (X,\ep)\in {\mathcal U}.
\]
Then the assertion follows from the relations between $\{\nu_i\}$ and $\{a_i\}$.
\end{proof}
\begin{lem}
\label{lem:nu:1sa:nuj}
Denote $\nu$ defined in \eqref{eq:defnu} by $\nu_1$ and by $\nu_j$, $j=2,3$ the other roots of ${\bar \Delta}=0$.  Then one can find a neighborhood ${\mathcal U}$ of $({\bar X},0)$ and $c_i>0$ such that
\begin{equation}
\label{eq:nu_1toj}
\text{if}\quad \nu_1+a_1<2\,c_1\,\rho,\;\;(X,\ep)\in {\mathcal U}\quad\text{then}\quad  |\nu_1-\nu_j|\geq c_2\,\rho,\;\;j=2,3.
\end{equation}
In particular $\nu_1(X,\ep)$ is smooth  in ${\mathcal U}\cap\{ \nu_1+a_1<2\,c_1\rho\}$.
\end{lem}
\begin{proof}  Write
\begin{equation}
\label{eq:Del:bunkai}
{\bar \Delta}(t)=\prod_{j=1}^3(t-\nu_j)=(t-\nu_1)((t+A_1/2)^2-D)
\end{equation}
so that ${\mathsf{Re}}\,\nu_j=-A_1/2$, $j=2, 3$ where $A_1=\nu_1+a_1$. 
Take $c_1=1/27$ and assume $A_1<2c_1\rho$. First note that if ${\mathsf{Re}}\,\nu_j\geq c_1\rho$, $j=2,3$ it is clear that $
|\nu_1-\nu_j|\geq |\nu_1-{\mathsf{Re}}\,\nu_j|\geq {\mathsf{Re}}\,\nu_j\geq c_1\rho$ because $\nu_1\leq 0$ then we may assume
\begin{equation}
\label{eq:nuj:bound}
-c_1\rho<{\mathsf{Re}}\,\nu_j=-A_1/2<c_1\rho,\quad j=2,3.
\end{equation}
If $D>0$ then one has $-A_1/2+\sqrt{D}\leq 0$. Otherwise ${\bar \Delta}(t)$ would be negative for some $t>0$ near $-A_1/2+\sqrt{D}$ which is a contradiction. Thus $\sqrt{D}\leq A_1/2\leq c_1\,\rho$ which shows that $
|\nu_2|$, $|\nu_3|\leq |A_1|/2+\sqrt{D}\leq 2c_1\,\rho<\rho/9$ 
hence $|\nu_1|\geq \rho/9=3c_1\rho$ by Lemma \ref{lem:nu_j:al}. Therefore $
|\nu_1-\nu_j|\geq |\nu_1|-|\nu_j|\geq c_1 \rho$. 
Turn to the case $D\leq 0$ such that $\nu_2$, $\nu_3=-A_1/2\pm i\sqrt{|D|}$. Thanks to Lemma \ref{lem:nu_j:al} again either $|\nu_1|\geq 3c_1\rho$ or $|\nu_2|=|\nu_3|\geq 3c_1\rho$. If $|\nu_1|\geq 3c_1\rho$ then it follows from \eqref{eq:nuj:bound} that
\[
|\nu_1-\nu_j|\geq |\nu_1+A_1/2|\geq |\nu_1|-|A_1|/2\geq 2c_1 \rho.
\]
If $|\nu_2|=|\nu_3|\geq 3c_1\rho$ so that $|A_1|/2+\sqrt{|D|}\geq 3c_1\rho$ hence $\sqrt{|D|}\geq 3c_1\rho-|A_1|/2\geq 2c_1\rho$ which proves $
|\nu_1-\nu_j|\geq \sqrt{|D|}\geq 2c_1\rho$ 
hence the assertion. 
\end{proof}
Now define $\psi(X,\ep)$ which plays a crucial role in our arguments deriving weighted energy estimates. Choose $\chi(s)\in C^{\infty}(\R)$ such that $0\leq \chi(s)\leq 1$ with $\chi(s)=1$ if $s\leq 0$ and $\chi(s)=0$ for $s\geq 1$.  Define
\[
\psi(X,\ep)=-\chi\Big(\frac{\nu_1+a_1}{2c_1\rho}\Big)\frac{\nu_1+a_1}{2},\quad \ep\neq 0.
\]
\begin{prop}
\label{pro:bound:Dis}One can find a neighborhood ${\mathcal U}$ of $({\bar X},0)$  such that
\begin{equation}
\label{eq:bound:Dis}
{\bar \Delta}(t,X,\ep)\geq {\bar c}\min{\big\{t^2, (t-\psi(X,\ep))^2\big\}}\,(t+\rho(X,\ep))
\end{equation}
holds for  $(X,\ep)\in{\mathcal U}$, $\ep\neq 0$ and $t\in[0, T]$ where ${\bar c}=1/32$.
\end{prop}
\begin{proof} Set $\delta=1/9$ in this proof.  First check that one can find $c\geq {\bar c}$ such that 
\begin{equation}
\label{eq:A_1:sei}
{\bar\Delta}(t,X,\ep)\geq c\,t^2(t+\rho)\quad \text{if}\quad A_1=\nu_1+a_1\geq 0.
\end{equation}
Let $D>0$ in \eqref{eq:Del:bunkai}. It was seen in the proof of Lemma \ref{lem:nu:1sa:nuj} that $-A_1/2+\sqrt{D}\leq 0$ so that $\nu_2, \nu_3=-A_1/2\pm \sqrt{D}\leq 0$. If $|\nu_1|\geq \delta\rho$ then $t-\nu_1=t+|\nu_1|\geq t+\delta\rho$ hence  $
\delta^{-1}(t-\nu_1)\geq t+\rho$ and $t-\nu_i=t+|\nu_i|\geq t$ then \eqref{eq:A_1:sei} holds with $c=\delta$. Consider the case  $D\leq 0$ so that $\nu_2, \nu_3=-A_1/2\pm i\sqrt{|D|}$. If $|\nu_1|\geq \delta\rho$ then $\delta^{-1}(t-\nu_1)\geq t+\rho$ as above and $|t-\nu_i|\geq |t+A_1/2|\geq t$ thus \eqref{eq:A_1:sei} holds with $c=\delta$. If $|\nu_2|=|\nu_3|\geq \delta\rho$ then $A_1/2+\sqrt{|D|}\geq \delta\rho$. Since
\begin{align*}
(t-\nu_2)(t-\nu_3)\geq \big(t+A_1/2+\sqrt{|D|}\,\big)^2/2\geq (t+\delta\rho)^2/2\geq \delta\,t\,(t+\rho)/2
\end{align*}
then  \eqref{eq:A_1:sei} holds with $c=\delta/2$.

Turn to the case $A_1<0$. Since $\psi=-(\nu_1+a_1)/2>0$, one can write $
{\bar\Delta}(t)=(t-\nu_1)((t-\psi)^2-D)$. 
Note that $D\leq 0$ otherwise $\psi+\sqrt{D}>0$ would be a positive simple root of ${\bar\Delta}(t)$ and a contradiction. Then $
(t-\psi)^2-D=(t-\psi)^2+|D|\geq (t-\psi)^2$. 
Consider the case $|\nu_1|\geq \delta\rho$.  Recalling $t-\nu_1=t+|\nu_1|\geq \delta(t+\rho)$ we get 
\begin{equation}
\label{eq:A_1:fu}
{\bar\Delta}(t,X,\ep)\geq c\,(t-\psi)^2(t+\rho)
\end{equation}
with $c=\delta$. When $|\nu_2|=|\nu_3|=\big|\psi\pm i\sqrt{|D|}\,\big|=\sqrt{\psi^2+|D|}\geq \delta\rho$ one has
\[
(t-\nu_2)(t-\nu_3)=(t-\psi)^2+|D|\geq \big(|t-\psi|+\sqrt{|D|}\,\big)^2/2.
\]
Assume $\psi\geq \sqrt{|D|}$ so that $\sqrt{2}\,\psi\geq \delta\rho$. For $0\leq t\leq \psi/2$  hence $t\leq |t-\psi|$ and $\psi/2\leq |t-\psi|$ one has
\begin{align*}
(1-\gamma)|t-\psi|+\gamma|t-\psi|\geq (1-\gamma)t+\gamma\psi/2
\geq \delta(2\sqrt{2}+\delta)^{-1}(t+\rho)
\end{align*}
with $\gamma=2\sqrt{2}/(2\sqrt{2}+\delta)$. Since $|t-\psi|+\sqrt{|D|}\geq |t-\psi|\geq t$ and $|t-\nu_1|=t+|\nu_1|\geq t$ it is clear that \eqref{eq:A_1:sei} holds with $c=\delta/(2\sqrt{2}+\delta)$. For $\psi/2\leq t$ such that $|t-\psi|\leq t$ one sees
\[
t-\nu_1\geq t= (1-\gamma)t+\gamma t\geq (1-\gamma)t+\gamma\psi/2\geq \delta(2\sqrt{2}+\delta)^{-1}(t+\rho)
\]
and hence $(t-\nu_1)((t-\psi)^2+|D|)\geq c\,(t+\rho)(t-\psi)^2$ which is \eqref{eq:A_1:fu} with $c=\delta/(2\sqrt{2}+\delta)$. Next assume $\sqrt{|D|}\geq \psi$ so that $\sqrt{2}\sqrt{|D|}\geq \delta\rho$. For $0\leq t\leq \psi/2$ one has $|t-\psi|\geq t$ and hence $
|t-\psi|+\sqrt{|D|}\geq t+\delta\rho/\sqrt{2}\geq (\delta/\sqrt{2})(t+\rho)$. 
Noting $|t-\nu_1|=t+|\nu_1|\geq t$ it is clear that \eqref{eq:A_1:sei} holds with $c=\delta/2\sqrt{2}$. For $\psi/2\leq t$ we see that
\[
|t-\psi|+\sqrt{|D|}\geq t-|\psi|+\sqrt{|D|}\geq t,\quad |t-\psi|+\sqrt{|D|}\geq \sqrt{|D|}\geq \delta\rho/\sqrt{2}
\]
which shows that $|t-\psi|+\sqrt{|D|}\geq \delta(\sqrt{2}+\delta)^{-1}(t+\rho)$. Recalling $|t-\nu_1|=t+|\nu_1|\geq t$ again one has \eqref{eq:A_1:sei} with $c=\delta/2(\sqrt{2}+\delta)$. Thus choosing ${\bar c}=1/32<1/(18\sqrt{2}+2)=\delta/2(\sqrt{2}+\delta)$ the proof is completed.
\end{proof}
\begin{lem}
\label{lem:pert:Dis}One can find a neighborhood ${\mathcal U}$ of $({\bar X},0)$ and  $C^*>0$ such that
\begin{equation}
\label{eq:pert:Dis}
\frac{|\dif_t{ \Delta}(t,X,\ep)|}{\Delta(t, X,\ep)}\leq C^*\Big(\frac{1}{t}+\frac{1}{|t-\psi|+\sqrt{a}\,\ep\,}\Big),\quad (X, \ep)\in {\mathcal U},\;\;\ep>0
\end{equation}
holds for  $t\in (0, T]$.
\end{lem}
\begin{proof} It will suffice to show \eqref{eq:pert:Dis} for ${\Delta}(t,X,\sqrt{2}\ep)$ which we denote by ${\tilde \Delta}(t, X,\ep)$. It is clear that
\[
{\tilde \Delta}=\Delta+4e^3\big(3(t+\rho)^2\ep^2+3(t+\rho)\ep^4+\ep^6\big)=\Delta+\Delta_r.
\]
Writing ${\tilde \Delta}={\tilde e}\big({\bar\Delta}+{\bar \Delta}_r\big)$ with $\Delta_r={\tilde e}{\bar \Delta}_r$ it suffices to show the assertion for  ${\bar \Delta}+{\bar\Delta}_r$ instead of ${\tilde \Delta}$. Note that 
\begin{equation}
\label{eq:per:a}
|\dif_t{\bar\Delta}_r|/\bar\Delta_r\leq C\big(1+1/(t+\rho)\big)\leq C'/t
\end{equation}
always holds. Write ${\bar\Delta}=\prod_{j=1}^3(t-\nu_j)$ and note that $
\dif_t{\bar\Delta}/{\bar\Delta}=\sum_{j=1}^3(t-\nu_j)^{-1}$. 
When $A_1\geq 0$ we see from the proof of Proposition \ref{pro:bound:Dis} that $|t-\nu_j|\geq t$ hence $
\big|\dif_t{\bar\Delta}/{\bar\Delta}\big|\leq 3/t$. 
Therefore one has
\begin{gather*}
|\dif_t{\tilde \Delta}|/{\tilde \Delta}\leq |\dif_t{\bar\Delta}|/\big({\bar\Delta}+{\bar\Delta}_r\big)+|\dif_t{\bar\Delta}_r|/\big({\bar\Delta}+{\bar\Delta}_r\big)\\
\leq |\dif_t{\bar\Delta}|/{\bar\Delta}+|\dif_t{\bar\Delta}_r|/{\bar\Delta}_r
\end{gather*}
which proves the assertion. Let $A_1<0$ then $
{\bar\Delta}=(t-\nu_1)((t-\psi)^2-D)$ 
where $\psi>0$ and $D\leq 0$ as seen in the proof of Proposition \ref{pro:bound:Dis}. If $|D|\geq a\ep^2$ 
\[
|t-\psi|(|t-\psi|+\sqrt{a}\,\ep)\leq \sqrt{2}((t-\psi)^2+|D|)
\]
which shows the assertion since $|t-\nu_1|=t+|\nu_1|\geq t$. Similarly if $|t-\psi|\geq \sqrt{a}\,\ep$ one has $|t-\psi|(|t-\psi|+\sqrt{a}\,\ep)\leq 2(t-\psi)^2\leq 2((t-\psi)^2+|D|)$ hence the assertion. If $|D|<a\ep^2$ and $|t-\psi|<\sqrt{a}\,\ep$ it follows that
\[
|\dif_t{\bar\Delta}|\leq (t-\psi)^2+|D|+2|t-\nu_1||t-\psi|\leq 2a\ep^2+Ca^{3/2}\ep
\]
because $|t-\nu_1|\leq Ca$. In view of $C{\bar\Delta}_r\geq a^2\ep^2$ one concludes that
\begin{gather*}
|\dif_t{\bar\Delta}|/\big({\bar\Delta}+{\bar\Delta}_r\big)\leq |\dif_t{\bar\Delta}|/{\bar\Delta}_r\leq C\,\big(2a\ep^2+Ca^{3/2}\ep\big)/(a^2\ep^2)\\
\leq C\Big(\frac{1}{a}+\frac{1}{\sqrt{a}\,\ep}\Big)\leq C'\Big(\frac{1}{t}+\frac{1}{|t-\psi|+\sqrt{a}\,\ep}\Big)
\end{gather*}
which together with \eqref{eq:per:a} proves the assertion.
\end{proof}
%

\section{Localized symbols}
\label{sec:kakutyo}

In the preceding Sections \ref{subsec:coeff:dis} and \ref{subsec:bound:dis} all symbols we have studied are defined in some conic (in $\xi$) neighborhood of $(X,\ep)=({\bar X},0)$ or $X={\bar X}$. 
In this section we define symbols on $\R^d\times\R^d$ which localizes such symbols around $(X,\ep)=({\bar X},0)$ or $X={\bar X}$ with a parameter $M$.

\subsection{Localization via localized coordinates functions}

Let ${\bar X}=(0, {\bar\xi})$ with $|{\bar\xi}|=1$. Let $\chi(s)\in C^{\infty}(\R)$ be equal to $1$ in $|s|\leq 1$, vanishes in $|s|\geq 2$ such that $0\leq \chi(s)\leq 1$. Define $y(x)$ and $\eta(\xi)$ by
\begin{align*}
y_j(x)=\chi(M^2 x_j)x_j,\quad
 \eta_j(\xi)=\chi(M^2(\xi_j\lr{\xi}^{-1}-{\bar\xi}_j))(\xi_j-{\bar\xi}_j\lr{\xi})+{\bar\xi}_j\lr{\xi}
\end{align*}
for $j=1,2,\ldots, d$ with $
 \lr{\xi}=(\gamma^2+|\xi|^2)^{1/2}$ 
where $M$ and $\gamma$ are large positive parameters constrained 
\begin{equation}
\label{eq:seigen}
\gamma\geq M^5.
\end{equation}
It is easy to see that  $(1-CM^{-2})\lr{\xi}\leq |\eta|\leq (1+CM^{-2})\lr{\xi}$ and
\begin{equation}
\label{eq:y:atai}
|y|\leq CM^{-2},\quad |\eta/|\eta|-{\bar\xi}|\leq CM^{-2}
\end{equation}
with some $C>0$ so that $(y,\eta)$ is contained in a conic neighborhood of $(0,{\bar\xi})$, shrinking with $M$. Note that $(y,\eta)=(x,\xi)$  on the conic neighborhood of $(0,{\bar\xi})$
\begin{equation}
\label{eq:conic:nbd}
W_M=\big\{(x,\xi)\mid |x|\leq M^{-2},\; |\xi_j/|\xi|-{\bar\xi}_j|\leq M^{-2}/2,\; |\xi|\geq \gamma M \big\}
\end{equation}
since $
\big|\xi_j/\lr{\xi}-{\bar\xi}_j\big|\leq \big|\xi_j/\lr{\xi}-\xi_j/|\xi|\big|+\big|\xi_j/|\xi|-{\bar\xi}_j\big|\leq M^{-2}$ 
if $(x,\xi)\in W_M$ where $\delta_{i j}$ is the Kronecker's delta. Let $f(X,\ep)$, $h(X)$  be smooth functions in a conic neighborhood of $({\bar X},0)$, ${\bar X}$ respectively  which are homogeneous of degree $0$ in $\xi$ then we define localized symbols of $f(x, \xi)$, $h(x,\xi)$ of $f(X,\ep)$, $h(X)$ by 
\begin{align*}
f(x,\xi)=f(y(x),\eta(\xi), \ep(\xi)),\quad h(x,\xi)=h(y(x),\eta(\xi))
\end{align*}
with $\ep(\xi)=M^{1/2}\lr{\xi}^{-1/2}$ or $\ep(\xi)=\sqrt{2}M^{1/2}\lr{\xi}^{-1/2}$. 
 In view of \eqref{eq:y:atai} such extended symbols are defined on $\R^d\times\R^d$, taking $M$ large if necessary. Let 
\[
G=M^4(|dx|^2+\lr{\xi}^{-2}|d\xi|^2).
\]
Then it is easy to see
\begin{equation}
\label{eq:y:eta}
y_j\in S(M^{-2},G),\;\; \eta_j-{\bar\xi}_j\lr{\xi}\in S(M^{-2}\lr{\xi},G),\;\;\ep(\xi)\in S(M^{-2}, G)
\end{equation}
for $j=1,\ldots, d$. To avoid confusions we denote $|\eta(\xi)|$ by $[\xi]$ hence
\begin{equation}
\label{eq:xi:kakudai}
[\xi]\in S(\lr{\xi}, G),\quad [\xi]\lr{\xi}^{-1}-1\in S(M^{-2}, G).
\end{equation}
\begin{lem}
\label{lem:kakutyo:1}Let $f(X,\ep)$ be a smooth function in a conic neighborhood of $({\bar X},0)$ which is homogeneous of degree $0$ in $\xi$. 
If $\dif_x^{\al}\dif_{\xi}^{\be}\dif_{\ep}^kf({\bar X},0)=0$ for $0\leq |\al+\be|+k<r$ then $f(x,\xi)=f(y(x),\eta(\xi),\ep(\xi))\in S(M^{-2r}, G)$. Let $h(X)$ be a smooth function in a conic neighborhood of ${\bar X}$ which is homogeneous of degree $0$ in $\xi$. Then $
h(x,\xi)-h(0,{\bar\xi})\in S(M^{-2},G)$. 
\end{lem}
\begin{proof}We prove the first assertion. By the Taylor formula one can write
\begin{align*}
f(y,\eta,\ep)=\sum_{|\al+\be|+k=r}\frac{1}{\al!\be!k!}y^{\al}(\eta-{\bar\xi}\lr{\xi})^{\be}\ep^k\dif_x^{\al}\dif_{\xi}^{\be}\dif_{\ep}^k f(0,{\bar\xi}\lr{\xi},0)\\
+(r+1)\sum_{|\al+\be|+k=r+1}\Big[\frac{1}{\al!\be!k!}y^{\al}(\eta-{\bar\xi}\lr{\xi})^{\be}\ep^k\\
\times \int_0^1(1-\theta)^r\dif_x^{\al}\dif_{\xi}^{\be}\dif_{\ep}^k f(\theta y,\theta(\eta-{\bar\xi}\lr{\xi})+{\bar\xi}\lr{\xi},\theta \ep)
d\theta \Big].
\end{align*}
It is clear that 
\[
y^{\al}(\eta-{\bar\xi}\lr{\xi})^{\be}\ep^k\dif_x^{\al}\dif_{\xi}^{\be}\dif_{\ep}^k f(0,{\bar\xi},0)\lr{\xi}^{-|\be|}\in S(M^{-2r},G)
\]
for $|\al+\be|+k=r$ in view of \eqref{eq:y:eta}. Since $\lr{\xi}/C\leq |\theta(\eta-{\bar\xi}\lr{\xi})+{\bar\xi}\lr{\xi}|\leq C\lr{\xi}$ the integral belongs to $S(\lr{\xi}^{-|\be|}, G)$ hence  the second term on the right-hand side is in $S(M^{-2r-2},G)$ thus the assertion.
\end{proof}
%

\subsection{Estimate of localized symbols}
\label{subsec:kakudai}

From now on it is assumed that all constants are independent of $M$ and $\gamma$. As explained before we write $A \precsim B$ if $A$ is bounded by constant, independent of parameters, $M$ and $\gamma$, times $B$. Let $\rho(x, \xi)$ be the localized symbol of  ${ \rho}(X,\ep)$ with $\ep=M^{1/2}\lr{\xi}^{-1/2}$ so that
\[
\rho(x,\xi)=\al(x,\xi)+M\lr{\xi}^{-1}.
\]
From Lemma \ref{lem:kakutyo:1} we see $\rho\in S(M^{-4},G)$ hence $|\dif_x^{\al}\dif_{\xi}^{\be}\rho|\precsim \lr{\xi}^{-|\be|}$ for $|\al+\be|=2$. Since $\rho\geq 0$ it follows from the Glaeser's inequality that
\begin{equation}
\label{eq:rho:Glae}
\big|\dif_x^{\al}\dif_{\xi}^{\be}\rho\big|\precsim \sqrt{\rho}\,\lr{\xi}^{-|\be|},\qquad |\al+\be|=1.
\end{equation}
\begin{lem}
\label{lem:a:bound:rho}Assume that $f(X,\ep)$ is smooth and homogeneous of degree $0$ in $\xi$ in a conic neighborhood of $({\bar X},0)$ and satisfies  $|f(X,\ep)|\leq C{ \rho}(X,\ep)^n$ with some $n>0$ there. For the localized symbol $f(x,\xi)$ there is $C_{\al\be}>0$ such that 
\begin{equation}
\label{eq:a:to:rho}
\big|\dif_x^{\al}\dif_{\xi}^{\be}f(x,\xi)\big| \leq C_{\al\be} \rho(x,\xi)^{n-|\alpha+\be|/2}\lr{\xi}^{-|\be|}.
\end{equation}
\end{lem}
\begin{proof}From the assumption it follows that $\dif_x^{\al}\dif_{\xi}^{\be}\dif_{\ep}^k f(0,{\bar\xi},0)=0$ for $|\al+\be|+k<2n$ hence Lemma \ref{lem:kakutyo:1} shows that $f(x,\xi)\in S(M^{-4n}, G)$. Therefore for $|\al+\be|\geq 2n$ one sees $
\big|\lr{\xi}^{|\be|}\dif_x^{\al}\dif_{\xi}^{\be} f(x,\xi)\big|\leq CM^{2|\al+\be|-4n}\leq C\big(C_0\rho^{-1}\big)^{|\al+\be|/2-n}$ 
because $M^4\leq C_0\rho^{-1}$. Hence \eqref{eq:a:to:rho} holds for $|\al+\be|\geq 2n$. Assume $|\al+\be|\leq 2n-1$. 
Writing $X=(x,\xi)$, $Y=(y,\eta\lr{\xi})$ and applying the Taylor formula to obtain
\begin{equation}
\label{eq:a_jdif:1}
\begin{split}
\big| f(X+sY)\big|=\Big|\sum_{j=0}^{2n-1}\frac{s^j}{j!}d^j f(X; Y)+\frac{s^{2n}}{(2n)!}d^{2n} f(X+s\theta Y; Y)\Big|\\
\leq C\Big(\sum_{j=0}^{2n-1}\frac{s^j}{j!}d^j\rho(X; Y)+\frac{s^{2n}}{(2n)!}d^{2n}\rho(X+s\theta' Y; Y)\Big)^n
\end{split}
\end{equation}
where $
d^j f(X;Y)=\sum_{|\al+\be|=j}(j!/\al!\be!)\dif_x^{\al}\dif_{\xi}^{\be}a(x,\xi)y^{\al}\eta^{\be}\lr{\xi}^{|\be|}$, $0<\theta, \theta'<1$. 
If $\rho(x,\xi)=0$ then $\dif_x^{\al}\dif_{\xi}^{\be}\rho(x,\xi)=0$ for $|\al+\be|=1$ because $\rho\geq 0$ and  then it follows from \eqref{eq:a_jdif:1} that $\dif_x^{\al}\dif_{\xi}^{\be} f(x,\xi)=0$ for $|\al+\be|\leq 2n-1$  hence \eqref{eq:a:to:rho}. We fix a small $s_0>0$. If $\rho(x,\xi)\geq s_0$ then one has
\begin{gather*}
\big|\dif_x^{\al}\dif_{\xi}^{\be} f(x,\xi)\lr{\xi}^{|\be|}\big|\leq C_{\al\be}
\leq C_{\al\be}s_0^{-n+|\al+\be|/2}\rho^{n-|\al+\be|/2}
\end{gather*}
for $|\al+\be|\leq 2n-1$ which proves \eqref{eq:a:to:rho}. Assume $0<\rho(x,\xi)<s_0$. Note that 
\[
\big |d^{2n} f(X+s \theta Y;Y)\big|\leq C,\quad d^{2n}\rho(X+s \theta' Y;Y)\leq C\rho(X)^{1-n}
\]
for any $|(y,\eta)|\leq 1/2$. Indeed the first one is clear from $f(x,\xi)\in S(M^{-4n}, G)$. To check the second one it is enough to note that  $\rho(x,\xi)\in S(M^{-4}, G)$ and 
\begin{equation}
\label{eq:jimei:1}
M^{-4+2j}\leq (C_0\rho^{-1})^{j/2-1},\;\;j\geq 2, \quad \sqrt{2}\lr{\xi+\theta\lr{\xi}\eta}\geq \lr{\xi}/2
\end{equation}
for $|\eta|\leq 1/2$ and $|\theta|<1$. Take $s=\rho(X)^{1/2}$ in  \eqref{eq:a_jdif:1} to get
\begin{align*}
\Big|\sum_{j=0}^{2n-1}\frac{1}{j!}d^j f(X;Y)\rho(X)^{j/2}\Big|
\leq C\Big(\sum_{j=0}^{2n-1}\frac{1}{j!}d^j\rho(X;Y)\rho(X)^{j/2}\Big)^n+C\rho(X)^n
\end{align*}
where the right-hand side is bounded by $C\rho(X)^n$ for $|d\rho(X;Y)|\leq C'\rho(X)^{1/2}$ in view of \eqref{eq:rho:Glae} and \eqref{eq:jimei:1} for $j\geq 3$. Replacing $(y,\eta)$ by $s(y,\eta)$, $|(y,\eta)|=1/2$, $0<|s|<1$ one obtains
\[
\Big|\sum_{j=1}^{2n-1}\frac{s^j}{j!}d^j f(X;Y)\frac{\rho(X)^{j/2}}{\rho(X)^{n}}\Big|\leq C_1.
\]
Since two norms $\sup_{|s|\leq 1}|p(s)|$ and $ \max{\{|c_j|\}}$ 
on the vector space consisting of all polynomials $p(s)=\sum_{j=0}^{2n-1}c_j s^j$ are equivalent one obtains $|d^j f(X;Y)|\leq B'\rho(X)^{n-j/2}$. Since $|(y,\eta)|=1/2$ is arbitrary one concludes \eqref{eq:a:to:rho}. 
\end{proof}
\begin{lem}
\label{lem:hyoka:rho:s}
Let $s\in\R$. Then $
\big|\dif_x^{\alpha}\dif_{\xi}^{\beta}\rho^s\big|\precsim \rho^{s-|\alpha+\beta|/2}\lr{\xi}^{-|\beta|}$.
\end{lem}
\begin{proof} When $s=1$ the assertion follows from Lemma \ref{lem:a:bound:rho}. 
Since 
\[
\dif_x^{\alpha}\dif_{\xi}^{\beta}\rho^s=\sum C_{\al^{(j)}\be^{(j)}}\rho^s\,\big(\dif_x^{\alpha^{(1)}}\dif_{\xi}^{\beta^{(1)}}\rho/\rho\big)\cdots\big(\dif_x^{\alpha^{(k)}}\dif_{\xi}^{\beta^{(k)}}\rho/\rho\big)
\]
the proof for the case $s\in\R$ is clear.
\end{proof}
\begin{lem}
\label{lem:a:j:extend}Let $a_j(x,\xi)$ be the localized symbol of $a_j(X,\ep)$ then 
\[
\big|\dif_x^{\al}\dif_{\xi}^{\be}a_j(x,\xi)\big|\precsim \rho(x,\xi)^{j-|\alpha+\be|/2}\lr{\xi}^{-|\be|},\quad j=1,2,3.
\]
\end{lem}
\begin{proof}The assertion follows from Lemmas \ref{lem:a_j:al} and \ref{lem:a:bound:rho}.
\end{proof}
For the localized symbol  $\psi(x, \xi)$   of $\psi(X,\ep)$ with $\ep=M^{1/2}\lr{\xi}^{-1/2}$ 
we have
\begin{lem}
\label{lem:dif:psi:1}One has $
\big|\dif_x^{\al}\dif_{\xi}^{\be}\psi(x,\xi)\big|\precsim \rho(x,\xi)^{1-|\alpha+\be|/2}\lr{\xi}^{-|\be|}$.
\end{lem}
\begin{proof}Since Lemma \ref{lem:a:bound:rho} is not available for ${\psi}(X,\ep)$ because it is not defined for $\ep=0$ then we show the assertion directly. Let $\nu_1(x,\xi)$, $a_1(x,\xi)$ and ${\bar \Delta}(t, x,\xi)$ be localized symbols of $\nu_1(X,\ep)$, $a_1(X,\ep)$ and ${\bar\Delta}(t,X,\ep)$ with $\ep=M^{1/2}\lr{\xi}^{-1/2}$ and hence ${\bar \Delta}(\nu_1(x,\xi),x, \xi)=0$. Note that $|\dif_t{\bar\Delta}(\nu_1, x, \xi)|\geq 4\,c_2^2\,\rho^2(x,\xi)$ if $\nu_1(x,\xi)+a_1(x,\xi)<2c_1\rho(x,\xi)$ thanks to Lemma \ref{lem:nu:1sa:nuj}. Starting with
\[
\dif_t{\bar \Delta}(\nu_1, x, \xi)\partial_x^{\alpha}\dif_{\xi}^{\be}\nu_1+\partial_x^{\alpha}\dif_{\xi}^{\be}{\bar\Delta}(\nu_1, x, \xi)=0,\quad |\al+\be|=1
\]
a repetition of the same argument  in Lemma \ref{lem:lam:bibun} below together with Lemma \ref{lem:a:j:extend} shows
\begin{equation}
\label{eq:hyoka:nu1}
\big|\partial_x^{\alpha}\dif_{\xi}^{\be}\nu_1\big|\precsim \rho^{1-|\al+\be|/2}\lr{\xi}^{-|\be|},\quad \nu_1+a_1<2c_1\rho.
\end{equation}
Here we have used $|\nu_1|\precsim \rho$ which also follows from Lemma \ref{lem:a:j:extend}. Using \eqref{eq:hyoka:nu1}  and   Lemmas \ref{lem:hyoka:rho:s} and \ref{lem:a:j:extend} the assertion follows easily.
\end{proof}
%

\subsection{Estimate of discriminant}

Let $\al(x,\xi)$, $a(t, x,\xi)$, $b(t, x,\xi)$, $e(t, x,\xi)$ be localized symbols  of $\al(X)$, $a(t, X)$, $b(t, X)$, $e(t, X)$  so that 
$\tau^3-e(t, x,\xi)(t+\al(x,\xi))[\xi]^2\tau+b(t, x,\xi)[\xi]^3$ 
is now defined on $\R^d\times \R^d$ and coincides with the original $p$ in a conic neighborhood $W_M$ of $(0,{\bar \xi})$. We add a term $2Me(t, x,\xi)\lr{\xi}^{-1}[\xi]^2$ to this;
\begin{equation}
\label{eq:teigi:hat:p}
{\hat p}=\tau^3-e(t+\al+2M\lr{\xi}^{-1})[\xi]^2\tau-b\,[\xi]^3
\end{equation}
where we denote
\begin{equation}
\label{eq:aM:teigi}
a_M(t, x,\xi)=e(t, x,\xi)(t+\al(x,\xi)+2M\lr{\xi}^{-1})
\end{equation}
which is the localized symbol of $a(t, X, \ep)=a(t,X)+\ep^2$ with $\ep=\sqrt{2}M^{1/2}\lr{\xi}^{-1/2}$. Consider the discriminant
\begin{equation}
\label{eq:DM:Dis}
\begin{split}
\Delta_M(t, x,\xi)=4\,e^3\big(t+\al+2M\lr{\xi}^{-1}\big)^3-27\,b^2\\
= 4\,e^3\big(t+\al+M\lr{\xi}^{-1})^3-27\,b^2
+\Delta_r(t, x,\xi)
\end{split}
\end{equation}
where, recalling $\al(x, \xi)+M\lr{\xi}^{-1}=\rho(x,\xi)$, we have
\begin{align*}
\Delta_r=4e^3\big(3(t+\rho)^2M\lr{\xi}^{-1}+3(t+\rho)M^2\lr{\xi}^{-2}+M^3\lr{\xi}^{-3}\big)\\
=12e^3\big(c_1(x,\xi)t^2+c_2(x,\xi)t+c_3(x,\xi)\big)\geq 12e^3M(t+\rho)^2\lr{\xi}^{-1}.
\end{align*}
It is clear that $c_j(x,\xi)$ verifies $\big|\dif_x^{\al}\dif_{\xi}^{\be}c_j\big|\precsim \rho^{j-|\al+\be|/2}\lr{\xi}^{-|\be|}$. Let $\Delta(t, x, \xi)$, ${\bar \Delta}(t, x, \xi)$ be localized symbols of $\Delta(t, X, \ep)$, ${\bar \Delta}(t, X,\ep)$ with $\ep=M^{1/2}\lr{\xi}^{-1/2}$. 
Thanks to Proposition \ref{pro:bound:Dis}  one has $
{\bar \Delta}(t, x,\xi)\geq {\bar c}\,\min{\{t^2,(t-\psi)^2\}}(t+\rho)$. 
Noting that $\Delta(t, x, \xi)={\tilde e}\,{\bar \Delta}$ we see
\begin{equation}
\label{eq:DM:Dis:bis}
\begin{split}
\Delta(t, x,\xi)={\tilde e}\,{\bar \Delta}\geq {\tilde e}\,{\bar c}\, \min{\{t^2,(t-\psi)^2\}}(t+\rho)\\
\geq \big({\tilde e}/e\big){\bar c}\,\min{\{t^2,(t-\psi)^2\}}e(t+\rho).
\end{split}
\end{equation}
Therefore  choosing a constant ${\bar \nu}>0$ such that $12\,e^2\geq \big({\tilde e}/e\big){\bar c}\, {\bar \nu}$ one obtains from \eqref{eq:DM:Dis}, \eqref{eq:DM:Dis:bis} that
\begin{equation}
\label{eq:lam:a:rho}
\begin{split}
\Delta_M\geq \big({\tilde e}/e\big){\bar c}\,\min{\big\{t^2,(t-\psi)^2\big\}}e(t+\rho)
+12\,e^3(t+\rho)^2M\lr{\xi}^{-1}\\
\geq \big({\tilde e}/e\big){\bar c}\,\big(\min{\big\{t^2,(t-\psi)^2\big\}}+{\bar \nu}(t+\rho)M\lr{\xi}^{-1}\big)e(t+\rho)\\
\geq \big({\tilde e}/e\big){\bar c}\,\min{\big\{t^2,(t-\psi)^2+{\bar \nu}M\rho\lr{\xi}^{-1}\big\}}e(t+\rho),\qquad t\geq 0.
\end{split}
\end{equation}
\begin{prop}
\label{pro:ext:matome}
One can write
\[
\Delta_M=e\big(t^3+a_1(x,\xi)t^2+a_2(x,\xi)t+a_3(x,\xi)\big)
\]
where $0<e\in S(1,G)$ uniformly in $t$ and $\big|\dif_x^{\al}\dif_{\xi}^{\be}a_j\big|\precsim \rho^{j-|\al+\be|/2}\lr{\xi}^{-|\be|}$. Moreover there exist ${\bar \nu}>0$ and $c>0$ such that
\begin{equation}
\label{eq:DM:to:aM}
\frac{\Delta_M}{a_M}\geq \frac{\tilde e}{2e}\,{\bar c}\,\min{\big\{t^2,(t-\psi)^2+{\bar \nu}M\rho\lr{\xi}^{-1}\big\}},\;\; \frac{\Delta_M}{a_M}\geq c\,M\lr{\xi}^{-1}\,a_M
\end{equation}
for $0\leq t\leq T$ where $\psi$ and $\rho$ satisfy $
\big|\dif_x^{\al}\dif_{\xi}^{\be}\psi\big|$, $\big|\dif_x^{\al}\dif_{\xi}^{\be}\rho\big|\precsim \rho^{1-|\al+\be|/2}\lr{\xi}^{-|\be|}$. 
\end{prop}
\begin{proof}
Choosing $\ep=\sqrt{2}M^{1/2}\lr{\xi}^{-1/2}$ and applying Lemma \ref{lem:DMal} one can write $\Delta_M$  as a third order polynomial in $t$, up to non-zero factor and can estimate the coefficients thanks to Lemmas \ref{lem:a_j:al} and \ref{lem:a:bound:rho} in terms of $\al+2M\lr{\xi}^{-1}$. Noting $
\rho(x,\xi)\leq \alpha(x,\xi)+2M\lr{\xi}^{-1}\leq 2\rho(x,\xi)$ 
we have the desired estimates for $a_j$. The assertion \eqref{eq:DM:to:aM} follows from \eqref{eq:lam:a:rho} for $a_M=e(t+\rho+M\lr{\xi}^{-1/2})\leq 2\,e(t+\rho)$. The estimates for $\psi$ and $\rho$  are nothing but Lemmas \ref{lem:hyoka:rho:s} and \ref{lem:dif:psi:1} with $\ep=M^{1/2}\lr{\xi}^{-1/2}$.
\end{proof}
\begin{remark}
\label{rem:fukusyu}\rm
Denoting ${\bar e}=e(0, 0, {\bar\xi})=\dif_ta(0, 0,{\bar\xi})$ it is clear from \eqref{eq:koyuti} that 
\[
{\bar e}\;\;\text{  is the nonzero positive real eigenvalue of} \;\; F_p(0, 0, 0, {\bar\xi})
\]
and the coefficient of the right-hand side of \eqref{eq:DM:to:aM} is ${\tilde e}\,{\bar c}/(2e)=2\,{\bar e}^2\,{\bar c}(1+O(M^{-2}))$. On the other hand, denoting the subprincipal symbol of $P$ by $P_{sub}$ and $b_3(0, 0, {\bar\xi})$ by ${\bar b}_3$,  it is easy to see that 
\begin{equation}
\label{eq:P:sub}
P_{sub}(0, 0, 0, {\bar\xi})={\bar e}/(2i)
+{\bar b}_3.
\end{equation}
\end{remark}
\begin{lem}
\label{lem:b:hi:a}With ${\bar e}=e(0, 0, {\bar\xi})$ we have 
\[
\big|\dif_t b\big|\leq (1+CM^{-2})\big(2\sqrt{2/3}\big){\bar e}\sqrt{a_M},
\quad 0\leq t\leq M^{-2}.
\]
\end{lem}
\begin{proof}
Write $b=\be_0(x,\xi)+t\be_1(x, \xi)+t^2\be_3(t, x,\xi)$. 
Setting $t=0$ in  $27b^2\leq 4a^3$ it is clear that $|\be_0|\leq (2/3\sqrt{3}){\bar e}^{3/2}(1+CM^{-2})\al^{3/2}$. We first check that
\begin{equation}
\label{eq:be:hi:al}
|\be_1|\leq (1+CM^{-2})(2/\sqrt{3}){\bar e}^{3/2}\sqrt{\al}.
\end{equation}
If $\al(x, \xi)=0$ then $\be_1(x, \xi)=0$ hence \eqref{eq:be:hi:al} is clear. If $\al(x,\xi)>0$ taking $t=3\al$ and noting $e(3\al, x, \xi)\leq (1+CM^{-2}){\bar e}+C\al$ it follows that
\begin{align*}
3\al |\be_1|\leq 2 (4^{3/2}(1+CM^{-2}){\bar e}^{3/2}/3\sqrt{3})\al^{3/2}+|\be_0|+C\al^2\\
\leq (6/\sqrt{3})(1+CM^{-2}){\bar e}^{3/2}\al^{3/2}+C\al^2
\leq (6/\sqrt{3})(1+CM^{-2}){\bar e}^{3/2}\al^{3/2}
\end{align*}
because $\al\leq CM^{-4}$ which proves \eqref{eq:be:hi:al}. Since $|\dif_t b|\leq |\be_1|+Ct$  we see that $
|\dif_t b|\leq (1+CM^{-2})(2/\sqrt{3}){\bar e}^{3/2}\sqrt{\al}+CM^{-2}\sqrt{t}$ 
thus the proof is immediate. 
\end{proof}
%

\section{B\'ezout matrix and diagonal symmetrizer}
\label{sec:Bezout}

Add $-2M\op{e(t, x,\xi)\lr{\xi}^{-1}}[D]^2D_t$ to the principal part and subtract the same one from the lower order part so that the operator is left to be invariant;
\begin{align*}
{\hat P}=D_t^3-a_M(t, x, D)[D]^2D_t-b(t, x, D)\,[D]^3+b_1(t, x, D)D_t^2\\
+\big(b_2(t, x, D)+d_M(t, x, D))[D]D_t+b_3(t, x, D)[D]^2
\end{align*}
where $b_j(t, x,\xi)\in S(1,G)$ and $
 d_M(t, x,\xi)=2M(e\lr{\xi}^{-1})\#[\xi]\in S(M, G)$. 
Here from Lemma \ref{lem:kakutyo:1} and  \eqref{eq:xi:kakudai} it follows that
\begin{equation}
\label{eq:dM:seimitu}
d_M(t, x,\xi)-2M{\bar e}\in S(M^{-1}, g).
\end{equation}
With $U={^t}(D_t^2u,[D]D_tu,[D]^2u)$ the equation ${\hat P} u=f$ is transformed to 
\begin{equation}
\label{eq:redE}
D_tU=A(t, x, D)[D]U+B(t, x, D)U+F
\end{equation}
where $F={^t}(f,0,0)$ and
\[
 A(t, x,\xi)=
\begin{bmatrix}0&a_M&b\\
1&0 &0\\
0&1&0
\end{bmatrix},\quad B(t, x,\xi)=
\begin{bmatrix}b_1&b_2+d_M&b_3\\
0&0 &0\\
0&0&0
\end{bmatrix}.
\]
Let $S$ be  the B\'ezout matrix  of ${\hat p}$ and $\dif {\hat p}/\dif \tau$, that is
\[
S(t, x,\xi)=
\begin{bmatrix}3&0&-a_M\\
0&2a_M&3b\\
-a_M&3b&a_M^2
\end{bmatrix}
\]
then $S$ is positive semidefinite and symmetrizes $S$, that is $SA$ is symmetric which is easily examined directly, though this is a special case of a general fact (see \cite{Ja2}, \cite{Ni3}).

\subsection{Eigenvalues of B\'ezout matrix}

To simplify notation denote
\[
\sigma(t, x,\xi)=t+\al(x,\xi)+2M\lr{\xi}^{-1}=t+\rho(x,\xi)+M\lr{\xi}^{-1}
\]
hence $a_M(t, x,\xi)=e(t, x,\xi)\sigma(t, x,\xi)$ and $
(1-CM^{-2}){\bar e}\,\sigma\leq a_M\leq (1+CM^{-2}){\bar e}\,\sigma$.
In what follows we assume that  $t$ varies in the interval
\[
0\leq t\leq M^{-4}.
\]
Since $\rho\in S(M^{-4},G)$ it is clear that $\sigma(t, x,\xi)\in S(M^{-4}, G)$. 
\begin{lem}
\label{lem:dif:sig}We have $
\big|\dif_x^{\al}\dif_{\xi}^{\be}\sigma\big|\precsim \sigma^{1-|\al+\be|/2}\lr{\xi}^{-|\be|}$. 
 In particular $\sigma\in S(\sigma, g)$.
\end{lem}
\begin{proof}It is clear from \eqref{eq:rho:Glae} that  $
\big|\dif_x^{\al}\dif_{\xi}^{\be}\sigma\big|\precsim \sqrt{\sigma}\,\lr{\xi}^{-|\be|}$ 
for $|\al+\be|=1$. For $|\al+\be|\geq 2$ it results $
\big|\dif_x^{\al}\dif_{\xi}^{\be}\sigma\big|\precsim M^{2|\al+\be|-4}\lr{\xi}^{-|\be|}\precsim \sigma^{1-|\al+\be|/2}\lr{\xi}^{-|\be|}$ from $\rho\in S(M^{-4}, G)$ since $C\sigma^{-1}\geq M^4$. The second assertion is clear from $\sigma^{-1}\leq M^{-1}\lr{\xi}$.
\end{proof}
\begin{cor}
\label{cor:dif:sig:b}Let $s\in \R$. Then $
\big|\dif_x^{\al}\dif_{\xi}^{\be}\sigma^s\big|\precsim \sigma^{s-|\al+\be|/2}\lr{\xi}^{-|\be|}$. 
In particular $\sigma^s\in S(\sigma^s, g)$.
\end{cor}
\begin{definition}\rm
\label{dfn:calC}
To simplify notations we denote by ${\mathcal C}(\sigma^s)$ the set of symbols $r(t, x,\xi)$ satisfying $
\big|\dif_x^{\al}\dif_{\xi}^{\be}r\big|\precsim \sigma^{s-|\al+\be|/2}\lr{\xi}^{-|\be|}$.
\end{definition}
It is clear that ${\mathcal C}(\sigma^s)\subset  S(\sigma^s, g)$ 
because $\sigma^{-|\al+\be|/2}\leq M^{-|\al+\be|/2}\lr{\xi}^{|\al+\be|/2}$. It is also clear that if $p\in {\mathcal C}(\sigma^s)$ with $s>0$ then $(1+p)^{-1}-1\in {\mathcal C}(\sigma^s)$.
\begin{lem}
\label{lem:NiP}
One has $
a_M^s\in {\mathcal C}(\sigma^s)$ for $s\in\R$, $b\in {\mathcal C}(\sigma^{3/2})$, $\dif_ta_M\in {\mathcal C}(1)$ and $\partial_tb\in {\mathcal C}(\sqrt{\sigma})$.
\end{lem}
\begin{proof}The first assertion is clear from Corollary \ref{cor:dif:sig:b} because $a_M=e\sigma$ and $e\in S(1,G)$, $1/C\leq e\leq C$. To show the second assertion, recalling that $b(t, x, \xi)$ is the localized symbol of $b(t, X)$,  write
\begin{equation}
\label{eq:b:tenkai}
\begin{split}
b(t, x,\xi)={b}(0, y(x),\eta(\xi))+\dif_t{ b}(0,y(x),\eta(\xi))t\\
+\int _0^1(1-\theta)\dif_t^2{ b}(\theta t, y(x),\eta(\xi))d\theta\cdot t^2.
\end{split}
\end{equation}
Since $\dif_x^{\al}\dif_{\xi}^{\be}{ b}(0,0,{\bar\xi})=0$ for $|\al+\be|\leq 2$ and $\dif_t{ b}(0,0,{\bar\xi})=0$ then thanks to Lemma \ref{lem:kakutyo:1} one has ${ b}(0,y(x),\eta(\xi))\in S(M^{-6},G)$ and $\dif_t{ b}(0,y(x),\eta(\xi))\in S(M^{-2},G)$. Since $0\leq t\leq M^{-4}$ we conclude that $b(t, x,\xi)\in S(M^{-6},G)$. Since $|b|\leq C\sigma^{3/2}$ and $\sigma\in S(M^{-4},G)$ a repetition of  the same arguments proving Lemma \ref{lem:a:bound:rho} shows the second assertion. The third assertion is clear because $\dif_t a_M=e+(\dif_t e)\sigma$. As for the last assertion, recall Lemma \ref{lem:b:hi:a} that $|\dif_t b|\leq Ca_M^{1/2}\leq C' \sigma^{1/2}$. Noting $\dif_tb\in S(M^{-2},G)$ which results from \eqref{eq:b:tenkai} one sees $|\lr{\xi}^{|\be|}\dif_x^{\al}\dif_{\xi}^{\be}\dif_t b|\precsim M^{2|\al+\be|-2}\precsim \sigma^{1/2-|\al+\be|/2}$ for $|\al+\be|\geq 1$ hence the assertion.
\end{proof}
Recall \cite[Proposition 2.1]{Ni4} 
\begin{prop}
\label{pro:Skon}Let $
0\leq \lambda_1(t, x,\xi)\leq \lambda_2(t, x,\xi)\leq \lambda_3(t, x,\xi)$ 
be the eigenvalues of $S(t, x,\xi)$. There exist $M_0$ and $K>0$ such that one has for $M\geq M_0$
\begin{gather*}
\Delta_M/(6a_M+2a_M^2+2a_M^3)\leq \lambda_1\leq \big(2/3+Ka_M\big)\,a_M^2,\\
(2-Ka_M)\,a_M\leq \lambda_2\leq (2+Ka_M)\,a_M,\\
3\leq \lambda_3\leq 3+Ka_M^2.
\end{gather*}
\end{prop}
\begin{proof}
Since $a_M=e\,\sigma$ and $\sigma\in S(M^{-4},G)$ then for any ${\bar\ep}>0$ there is $M_0$ such that $e\,M_0^{-4}\leq {\bar \varep}$. Then the assertion follows from \cite[Proposition 2.1]{Ni4}.
\end{proof}
\begin{cor}
\label{cor:konname}The eigenvalues $\lambda_i(t, x,\xi)$ are smooth in $(0,M^{-4}]\times \R^d\times \R^d$.
\end{cor}
 \subsection{Estimates  of eigenvalues}
 
First we prove
\begin{lem}
\label{lem:lam:bibun} One has $\lambda_j\in {\mathcal C}( \sigma^{3-j})$ for $ j=1, 2, 3$.
\end{lem}
 Denote $q(\lambda)={\rm det}\,(\lambda I-S)$ so that
\begin{equation}
\label{eq:qkata}
q(\lambda)=\lambda^3-(3+2a_M+a_M^2)\lambda^2+(6a_M+2a_M^2+2a_M^3-9b^2)\lambda-\Delta_M.
\end{equation}
Note  that $\dif_{\lambda}q(\lambda_i)\partial_x^{\alpha}\dif_{\xi}^{\be}\lambda_i+\partial_x^{\alpha}\dif_{\xi}^{\be}q(\lambda_i)=0$ for $|\alpha+\be|=1$. Let us write $\partial_x^{\alpha}\dif_{\xi}^{\be}=\dif_{x,\xi}^{\al,\be}$ for simplicity. We show by induction on $|\alpha+\be|$ that
\begin{equation}
\label{eq:kinoho}
\begin{split}
\dif_{\lambda}q(\lambda_i)\partial^{\alpha,\be}_{x,\xi}\lambda_i=\sum_{2|\mu+\nu|+s\geq 2} C_{\mu,\nu, \gamma^{(j)},\delta^{(j)},s}\partial_{x,\xi}^{\mu,\nu}\partial_{\lambda}^s q(\lambda_i)\\
\times \big(\partial_{x,\xi}^{\gamma^{(1)},\delta^{(1)}}\lambda_i\big)\cdots \big(\partial_{x,\xi}^{\gamma^{(s)},\delta^{(s)}}\lambda_i\big)
\end{split}
\end{equation}
where $\mu+\sum \gamma^{(i)}=\alpha$, $\nu+\sum\delta^{(j)}=\be$ and $|\gamma^{(i)}+\delta^{(j)}|\geq 1$. The assertion $|\alpha+\be|=1$ is clear. Suppose that \eqref{eq:kinoho} holds for $|\alpha+\be|=m$. With $|e+f|=1$ after operating $\partial_{x,\xi}^{e, f}$ to \eqref{eq:kinoho} the resulting left-hand side is
\begin{gather*}
\dif_{\lambda}q(\lambda_i)\partial_{x,\xi}^{\alpha+e,\be+f}\lambda_i\\
-\sum_{2|\mu+\nu|+s\geq 2} C_{\mu,\nu, \gamma^{(j)}, \delta^{(j)}, s}\partial_{x,\xi}^{\mu,\nu}\partial_{\lambda}^sq(\lambda_i)\big(\partial_{x,\xi}^{\gamma^{(1)},\delta^{(1)}}\lambda_i\big)\cdots \big(\partial_{x,\xi}^{\gamma^{(s)},\delta^{(s)}}\lambda_i\big)
\end{gather*}
while the resulting right-hand side is
\begin{align*}
\sum C_{\ldots}\dif_{x,\xi}^{\mu+e,\nu+f}\dif_{\lambda}^sq(\lambda_i)\big(\partial_{x,\xi}^{\gamma^{(1)},\delta^{(1)}}\lambda_i\big)\cdots \big(\partial_{x,\xi}^{\gamma^{(s)},\delta^{(s)}}\lambda_i\big)\\
+\sum C_{\ldots}\dif_{x,\xi}^{\mu,\nu}\dif_{\lambda}^{s+1}q(\lambda_i)\big(\dif_{x,\xi}^{e, f}\lambda_i\big)\big(\partial_{x,\xi}^{\gamma^{(1)},\delta^{(1)}}\lambda_i\big)\cdots \big(\partial_{x,\xi}^{\gamma^{(s)},\delta^{(s)}}\lambda_i\big)\\
+\sum_{j=1}^s\sum C_{\ldots}\dif_{x,\xi}^{\mu,\nu}\dif_{\lambda}^{s}q(\lambda_i)\big(\dif_{x,\xi}^{\gamma^{(1)},\delta^{(1)}}\lambda_i\big)\cdots\big(\dif_{x,\xi}^{\gamma^{(j)}+e,\delta^{(j)}+f}\lambda_i\big)\cdots  \big(\partial_{x,\xi}^{\gamma^{(s)},\delta^{(s)}}\lambda_i\big)
\end{align*}
which can be written as
\[
\sum_{2|\mu+\nu|+s\geq 2} C_{\mu, \nu, \gamma^{(j)}, \delta^{(j)},s}\partial_{x,\xi}^{\mu,\nu}\partial_{\lambda}^s q(\lambda_i)\big(\partial_{x,\xi}^{\gamma^{(1)},\delta^{(1)}}\lambda_i\big)\cdots \big(\partial_{x,\xi}^{\gamma^{(s)},\delta^{(s)}}\lambda_i\big)
\]
where $ \mu+\sum \gamma^{(i)}=\alpha+e$, $\nu+\sum \delta^{(i)}=\be+f$ and $|\gamma^{(j)}+\delta^{(j)}|\geq 1$. Therefore we conclude \eqref{eq:kinoho}. 
In order to estimate $\dif_{x,\xi}^{\alpha,\be}\lambda_i$ one needs to estimate $\dif_{x,\xi}^{\mu,\nu}\dif_{\lambda}^s q(\lambda_i)$.
\begin{lem}
\label{lem:qram}
For any $s\in\N$ and $\al, \be$ it holds that
\begin{gather*}
|\dif_{x,\xi}^{\alpha,\be}\dif_{\lambda}^s q(\lambda_j)|\precsim \sigma^{4-j-(3-j)s-|\alpha+\be|/2}\lr{\xi}^{-|\be|},\;\;j=1, 2,\\
|\dif_{x,\xi}^{\alpha,\be}\dif_{\lambda}^s q(\lambda_3)|\precsim \sigma^{-|\al+\be|/2}\lr{\xi}^{-|\be|}.
\end{gather*}
\end{lem}
\begin{proof}
From Proposition \ref{pro:Skon} and \eqref{eq:qkata} one sees that
\begin{gather*}
|q(\lambda_i)|\precsim |\lambda_i|^2+|a_M||\lambda_i|+|a_M|^3,\\
|\dif_{x,\xi}^{\alpha,\be}q(\lambda_i)|\precsim \big(|\dif_{x,\xi}^{\alpha,\beta}a_M|+|\dif_{x,\xi}^{\alpha,\be}b^2|\big)|\lambda_i|+|\dif_{x,\xi}^{\alpha,\be}a_M^3|+|\dif_{x,\xi}^{\alpha,\be}b^2|,\quad |\alpha+\be|\geq 1
\end{gather*}
because $|\Delta_M|\precsim a_M^3$ and $|b|\precsim a_M^{3/2}$. Therefore thanks to Proposition \ref{pro:Skon} and Lemma \ref{lem:NiP} one obtains the assertions for the case $s=0$. Since
\begin{gather*}
|\dif_{\lambda}q(\lambda_i)|\precsim |\lambda_i|+|a_M|,\qquad |\dif_{\lambda}^s q(\lambda_i)|\precsim 1,\;\;s\geq 2,\\
|\dif_{x,\xi}^{\alpha,\be}\dif_{\lambda}q(\lambda_i)|\precsim |\dif_{x,\xi}^{\alpha,\be}a_M||\lambda_i|+|\dif_{x,\xi}^{\alpha,\be}a_M|+|\dif_{x,\xi}^{\alpha,\be}b^2|,\quad |\alpha+\be|\geq 1,\\|\dif_{x,\xi}^{\alpha,\be}\dif_{\lambda}^2q(\lambda_i)|\precsim |\dif_{x,\xi}^{\alpha,\be}a_M|,\quad \dif_{x,\xi}^{\alpha,\be}\dif_{\lambda}^sq(\lambda_i)=0,\;\; s\geq 3,\quad |\alpha+\be|\geq 1 
\end{gather*}
the assertions for the case $s\geq 1$ are clear by Proposition \ref{pro:Skon} and Lemma \ref{lem:NiP}. 
\end{proof}
\noindent
Proof of Lemma \ref{lem:lam:bibun}: Since $\dif_{\lambda}q(\lambda_i)=\prod_{k\neq i}(\lambda_i-\lambda_k)$ it follows from Proposition \ref{pro:Skon} that
\begin{equation}
\label{eq:dmod}
6a_M(1-Ca_M)\leq \big|\dif_{\lambda}q(\lambda_i)\big|\leq  6a_M(1+Ca_M),\;i=1,2,\;\; \dif_{\lambda}q(\lambda_3)\simeq 1.
\end{equation}
Then for $|\alpha+\be|=1$ one has
\[
\big|\dif_{x,\xi}^{\alpha,\be}\lambda_j\big|\precsim \big|\dif_{x,\xi}^{\alpha,\be}q(\lambda_j)/\dif_{\lambda}q(\lambda_j)\big|\precsim \sigma^{3-j-1/2}\lr{\xi}^{-|\be|},\;\;j=1,2,3
\]
by Lemma \ref{lem:qram} with $s=0$.  Assume that $
\big|\dif_{x,\xi}^{\al,\be}
\lambda_j\big|\precsim \sigma^{3-j-|\al+\be|/2}\lr{\xi}^{-|\be|}$, $j=1,2,3$ 
holds for $|\alpha+\be|\leq m$.   Lemma \ref{lem:qram} and \eqref{eq:kinoho} show that
\begin{align*}
\big|\dif_{\lambda}q(\lambda_1)\dif_{x,\xi}^{\alpha,\be}\lambda_1\big|\precsim \sum \sigma^{3-2s-|\mu+\nu|/2}\sigma^{2-|\gamma^{(1)}+\delta^{(1)}|/2}\cdots \sigma^{2-|\gamma^{(s)}+\delta^{(s)}|/2}\lr{\xi}^{-|\be|}\\
\precsim \sum \sigma^{3-|\mu+\nu|/2}\sigma^{-|\gamma^{(1)}+\delta^{(1)}|/2}\cdots \sigma^{-|\gamma^{(s)}+\delta^{(s)}|/2}\lr{\xi}^{-|\be|}\precsim \sigma^{3-|\alpha+\be|/2}\lr{\xi}^{-|\be|}.
\end{align*}
This together with \eqref{eq:dmod} proves the estimate for $\lambda_1$.  The same arguments show the assertion for $\lambda_2$. The estimate for $\lambda_3$ is clear from \eqref{eq:kinoho} because of \eqref{eq:dmod}. Thus we have the assertion for $|\alpha+\beta|=m+1$ and the proof is completed  by induction on $|\al+\be|$.
\qed

\smallskip

\begin{lem}
\label{lem:bibun:t} One has $\dif_t\lambda_1\in {\mathcal C}(\sigma)$, $\dif_t\lambda_2\in {\mathcal C}(1)$ and $\dif_t\lambda_3\in {\mathcal C}(1)$.
\end{lem}
\begin{proof}
First examine that $\dif_{\lambda}q(\lambda_i)\partial_{x,\xi}^{\alpha,\be}\dif_t\lambda_i$ can be written as
\begin{equation}
\label{eq:kinoho:t}
\begin{split}
\sum_{|\al'+\be'|<|\al+\be|} C_{\ldots}\partial_{x,\xi}^{\mu,\nu}\partial_{\lambda}^{s+1}q(\lambda_i)\big(\dif_{x,\xi}^{\al',\be'}\dif_t\lambda_i\big)\big(\partial_{x,\xi}^{\gamma^{(1)}+\delta^{(1)}}\lambda_i\big)\cdots \big(\partial_{x,\xi}^{\gamma^{(s)}+\delta^{(s)}}\lambda_i\big)\\
+\sum C_{\ldots}\partial_{x,\xi}^{\mu,\nu}\partial_{\lambda}^{s}\dif_t q(\lambda_i)\big(\partial_{x,\xi}^{\gamma^{(1)}+\delta^{(1)}}\lambda_i\big)\cdots \big(\partial_{x,\xi}^{\gamma^{(s)}+\delta^{(s)}}\lambda_i\big)
\end{split}
\end{equation}
where $\al'+\mu+\sum\gamma^{(i)}=\al$, $\be'+\nu+\sum\delta^{(i)}=\be$ and $ |\gamma^{(i)}+\delta^{(i)}|\geq 1$. Indeed  \eqref{eq:kinoho:t} is clear when $|\alpha+\be|=0$ from $
\dif_{\lambda}q(\lambda_i)\dif_t \lambda_i+\dif_t q(\lambda_i)=0$. 
Differentiating this by $\dif_{x,\xi}^{e, f}$ and repeating the same arguments proving \eqref{eq:kinoho} one obtains \eqref{eq:kinoho:t} by induction. To prove Lemma \ref{lem:bibun:t} first check that
\begin{equation}
\label{eq:difq}
|\dif_{x,\xi}^{\alpha,\be}\dif_{\lambda}^s\dif_t q(\lambda_j)|\precsim \sigma^{3-j-(3-j)s-|\alpha+\be|/2}\lr{\xi}^{-|\be|},\;\;j=1, 2, 3.
\end{equation}
In fact from 
\begin{equation}
\label{eq:difq:tt}
\dif_t q(\lambda)=-\dif_t (2a_M+a_M^2)\lambda^2+\dif_t (6a_M+2a_M^2+2a_M^3-9b^2)\lambda-\dif_t \Delta_M
\end{equation}
it follows that $|\dif_tq(\lambda_i)|\precsim \lambda_i+\sigma^2$ and $|\dif_{x,\xi}^{\alpha,\be}\dif_tq(\lambda_i)|\precsim (\lambda_i+\sigma^2) \sigma^{-|\al+\be|/2}\lr{\xi}^{-|\be|}$ for $|\alpha+\be|\geq 1$ in view of Lemma \ref{lem:NiP} and hence the assertion for $s=0$. Since $|\dif_{x,\xi}^{\alpha,\be}\dif_{\lambda}^s\dif_t q(\lambda_i)|\precsim \sigma^{-|\alpha+\be|/2}\lr{\xi}^{-|\be|}$ for $s\geq 1$ the assertion  can be proved. We now show Lemma \ref{lem:bibun:t} for $\lambda_1$ by induction on $|\alpha+\be|$. Assume 
\begin{equation}
\label{eq:bibun:b:ind}
\big|\dif_{x,\xi}^{\al,\be}\dif_t \lambda_1\big|\precsim \sigma^{1-|\al+\be|/2}\lr{\xi}^{-|\be|}.
\end{equation}
It is clear from  \eqref{eq:dmod} and \eqref{eq:difq} that \eqref{eq:bibun:b:ind} holds for $|\alpha+\be|=0$.  Assume that \eqref{eq:bibun:b:ind} holds for $|\alpha+\be|\leq m$. For $|\alpha+\be|=m+1$, thanks to the inductive assumption, Lemma \ref{lem:qram} and Lemma \ref{lem:lam:bibun} it follows that
\begin{gather*}
\sum_{ |\al'+\be'|<|\alpha+\be|}\big|\partial_{x,\xi}^{\mu,\nu}\partial_{\lambda}^{s+1}q(\lambda_1)\big(\dif_{x,\xi}^{\al',\be'}\dif_t\lambda_1\big)\big(\partial_{x,\xi}^{\gamma^{(1)}+\delta^{(1)}}\lambda_1\big)\cdots \big(\partial_{x,\xi}^{\gamma^{(s)}+\delta^{(s)}}\lambda_1\big)\big|\\
\precsim \sum \sigma^{3-2(s+1)-|\mu+\nu|/2}\sigma^{1-|\al'+\be'|/2}\sigma^{2-|\gamma^{(1)}+\delta^{(1)}|/2}\cdots \sigma^{2-|\gamma^{(s)}+\delta^{(s)}|/2}\lr{\xi}^{-|\be|}
\end{gather*}
which is bounded by $\sigma^{2-|\alpha+\be|/2}\lr{\xi}^{-|\be|}$. On the other hand one sees
\begin{gather*}
\sum \big|\partial_{x,\xi}^{\mu,\nu}\partial_{\lambda}^{s}\dif_t q(\lambda_1)\big(\partial_{x,\xi}^{\gamma^{(1)}+\delta^{(1)}}\lambda_1\big)\cdots \big(\partial_{x,\xi}^{\gamma^{(s)}+\delta^{(s)}}\lambda_1\big)\big|\\
\preceq 
\sum\sigma^{2-2s-|\mu+\nu|/2}\sigma^{2-|\gamma^{(1)}+\delta^{(1)}|/2}\cdots \sigma^{2-|\gamma^{(s)}+\delta^{(s)}|/2}\lr{\xi}^{-|\be|}
\precsim \sigma^{2-|\alpha+\be|/2}\lr{\xi}^{-|\be|}
\end{gather*}
in view of \eqref{eq:difq} and Lemma \ref{lem:lam:bibun}. This proves that  \eqref{eq:bibun:b:ind} holds for $|\alpha+\be|=m+1$ and hence for all $\al$, $\be$. As for $\lambda_2$, $\lambda_3$  the proof is similar. 
\end{proof}
%

\subsection{Eigenvectors of  B\'ezout matrix}
\label{sec:defT}

We sometimes denote by $\co(\sigma^s)$ a function belonging to $\co(\sigma^s)$. 
If we write $n_{i j}$ for the $(i, j)$-cofactor of $\lambda_kI-S$ then $^t(n_{j1},n_{j2},n_{j3})$ is, if non-trivial, an eigenvector of $S$ corresponding to $\lambda_k$. We take $k=1$, $j=3$ and hence
\[
\begin{bmatrix}a_M(2\,a_M-\lambda_1)\\
3\,b(\lambda_1-3)\\
(\lambda_1-3)(\lambda_1-2\,a_M)
\end{bmatrix}=\begin{bmatrix}\ell_{11}\\
\ell_{21}\\
\ell_{31}
\end{bmatrix}
\]
is an eigenvector of $S$ corresponding to $\lambda_1$ and therefore
\[
{\bf t}_1=\begin{bmatrix}t_{11}\\
t_{21}\\
t_{31}
\end{bmatrix}=\frac{1}{d_1}\begin{bmatrix}\ell_{11}\\
\ell_{21}\\
\ell_{31}
\end{bmatrix},\quad d_1=\sqrt{\ell_{11}^2+\ell_{21}^2+\ell_{31}^2}
\]
is a unit eigenvector of $S$ corresponding to $\lambda_1$. Thanks to Proposition \ref{pro:Skon} and recalling $b\in \co(\sigma^{3/2})$ it is clear that $
d_1=\sqrt{36\,a_M^2+\co(\sigma^3)}=6a_M(1+\co(\sigma))$.
Therefore since $\ell_{11}=\co(\sigma^2)$, $\ell_{21}=\co(\sigma^{3/2})$ and $\ell_{31}=6\,a+\co(\sigma^2)$ we have
\[
{\bf t}_1=\begin{bmatrix}t_{11}\\
t_{21}\\
t_{31}
\end{bmatrix}=\begin{bmatrix}a_M/3+{\mathcal C}(\sigma^2)\\
-3b/(2a_M)+\co(\sigma)\\
1+\co(\sigma)
\end{bmatrix}.
\]
Similarly choosing $k=2, j=2$ and $k=3, j=1$ 
\[
\begin{bmatrix}-3a_Mb\\
(\lambda_2-3)(\lambda_2-a_M^2)-a_M^2\\
3b(\lambda_2-3)
\end{bmatrix}=\begin{bmatrix}\ell_{12}\\
\ell_{22}\\
\ell_{32}
\end{bmatrix},\;\; \begin{bmatrix}(\lambda_3-2a_M)(\lambda_3-a_M^2)-9b^2\\
-3a_Mb\\
-a_M(\lambda_3-2a_M)
\end{bmatrix}=\begin{bmatrix}\ell_{13}\\
\ell_{23}\\
\ell_{33}
\end{bmatrix}
\]
are eigenvectors of $S$ corresponding to $\lambda_2$ and $\lambda_3$ respectively and
\[
{\bf t}_j=\begin{bmatrix}t_{1j}\\
t_{2j}\\
t_{3j}
\end{bmatrix}=\frac{1}{d_j}\begin{bmatrix}\ell_{1j}\\
\ell_{2j}\\
\ell_{3j}
\end{bmatrix},\quad d_j=\sqrt{\ell_{1j}^2+\ell_{2j}^2+\ell_{3j}^2}
\]
are unit eigenvectors of $S$ corresponding to $\lambda_j$, $j=2, 3$. Thanks to Proposition \ref{pro:Skon} it is easy to see $
d_2=3\lambda_2(1+\co(\sigma))$ and $d_3=\lambda_3^2(1+\co(\sigma))$. 
Then repeating the same arguments one concludes
\[
\begin{bmatrix}t_{12}\\
t_{22}\\
t_{32}
\end{bmatrix}=\begin{bmatrix}\co(\sigma^{3/2})\\
-1+\co(\sigma)\\
-3b/\lambda_2+{\mathcal C}(\sigma)
\end{bmatrix},\qquad\begin{bmatrix}t_{13}\\
t_{23}\\
t_{33}
\end{bmatrix}=\begin{bmatrix}1+\co(\sigma)\\
\co(\sigma^{5/2})\\
-a_M/\lambda_3+\co(\sigma^2)
\end{bmatrix}.
\]
Now $T=({\bf t}_1,{\bf t}_2, {\bf t}_3)=(t_{i j})$ is an orthogonal matrix which diagonalizes $S$;
\[
\varLambda=T^{-1}ST={^t}TST=\begin{bmatrix}\lambda_1&0&0\\
0&\lambda_2&0\\
0&0&\lambda_3
\end{bmatrix}.
\]
Note that $\varLambda A^T$ is symmetric. Summarize what we have proved in
\begin{lem}
\label{lem:T:seimitu}
Let $T$ be defined as above. Then there is $M_0$ such that $T$ has the form
\begin{align*}
T=\begin{bmatrix}a_M/3+{\mathcal C}(\sigma^2)&\co(\sigma^{3/2})&1+\co(\sigma)\\
-3b/(2a_M)+\co(\sigma)&-1+\co(\sigma)&\co(\sigma^{5/2})\\
1+\co(\sigma)&-3b/\lambda_2+{\mathcal C}(\sigma)&-a_M/\lambda_3+\co(\sigma^2)\\
\end{bmatrix}\\
=\begin{bmatrix}{\mathcal C}(\sigma)&\co(\sigma^{3/2})&1+\co(\sigma)\\
\co(\sigma^{1/2})&-1+\co(\sigma)&\co(\sigma^{5/2})\\
1+\co(\sigma)&{\mathcal C}(\sigma^{1/2})&\co(\sigma)\\
\end{bmatrix},\quad M\geq M_0.
\end{align*}
In particular $T, \; T^{-1}\in S(1,g)$.
\end{lem}
\begin{lem}
\label{lem:diftT:seimitu}
We have
\begin{align*}
\dif_tT=\begin{bmatrix}\dif_t(a_M/3)+\co(\sigma)&\co(\sigma^{1/2})&\co(1)\\
-\dif_t(3b/2a_M)+{\mathcal C}(1)&\co(1)&\co(\sigma^{3/2})\\
\co(1)&-\dif_t(3b/\lambda_2)+{\mathcal C}(1)&-\dif_t(a_M/\lambda_3)+\co(\sigma)\\
\end{bmatrix}\\
=\begin{bmatrix}\co(1)&\co(\sigma^{1/2})&\co(1)\\
{\mathcal C}(\sigma^{-1/2})&\co(1)&\co(\sigma^{3/2})\\
\co(1)&{\mathcal C}(\sigma^{-1/2})&\co(1)\\
\end{bmatrix},\quad M\geq M_0.
\end{align*}
\end{lem}
\begin{proof} Note that every entry of $T$ is a function in $a_M$, $b$ and $\lambda_j$.  Then the assertion is clear from Lemmas \ref{lem:NiP} and \ref{lem:bibun:t}.
\end{proof}
From Lemma \ref{lem:T:seimitu} it follows that 
\begin{equation}
\label{eq:difYT:seimitu}
\lr{\xi}^{|\be|}\dif_x^{\alpha}\dif_{\xi}^{\be}T=\begin{bmatrix}\co(\sqrt{\sigma})&\co(\sigma)&\co(\sqrt{\sigma})\\
\co(1)&\co(\sqrt{\sigma})&\co(\sigma^{2})\\
\co(\sqrt{\sigma})&\co(1)&\co(\sqrt{\sigma})\\
\end{bmatrix},\quad |\alpha+\be|=1.
\end{equation}
\begin{lem}
\label{lem:nami:a}There is $M_0$ such that  $A^T=T^{-1}AT$ has the form
\[
A^T=\begin{bmatrix}
\co(\sqrt{\sigma})&-1+\co(\sigma)&\co(\sqrt{\sigma})\\
\lambda_1\,{\mathcal C}(\sigma^{-1})&\co(\sqrt{\sigma})&-1+\co(\sigma)\\
\lambda_1\,{\mathcal C}(\sqrt{\sigma})&\lambda_2\,{\mathcal C}(1)&\co(\sigma^{5/2})\\
\end{bmatrix},\quad M\geq M_0.
\]
\end{lem}
\begin{proof}Writing $A^T=({\tilde a}_{i j})$ it is clear that
\[
{\tilde a}_{ij}=t_{1i}\,a_M\,t_{2j}+t_{1i}\,b\,t_{3j}+t_{2i}t_{1j}+t_{3i}t_{2j}
\]
from which the assertions for ${\tilde a}_{ij}$, $j\geq i$ follows easily. 
Therefore one sees
\[
\varLambda A^T=\begin{bmatrix}\lambda_1\co(\sqrt{\sigma})&\lambda_1(-1+\co(\sigma))&\lambda_1\co(\sqrt{\sigma})\\
\lambda_2{\tilde a}_{21}&\lambda_2{\tilde a}_{22}&\lambda_2(-1+\co(\sigma))\\
\lambda_3{\tilde a}_{31}&\lambda_3{\tilde a}_{32}&\lambda_3{\tilde a}_{33}\\
\end{bmatrix}.
\]
Since $\varLambda A^T$ is symmetric it follows immediately 
\[
{\tilde a}_{31}=\lambda_1{\mathcal C}(\sqrt{\sigma})/\lambda_3,\; {\tilde a}_{32}=\lambda_2(-1+{\mathcal C}(\sigma))/\lambda_3,\; {\tilde a}_{21}=\lambda_1(-1+{\mathcal C}(\sqrt{\sigma}))/\lambda_2
\]
which proves the assertion because $1/\lambda_3\in {\mathcal C}(1)$ and $1/\lambda_2\in {\mathcal C}(\sigma^{-1})$.
\end{proof}
\begin{cor}
\label{cor:TAT:seimitu} There is $M_0$ such that  $A^T=T^{-1}AT$ has the form
\[
A^T=\begin{bmatrix}
\co(\sqrt{\sigma})&-1+\co(\sigma)&\co(\sqrt{\sigma})\\
\co(\sigma)&\co(\sqrt{\sigma})&-1+\co(\sigma)\\
\co(\sigma^{5/2})&\co(\sigma)&\co(\sigma^{5/2})\\
\end{bmatrix},\quad M\geq M_0.
\]
\end{cor}
\begin{cor}
\label{cor:dif:nami:b}We have
\[
\lr{\xi}^{|\be|}\dif_x^{\alpha}\dif_{\xi}^{\be}A^T=\begin{bmatrix}
\co(1)&\co(\sqrt{\sigma})&\co(1)\\
\co(\sqrt{\sigma})&\co(1)&\co(\sqrt{\sigma})\\
\co(\sigma^2)&\co(\sqrt{\sigma})&\co(\sigma^{2})\\
\end{bmatrix},\quad |\alpha+\be|=1.
\]
\end{cor}
\begin{proof}The proof is clear since $\lr{\xi}^{|\be|}\dif_x^{\alpha}\dif_{\xi}^{\be}(-1+\co(\sigma))=\co(\sqrt{\sigma})$.
\end{proof}
Finally consider $T^{-1}(\dif_tT)$. Note that $\langle{\dif_t {\bf t}_i,{\bf t}_j}\rangle+\langle{{\bf t}_i, \dif_t{\bf t}_j}\rangle=0$ so that $(\dif_t T^{-1})T$ is antisymmetric. From Lemmas \ref{lem:T:seimitu} and \ref{lem:diftT:seimitu} one has
\begin{equation}
\label{eq:dT:-1:T}
T^{-1}(\dif_tT)=\begin{bmatrix}0&-\dif_t(3b/2a_M)+{\mathcal C}(1)&\co(1)\\
\dif_t(3b/2a_M)+{\mathcal C}(1)&0&\co(\sqrt{\sigma})\\
\dif_t(a_M/3)+{\mathcal C}(\sigma)&\co(\sqrt{\sigma})&0\\
\end{bmatrix}.
\end{equation}
For later use we estimate $(2,1)$-th and $(3,1)$-th entries of $T^{-1}(\dif_t T)$. Recalling $a_M=e(t+\al+2M\lr{\xi}^{-1})$ and $0\leq t\leq M^{-4}$ it is clear $
\dif_t a_M-{\bar e}\in S(M^{-2}, g)$. 
Taking $|b^2/a_M^3|\leq 4/27$ into account, thanks to Lemma \ref{lem:b:hi:a} it follows that
\begin{equation}
\label{eq:c:bar:1}
\begin{split}
\big|\sqrt{a_M}\dif_t(3b/2a_M)\big|\leq 3\big(\big|\dif_t b/\sqrt{a_M}\big|+\big|b/a_M^{3/2}\big||\dif_t a_M|\big)/2\\
\leq (1+CM^{-2})\big((1+3\sqrt{2})/\sqrt{3}\big)\,{\bar e}.
\end{split}
\end{equation}
%
\section{Metric $g$ and estimates of $\omega$ and $\phi$}
\label{sec:kyori}

Introduce the metric
\[
g=g_{(x,\xi)}=M^{-1}\big(\lr{\xi}|dx|^2+\lr{\xi}^{-1}|d\xi|^2\big)
\]
which is a basic metric with which we work in this paper. Note that $
S(M^s, G)\subset S(M^s, g)$ 
because $M^{s+2|\al+\be|}\lr{\xi}^{-|\be|}\leq M^sM^{-|\al+\be|/2}\lr{\xi}^{(|\al|-|\be|)/2}$ 
in view of $\lr{\xi}\geq \gamma\geq M^5$. The metric $g$ is temperate (see \cite[Chapter 18]{Hobook}) {\it uniformly} in $\gamma\geq M^5\geq 1$ which will be checked in Section \ref{sec:metric}. 
\begin{lem}
\label{lem:hyoka:psi:b} One has
\[
\dif_x^{\alpha}\dif_{\xi}^{\beta}\psi\in S( M^{-(|\alpha+\beta|-1)/2}\rho^{1/2}\lr{\xi}^{-1/2}\lr{\xi}^{(|\alpha|-|\beta|)/2}, g),\quad |\alpha+\beta|\geq 1.
\]
\end{lem}
\begin{proof}It is enough to remark
\begin{gather*}
\big|\dif_x^{\alpha}\dif_{\xi}^{\beta}\psi\big|\precsim \rho^{1/2}\rho^{-(|\alpha+\beta|-1)/2}\lr{\xi}^{-|\beta|}
\precsim \rho^{1/2}(M^{-1}\lr{\xi})^{(|\alpha+\beta|-1)/2}\lr{\xi}^{-|\beta|}
\end{gather*}
for $|\al+\be|\geq 1$.
\end{proof}
%

\subsection{Estimate $\omega$ by metric $g$}

Taking Proposition \ref{pro:ext:matome} into account we introduce a preliminary weight
\[
\omega(t, x,\xi)=\sqrt{(t-\psi(x,\xi))^2+{\bar \nu}M\rho\lr{\xi}^{-1}}.
\]
Since the exact value of ${\bar\nu}>0$ is irrelevant in the following arguments so we  assume ${\bar \nu}=1$ from now on.  In what follows we work with symbols depending on $t$ where  $t$ varies in some fixed interval $[0, T]$ and it is assumed that all constants are independent of $t\in [0, T]$ and $\gamma$, $M$ if otherwise stated. Now $A \precsim B$ implies that  $A$ is bounded by constant, independent of $t$, $M$ and $\gamma$, times $B$. 
\begin{lem}
\label{lem:hyoka:omega}One has
\[
\dif_x^{\al}\dif_{\xi}^{\be}\omega^{s}\in S\big(M^{-(|\al+\be|-1)/2}\omega^s\omega^{-1}\rho^{1/2}\lr{\xi}^{-1/2+(|\al|-|\be|)/2}, g\big),\quad |\al+\be|\geq 1.
\]
\end{lem}
\begin{proof}
Recall $\omega^2=(t-\psi)^2+M\rho\lr{\xi}^{-1}$. Note that for $|\al+\be|\geq 2$
\begin{align*}
\big|\dif_x^{\alpha}\dif_{\xi}^{\beta}(t-\psi)^2\big|\precsim  \omega|\dif_x^{\alpha}\dif_{\xi}^{\beta}\psi|+\sum|\dif_x^{\alpha'}\dif_{\xi}^{\be'}\psi||\dif^{\alpha''}\dif_{\xi}^{\beta''}\psi|\\
\precsim \omega^2\big\{\omega^{-1}\rho^{1/2}\rho^{-(|\alpha+\beta|-1)/2}+\omega^{-2}\rho\rho^{-(|\alpha+\beta|-2)/2}\big\}\lr{\xi}^{-|\be|}
\\
\precsim \omega^2(\omega^{-1}\rho^{1/2}\lr{\xi}^{-1/2})M^{-(|\al+\be|-1)/2}\lr{\xi}^{(|\al|-|\be|)/2}
\end{align*}
since $\rho\geq M\lr{\xi}^{-1}$ and $\omega\geq \sqrt{M}\rho^{1/2}\lr{\xi}^{-1/2}$. When $|\alpha+\beta|=1$ it is clear
\begin{align*}
\big|\dif_x^{\alpha}\dif_{\xi}^{\beta}(t-\psi)^2\big|\precsim \omega\rho^{1/2}\lr{\xi}^{-|\be|}
=\omega^2(\omega^{-1}\rho^{1/2}\lr{\xi}^{-1/2})\lr{\xi}^{(|\alpha|-|\beta|)/2}.
\end{align*}
Next it is easy to see that for $|\al+\be|\geq 1$
\begin{align*}
\big|\dif_x^{\alpha}\dif_{\xi}^{\beta}(M\rho\lr{\xi}^{-1})\big|\precsim M\rho\lr{\xi}^{-1}\rho^{-|\alpha+\beta|/2}\lr{\xi}^{-|\beta|}\\
\precsim \omega^2\big(M\omega^{-2}\rho^{1/2}\lr{\xi}^{-1}\big)\big(M^{-1}\lr{\xi}\big)^{(|\alpha+\beta|-1)/2}\lr{\xi}^{-|\beta|}\\
\precsim \omega^2(\omega^{-1}\rho^{1/2}\lr{\xi}^{-1/2})M^{-(|\alpha+\beta|-1)/2}\lr{\xi}^{(|\alpha|-|\beta|)/2}
\end{align*}
because $\omega\geq \sqrt{M}\rho^{1/2}\lr{\xi}^{-1/2}\geq M\lr{\xi}^{-1}$.
Therefore one concludes that
\[
\big|\dif_x^{\alpha}\dif_{\xi}^{\beta}\omega^2\big|\precsim \omega^2(\omega^{-1}\rho^{1/2}\lr{\xi}^{-1/2})M^{-(|\alpha+\beta|-1)/2}\lr{\xi}^{(|\alpha|-|\beta|)/2}
\]
which proves the assertion for $s=2$. For general $s$ noting 
\begin{gather*}
\big|\dif_x^{\al}\dif_{\xi}^{\be}(\omega^2)^{s/2}\big|
\precsim \sum_{|\al^i+\be^i|\geq 1} \big|(\omega^2)^{s/2}\big(\dif_x^{\al^{1}}\dif_{\xi}^{\be^{1}}\omega^2/\omega^2\big)\cdots \big(\dif_x^{\al^{l}}\dif_{\xi}^{\be^{l}}\omega^2/\omega^2\big)\big|
\end{gather*}
the proof is immediate since $\omega^{-1}\rho^{1/2}\lr{\xi}^{-1/2}\leq M^{-1/2}\leq 1$.
\end{proof}
\begin{cor}
\label{cor:dif:omega} We have $\omega^s\in S(\omega^s, g)$ for $s\in \R$.
\end{cor}
%

\subsection{Estimate $\phi$ by metric $g$}
 
 Introduce a wight function which plays a crucial role in deriving energy estimates
\[
\phi(t, x,\xi)=\omega(t, x,\xi)+t-\psi(x,\xi).
\]
If $t-\psi(x,\xi)\geq 0$ then $\phi\geq \omega=\omega^2/\omega\geq M\rho\lr{\xi}^{-1}/\omega$ and if $t-\psi(x,\xi)\geq 0$ we see $\phi=M\rho\lr{\xi}^{-1}/(\omega+|t-\psi|)\geq \rho M\lr{\xi}^{-1}/(2\omega)$ hence   
\begin{equation}
\label{eq:Phi:sita}
\phi(t, x,\xi)\geq M\rho\lr{\xi}^{-1}/(2\,\omega).
\end{equation}
\begin{lem}
\label{lem:psi:t:sei}
There is $C>0$ such that $
 \phi(t, x,\xi)\geq M\lr{\xi}^{-1}/C$.
\end{lem}
\begin{proof} When $t-\psi(x,\xi)\geq 0$ then $\phi\geq \omega\geq M^{1/2}\rho^{1/2}\lr{\xi}^{-1/2}\geq M\lr{\xi}^{-1}$ is obvious for $\rho\geq M\lr{\xi}^{-1}$. 
Assume $t-\psi(x,\xi)<0$ then $0\leq t<\psi(x,\xi)\leq \delta\rho(x,\xi)$ with some $\delta>0$ by Lemma \ref{lem:dif:psi:1}. Noticing that $ |t-\psi(x,\xi)|=\psi(x,\xi)-t\leq \delta\rho(x,\xi)$ we have $
\omega^2(t, x,\xi)\leq \delta^2\rho^2+M\rho\lr{\xi}^{-1}
\leq \delta^2\rho^2+\rho^2=(\delta^2+1)\rho^2$. 
Now the proof is immediate from \eqref{eq:Phi:sita}.
\end{proof}
\begin{lem}
\label{lem:dif:Phi}We have $\phi\in S(\phi, g)$.
\end{lem}
\begin{proof}
Let $|\alpha+\beta|=1$ and write
\begin{equation}
\label{eq:Phi:bunkai}
\dif_x^{\alpha}\dif_{\xi}^{\beta}\phi=\frac{-\dif_x^{\alpha}\dif_{\xi}^{\beta}\psi}{\omega}\phi+\frac{\dif_x^{\alpha}\dif_{\xi}^{\beta}(M\rho\lr{\xi}^{-1})}{2\omega}=\phi_{\alpha\beta}\phi+\psi_{\alpha\beta}.
\end{equation}
From Corollary \ref{cor:dif:omega} and  Lemma \ref{lem:hyoka:rho:s} it follows that 
\begin{gather*}
\big|\dif_x^{\mu}\dif_{\xi}^{\nu}\big(\psi_{\alpha\beta}\big)\big|\precsim \omega^{-1}M\rho\lr{\xi}^{-1}M^{-|\alpha+\beta+\mu+\nu|/2}\lr{\xi}^{(|\alpha+\mu|-|\beta+\nu|)/2}\\
\precsim \phi M^{-|\alpha+\beta+\mu+\nu|/2}\lr{\xi}^{(|\alpha+\mu|-|\beta+\nu|)/2}
\end{gather*}
in view of \eqref{eq:Phi:sita}. On the other hand thanks to Lemma \ref{lem:hyoka:psi:b} and Corollary \ref{cor:dif:omega} it follows that $
|\dif_x^{\mu}\dif_{\xi}^{\nu}\phi_{\alpha\beta}|\precsim M^{-|\alpha+\beta+\mu+\nu|/2}\lr{\xi}^{(|\alpha+\mu|-|\beta+\nu|)/2}$. 
Hence using \eqref{eq:Phi:bunkai} the assertion is proved by induction on $|\alpha+\beta|$.
\end{proof}
We refine this lemma.
\begin{lem}
\label{lem:dif:Phi:seimitu}One has
\[
\dif_x^{\alpha}\dif_{\xi}^{\beta}\phi\in S(\phi\,M^{-(|\al+\be|-1)/2}\omega^{-1}\rho^{1/2}\lr{\xi}^{-1/2}\lr{\xi}^{(|\al|-|\beta|)/2},g),\quad |\alpha+\beta|\geq 1.
\]
\end{lem}
\begin{proof}From Lemma  \ref{lem:hyoka:psi:b} one has $\dif_x^{\alpha}\dif_{\xi}^{\beta}\psi\in S(\rho^{1/2}\lr{\xi}^{-1/2}\lr{\xi}^{(|\al|-|\beta|)/2},g)$ for $|\alpha+\beta|=1$ hence $\phi_{\alpha\beta}\in S(\omega^{-1}\rho^{1/2}\lr{\xi}^{-1/2}\lr{\xi}^{(|\al|-|\beta|)/2},g)$ for $|\alpha+\beta|=1$ by Lemma \ref{lem:hyoka:omega}. From Lemma \ref{lem:hyoka:omega}  it follows that
\[
\big|\dif_x^{\mu}\dif_{\xi}^{\nu}\big(\psi_{\alpha\beta}\big)\big|\precsim \omega^{-1}\rho^{1/2}M\lr{\xi}^{-1-|\beta|}M^{-|\mu+\nu|/2}\lr{\xi}^{(|\mu|-|\nu|)/2}
\]
for $|\alpha+\beta|=1$ because $\dif_x^{\al}\dif_{\xi}^{\be}(M\rho\lr{\xi}^{-1})\in S(M\rho^{1/2}\lr{\xi}^{-1-|\be|}, g)$. Thanks to Lemma \ref{lem:psi:t:sei} one sees  $M\lr{\xi}^{-1}\leq C\phi(t, x,\xi)$ and  hence 
\[
\psi_{\alpha\beta}\in S(\omega^{-1}\rho^{1/2}\lr{\xi}^{-1/2}\lr{\xi}^{(|\al|-|\beta|)/2}\phi,g),\quad |\alpha+\beta|=1.
\]
Since $\phi\in S(\phi, g)$ by Lemma \ref{lem:dif:Phi}  we conclude the assertion from \eqref{eq:Phi:bunkai}.
\end{proof}
%

\section{$\phi$ and $\lambda_j$  are admissible weights for $g$}
\label{sec:metric}

Write $z=(x,\xi)$ and $w=(y,\eta)$. It is clear that
\[
g^{\sigma}_z=M\big(\lr{\xi}|dx|^2+\lr{\xi}^{-1}|d\xi|^2\big)=M^2g_z
\]
where $g_z^{\sigma}(t_1, t_2)=\sup|\langle{t_2, s_1}\rangle-\langle{t_1, s_2}\rangle|^2/g_z(s_1, s_2)$ (see \cite[Chapter 18]{Hobook}). 
Note that $|\xi-\eta|\leq c\,\lr{\xi}$ with $0<c<1$ implies 
\[
(1-c)\lr{\xi}/\sqrt{2}\leq \lr{\eta}\leq \sqrt{2}\,(1+c)\lr{\xi}.
\]
If $g_z(w)<c$ then $|\xi-\eta|^2<c\,M\lr{\xi}=c\,M\lr{\xi}^{-1}\lr{\xi}^2\leq c\,\lr{\xi}^2$  then 
\[
g_z(X)/C\leq g_w(X)\leq Cg_z(X),\quad X\in \R^d\times \R^d
\]
with $C$ independent of $\gamma\geq M^5\geq 1$ that is $g_z$ is slowly varying uniformly in $\gamma\geq M^5\geq 1$. Similarly noting that $|\xi-\eta|\geq (\gamma+|\xi|)/2\geq \lr{\xi}/2$ if $\lr{\eta}\leq \lr{\xi}/2\sqrt{2}$ and $|\xi-\eta|\geq (\gamma+|\eta|)/2\geq \lr{\eta}/2$ if $\lr{\eta}\geq 2\sqrt{2}\lr{\xi}$ it is clear that 
\begin{equation}
\label{eq:g:sei:1}
\frac{\lr{\xi}}{\lr{\eta}}+\frac{\lr{\eta}}{\lr{\xi}}\leq C\big(1+\lr{\eta}^{-1}|\xi-\eta|^2\big)\leq C(1+g_w^{\sigma}(z-w)\big)
\end{equation}
hence $g_w(X)\leq Cg_z(X)\big(1+g_w^{\sigma}(z-w)\big)$, that is $g$ is  a temperate  metric uniformly in $\gamma\geq 0$ and $ M\geq 1$  (see \cite[Chapter 18]{Hobook}). It is clear from \eqref{eq:g:sei:1} that
\begin{equation}
\label{eq:g:sei:2}
g_z^{\sigma}(z-w)\leq C\big(1+g_w^{\sigma}(z-w)\big)^2.
\end{equation}
%

\subsection{$\rho$ and $\sigma$ are admissible weights for $g$}

We adapt the same convention as in  Sections \ref{sec:kyori},  \ref{sec:Bezout} even to weights for $g$ so that we omit  to say uniformly in  $t\in [0,M^{-4}]$.

\begin{lem}
\label{lem:rho:wei}$\rho$ is an admissible weight for $g$.
\end{lem}
\begin{proof}
First study $\rho^{1/2}$. Assume $
g_z(w)=M^{-1}\lr{\xi}\big(|y|^2+\lr{\xi}^{-2}|\eta|^2\big)<c\;(<1/2)$ 
so that $M^{-1}\lr{\xi}^{-1}|\eta|^2<c$ hence $|\eta|<c\lr{\xi}$ for $M\lr{\xi}^{-1}\leq 1$ hence
 \begin{equation}
 \label{eq:doto}
 \lr{\xi+s\eta}/C\leq \lr{\xi}\leq C\lr{\xi+s\eta}
 \end{equation}
 where $C$ is independent of  $|s|\leq 1$.  Lemma \ref{lem:hyoka:rho:s} shows
\[
|\rho^{1/2}(z+w)-\rho^{1/2}(z)|\leq C\big(|y|+\lr{\xi+s\eta}^{-1}|\eta|\big)\leq CM^{1/2}\lr{\xi}^{-1/2}g_z^{1/2}(w).
\]
Since $\rho(z)\geq M\lr{\xi}^{-1}$ this yields 
\begin{equation}
\label{eq:rho:sa}
|\rho^{1/2}(z+w)-\rho^{1/2}(z)|\leq C\rho^{1/2}(z)g_z^{1/2}(w).
\end{equation}
Choosing $c$ such that $C\,c<1/2$ one has $\big|\rho(z+w)/\rho(z)-1\big|<1/2$ which implies  $\rho^{1/2}(z+w)/2\leq \rho^{1/2}(z)\leq 3\,\rho^{1/2}(z+w)/2$ 
that is $\rho^{1/2}$ is $g$ continuous hence so is $\rho$. Note that $M\lr{\xi}^{-1}\leq  \rho(z)\leq CM^{-4}
\leq C$. 
If $|\eta|\geq c\,\lr{\xi}/2$ then $g_z^{\sigma}(w)\geq Mc^2\lr{\xi}/4$ and $g_z^{\sigma}(w)\geq Mc|\eta|/2$ thus
\[
\rho(z+w)\leq C
\leq C\lr{\xi}\rho(z)\leq C'\rho(z)(1+g_z^{\sigma}(w)).
\]
If $|\eta|\leq c\lr{\xi}$ then  \eqref{eq:rho:sa} gives
\begin{equation}
\label{eq:rho:ue}
\rho^{1/2}(z+w)\leq C\rho^{1/2}(z)(1+g_z(w))^{1/2}\leq C\rho^{1/2}(z)(1+g_z^{\sigma}(w))^{1/2}
\end{equation}
hence, in view of \eqref{eq:g:sei:2}, $\rho$ is an admissible weight.
\end{proof}
\begin{lem}
\label{lem:a:kihon}$\sigma$ is an admissible weight for $g$ and $\sigma\in S(\sigma, g)$.
\end{lem}
\begin{proof}Since $\rho(z)+M\lr{\xi}^{-1}$ is admissible for $g$ by Lemma  \ref{lem:rho:wei} then it is clear that so is $\sigma=t+\rho(z)+M\lr{\xi}^{-1}$ for $t\geq 0$. The second assertion is clear from $
\big|\dif_x^{\al}\dif_{\xi}^{\be}\sigma\big|\precsim \sigma^{1-|\al+\be|/2}\lr{\xi}^{-|\be|}\precsim \sigma(M^{-1}\lr{\xi})^{|\al+\be|/2}\lr{\xi}^{-|\be|}$ 
for $\sigma\geq M\lr{\xi}^{-1}$.
\end{proof}
%

\subsection{$\omega$ and $\phi$ are admissible weights for $g$}

We start with showing
\begin{lem}
\label{lem:g:conti}
$\omega$ and $\phi$ are $g$ continuous.
\end{lem}
\begin{proof} 
Denote $f=t-\psi$ and $h=M^{1/2}\rho^{1/2}\lr{\xi}^{-1/2}$ so that $\omega^2=f^2+h^2$. Note that
\begin{equation}
\label{eq:ome:sa:a}
\begin{split}
|\omega(z+w)-\omega(z)|=|\omega^2(z+w)-\omega^2(z)|/|\omega(z+w)+\omega(z)|\\
\leq 2|f(z+w)-f(z)|+2|h(z+w)-h(z)|
\end{split}
\end{equation}
because $
|f(z+w)+f(z)|/|\omega(z+w)+\omega(z)|\leq 2$, $|h(z+w)+h(z)|/|\omega(z+w)+\omega(z)|\leq 2$.
Assume $g_z(w)<c\;(\leq 1/2)$ hence \eqref{eq:doto}. It is assumed that constants $C$ may change from line to line but independent of $\gamma\geq M^5\geq 1$. Noting $|f(z+w)-f(z)|=|\psi(z+w)-\psi(z)|$ it follows  from Lemma \ref{lem:hyoka:psi:b} that
\begin{equation}
\label{eq:f:sasa}
\begin{split}
|f(z+w)-f(z)|\leq C \rho^{1/2}(z+s w)\big(|y|+\lr{\xi+s\eta}^{-1}|\eta|\big)\\\leq C \rho^{1/2}(z+s w)\big(|y|+\lr{\xi}^{-1}|\eta|\big)\leq CM^{1/2}\rho^{1/2}(z)\lr{\xi}^{-1/2}g_z^{1/2}(w)
\end{split}
\end{equation}
since $\rho$ is $g$ continuous. Noting that $\omega(z)\geq M^{1/2}\rho^{1/2}(z)\lr{\xi}^{-1/2}$ it results
\begin{equation}
\label{eq:f:sa}
|f(z+w)-f(z)|\leq C\omega(z) g_z^{1/2}(w).
\end{equation}
Similar argument shows that $
|h(z+w)-h(z)|\leq CM^{1/2}\lr{\xi}^{-1}g_z^{1/2}(w)$. 
Since $\omega(z)\geq M\lr{\xi}^{-1}$ we have $
|h(z+w)-h(z)|\leq CM^{-1/2}\omega(z)g_z^{1/2}(w)$. 
Therefore from \eqref{eq:ome:sa:a} one has $
|\omega(z+w)-\omega(z)|\leq C\omega(z)g_z^{1/2}(w)$. 
Choosing $c$ such that $C\,c<1/2$ we conclude that $\omega$ is $g$ continuous. Next consider $\phi=\omega+f$. Write 
\begin{equation}
\label{eq:Phi:sa}
\begin{split}
&\phi(z+w)-\phi(z)\\=&\frac{(f(z+w)-f(z))(\phi(z+w)+\phi(z))+h^2(z+w)-h^2(z)}{\omega(z+w)+\omega(z)}.
\end{split}
\end{equation}
Since $
\omega(z+w)/C\leq \omega(z)\leq C\,\omega(z+w)$, decreasing $c>0$ if necessary, which together with \eqref{eq:f:sa}  gives $
|f(z+w)-f(z)|/(\omega(z+w)+\omega(z))\leq Cg_z^{1/2}(w)$. 
Recalling $h^2(z)=M\rho(z)\lr{\xi}^{-1}$ and repeating similar arguments one sees
\begin{equation}
\label{eq:h2:sa}
\begin{split}
|h^2(z+w)-h^2(z)|\leq CM^{1/2}\rho(z)\lr{\xi}^{-1}g_z^{1/2}(w)
\end{split}
\end{equation}
for $\rho^{1/2}(z)\geq M^{1/2}\lr{\xi}^{-1/2}$. 
Taking \eqref{eq:Phi:sita} into account it follows from \eqref{eq:h2:sa} that
\[
|h^2(z+w)-h^2(z)|/(\omega(z+w)+\omega(z))\leq C\phi(z)g_z^{1/2}(w).
\]
Combining these estimates we obtain from \eqref{eq:Phi:sa} that
\[
\big|\phi(z+w)/\phi(z)-1\big|\leq C\big|\phi(z+w)/\phi(z)+1\big|g_z^{1/2}(w)+Cg_z^{1/2}(w)
\]
which gives $
\phi(z)/C\leq \phi(z+w)\leq C\,\phi(z)$ 
choosing $c>0$ small, showing that $\phi$ is $g$ continuous.
\end{proof}
\begin{lem}
\label{lem:Phi:kajyu}$\omega$ and $\phi$ are admissible weights for $g$  and $\omega\in S(\omega,g)$, $\phi\in S(\phi, g)$.
\end{lem}
\begin{proof}Note that $
\lr{\xi}^{-1}\leq M\lr{\xi}^{-1}\leq \sqrt{M}\sqrt{\rho} \lr{\xi}^{-1/2}\leq \omega \leq CM^{-4}\leq C$. 
Assume $|\eta|\geq c\,\lr{\xi}$ hence $g_z^{\sigma}(w)\geq Mc^2\lr{\xi}\geq c^2\lr{\xi}$. Therefore
\begin{equation}
\label{eq:ome:kan}
\omega(z+w)\leq C\leq C\lr{\xi}\omega(z)\leq C'\omega(z)(1+g^{\sigma}_z(w)).
\end{equation}
Assume $|\eta|\leq c\,\lr{\xi}$ and note that \eqref{eq:rho:ue}. Then checking the proof of Lemma \ref{lem:g:conti} we see that $|f(z+w)-f(z)|\leq C\omega(z)(1+g_z^{\sigma}(w))$ and $|h(z+w)-h(z)|\leq C\omega(z)(1+g_z^{\sigma}(w))^{1/2}$. Then \eqref{eq:ome:kan} follows from \eqref{eq:ome:sa:a} 
which proves that $\omega$ is admissible for $g$. Turn to $\phi$. From Lemma \ref{lem:psi:t:sei}  it follows that
\[
\lr{\xi}^{-1}/C\leq M\lr{\xi}^{-1}/C\leq \phi(z)=\omega(z)+f(z)\leq CM^{-4}\leq C.
\]
If $|\eta|\geq \lr{\xi}/2$ then $g_z^{\sigma}(w)\geq M\lr{\xi}/4\geq \lr{\xi}/4$ hence
\begin{align*}
\phi(z+w)\leq C\leq C^2\lr{\xi}\phi(z)\leq C\phi(z)(1+g_z^{\sigma}(w)).
\end{align*}
Assume $|\eta|\leq \lr{\xi}/2$ so that \eqref{eq:doto} holds. From \eqref{eq:rho:ue} and \eqref{eq:f:sasa} it results that
\[
|f(z+w)-f(z)|\leq C\rho^{1/2}(z)\lr{\xi}^{-1/2}(1+g_z^{\sigma}(w)).
\]
Recalling \eqref{eq:rho:ue}  and $M^2g_z(w)=g_z^{\sigma}(w)$ the same arguments obtaining \eqref{eq:h2:sa} shows that $
|h^2(z+w)-h^2(z)|\leq C\rho^{1/2}(z)\lr{\xi}^{-3/2	}(1+g_z^{\sigma}(w))$. 
Taking these into account \eqref{eq:Phi:sa} yeilds
\begin{equation}
\label{eq:Phi:sa:bis}
\begin{split}
|\phi(z+w)-\phi(z)|\leq C\Big(\frac{\rho^{1/2}(z)\lr{\xi}^{-1/2}}{\omega(z+w)+\omega(z)}\big(\phi(z+w)+\phi(z)\big)\\
+\frac{\rho^{1/2}(z)\lr{\xi}^{-3/2}}{\omega(z+w)+\omega(z)}\Big)(1+g_z^{\sigma}(w)).
\end{split}
\end{equation}
Applying Lemma \ref{lem:psi:t:sei} to \eqref{eq:Phi:sa:bis} to obtain
\begin{align*}
|\phi(z+w)-\phi(z)|
\leq C\big(\phi(z+w)+2\phi(z)\big)\frac{\rho^{1/2}(z)\lr{\xi}^{-1/2}}{\omega(z+w)+\omega(z)}(1+g_z^{\sigma}(w)).
\end{align*}
If $\rho^{1/2}(z)\lr{\xi}^{-1/2}(1+g_z^{\sigma}(w))\big/(\omega(z+w)+\omega(z))<1/4$ then it follows 
\[
\big|\phi(z+w)/\phi(z)-1\big|\leq \big(\phi(z+w)/\phi(z)+2\big)/4
\]
from which we have $\phi(z+w)\leq 2\phi(z)\leq 5\phi(z+w)$. If
\[
\rho^{1/2}(z)\lr{\xi}^{-1/2}(1+g_z^{\sigma}(w))\big/(\omega(z+w)+\omega(z))\geq 1/4
\]
we have $
32(1+g_z^{\sigma}(w))^2\geq 4\lr{\xi}\omega(z+w)\omega(z)\big/\rho(z)\geq \phi(z+w)\big/\phi(z)$ 
by \eqref{eq:Phi:sita} and an obvious inequality $\phi(z+w)\leq 2\,\omega(z+w)$. Thus we conclude that $\phi$ is  admissible for $g$.
\end{proof}
%

\subsection{$\lambda_j$ are admissible weights for $g$}

\begin{lem}
\label{lem:lam:ipan}Assume that $\lambda\in {\mathcal C}(\sigma^2)$ and $\lambda\geq cM\sigma\lr{\xi}^{-1}$ with some $c>0$. Then $\lambda$ is an admissible weight for $g$.
\end{lem}
\begin{proof}
Consider $\sqrt{\lambda}$. Assume $g_z(w)<c$ and hence $\lr{\xi+s\eta}\approx \lr{\xi}$. Since $\sqrt{\lambda}\in {\mathcal C}(\sigma)$  it follows that
\begin{equation}
\label{eq:lakanzo}
\begin{split}
\big|\sqrt{\lambda(z+w)}-\sqrt{\lambda(z)}\big|\leq C\sqrt{\sigma(z+s w)}M^{1/2}\lr{\xi}^{-1/2}g_z^{1/2}(w)
\end{split}
\end{equation}
with $|s|<1$ which is bounded by $C'\sqrt{\sigma(z)}M^{1/2}\lr{\xi}^{-1/2}g_z^{1/2}(w)$ since $\sigma$ is $g$ continuous. By assumption $\lambda(z)\geq cM\sigma(z)\lr{\xi}^{-1}$ one has
\[
 \big|\sqrt{\lambda(z+w)}-\sqrt{\lambda(z)}\big|\leq C''\sqrt{\lambda(z)}\,g_z^{1/2}(w).
 \]
 Choosing $c>0$ such that $C''\sqrt{c}<1$ shows that $\sqrt{\lambda(z)}$ is $g$ continuous and so is $\lambda(z)$. From $cM^2\lr{\xi}^{-2}\leq cM\sigma \lr{\xi}^{-1}\leq \lambda\leq C'\sigma^2\leq C'M^{-4}$ one sees
 \[
 c_1M\lr{\xi}^{-1}\leq c_1M^{1/2}\sigma^{1/2}\lr{\xi}^{-1/2}\leq\sqrt{\lambda(z)}\leq C.
 \]
If $|\eta|\geq \lr{\xi}/2$ hence $g_z^{\sigma}(w)\geq M\lr{\xi}/4$  then 
\begin{align*}
\sqrt{\lambda(z+w)}\leq C
\leq C(c_1M)^{-1}\lr{\xi}\sqrt{\lambda(z)}\leq C'\sqrt{\lambda(z)}\,g_z^{\sigma}(w).
\end{align*}
If $|\eta|\leq \lr{\xi}/2$, noting that $\sigma(z+w)\leq C\,\sigma(z)(1+g_z^{\sigma}(w))$, it follows from \eqref{eq:lakanzo}
\begin{align*}
\big|\sqrt{\lambda(z+w)}-\sqrt{\lambda(z)}\big| \leq C\sqrt{\lambda(z)}(1+g_z^{\sigma}(w))
\end{align*}
which proves that $\sqrt{\lambda}$ is admissible weight for $g$ and hence so is $\lambda$.
\end{proof}
\begin{lem}
\label{lem:a:ipan}Assume that $\lambda\in {\mathcal C}(\sigma)$ and $\lambda\geq cM\lr{\xi}^{-1}$ with some $c>0$. Then $\lambda$ is an admissible weight for $g$. If $\lambda\in {\mathcal C}(1)$ and $\lambda\geq c$ with some $c>0$ then $\lambda$ is an admissible weight for $g$.
\end{lem}
\begin{proof}It is enough to repeat the proof of Lemma \ref{lem:lam:ipan}.
\end{proof}
\begin{lem}
\label{lem:lam:Sg} Assume $\lambda\in {\mathcal C}(\sigma^2)$ and $\lambda\geq cM\sigma \lr{\xi}^{-1}$ with some $c>0$. Then 
\[
\dif_x^{\al}\dif_{\xi}^{\be}\lambda\in S(\sqrt{\sigma}\sqrt{\lambda}\lr{\xi}^{-|\be|}, g),\quad |\al+\be|=1.
\]
In particular $\lambda\in S(\lambda,g)$. 
\end{lem}
\begin{proof}
From $\lambda\in {\mathcal C}(\sigma^2)$ we have $|\lr{\xi}^{|\beta|}\dif_x^{\alpha}\dif_{\xi}^{\beta}\lambda|\leq C\, \sigma$ for $|\alpha+\beta|=2$. Since $\lambda\geq 0$, thanks to the Glaeser's inequality one has $
\big|\dif_x^{\alpha}\dif_{\xi}^{\beta}\lambda\big|\leq C'\,\sqrt{\sigma}\,\sqrt{\lambda}\,\lr{\xi}^{-|\beta|}$ for $|\alpha+\beta|=1$. 
For $|\al'+\be'|\geq 1$  note that 
\begin{align*}
\big|\dif_x^{\al'}\dif_{\xi'}^{\be'}\big(\dif_x^{\al}\dif_{\xi}^{\be}\lambda\big)\big|
\precsim \sigma^{1-(|\al'+\be'|-1)/2}\lr{\xi}^{-|\be|}\lr{\xi}^{-|\be'|}\\
\precsim \sigma M^{-|\al'+\be'|/2}M^{1/2}\lr{\xi}^{-1/2}\lr{\xi}^{(|\al'|-|\be'|)/2}\lr{\xi}^{-|\be|}\\
\precsim \sqrt{\sigma}M^{-|\al'+\be'|/2}\sqrt{\lambda}\,\lr{\xi}^{(|\al'|-|\be'|)/2}\lr{\xi}^{-|\be|}
\end{align*}
because $\lambda\geq cM\sigma\lr{\xi}^{-1}$ and $\sigma\geq M\lr{\xi}^{-1}$ which proves the first assertion. Noting $
\sqrt{\sigma}\lr{\xi}^{-|\be|}\leq CM^{-1/2}\sqrt{\lambda}\,\lr{\xi}^{(|\al|-|\be|)/2}
$ 
it is clear that $\lambda\in S(\lambda, g)$.
\end{proof}
\begin{lem}
\label{lem:a:Sg} Assume that $\lambda\in {\mathcal C}(\sigma)$ and $\lambda\geq cM\lr{\xi}^{-1}$ with some $c>0$. Then $\lambda\in S(\lambda, g)$. If $\lambda\in {\mathcal C}(1)$ and $\lambda\geq c$ with some $c>0$. Then $\lambda\in S(\lambda, g)$.
\end{lem}
\begin{proof}It suffices to repeat the proof of Lemma \ref{lem:lam:Sg}.
\end{proof}
\begin{cor}
\label{cor:latoa}For $s\in \R$ we have $\lambda_j^s\in S(\lambda_j^s,g)$, $j=1, 2, 3$.
\end{cor}
Define
\[
\kappa=\frac{1}{t}+\frac{1}{\omega}=\frac{t+\omega}{t\,\omega},\quad t>0.
\]
\begin{lem}
\label{lem:kapa}$\kappa$ is an admissible weight for $g$ and $\kappa^s\in S(\kappa^s, g)$ for $s\in\R$.
\end{lem}
\begin{proof}
Since $\omega^{-1}$ is admissible for $g$ it is clear that so is $\kappa=t^{-1}+\omega^{-1}$. Noting that $\omega^{-1}\in S(\omega^{-1},g)$ and $\omega^{-1}\leq \kappa$ it is also clear that 
\[
\big|\dif_x^{\al}\dif_{\xi}^{\be}\kappa\big|=\big|\dif_x^{\al}\dif_{\xi}^{\be}\omega^{-1}\big|\precsim M^{-|\al+\be|/2}\kappa\lr{\xi}^{(|\al|-|\be|)/2}
\]
for $|\al+\be|\geq 1$  which proves $\kappa\in S(\kappa, g)$.
\end{proof}
\begin{lem}
\label{lem:kapa:b}One has
\[
\dif_x^{\al}\dif_{\xi}^{\be}\kappa^{s}\in S\big(M^{-(|\al+\be|-1)/2}\kappa^s\omega^{-1}\rho^{1/2}\lr{\xi}^{-1/2+(|\al|-\be|)/2},g\big),\;\;|\al+\be|\geq 1.
\]
\end{lem}
\begin{proof}Since $\dif_x^{\al}\dif_{\xi}^{\be}\kappa^s=\kappa^{s-1}\dif_x^{\al}\dif_{\xi}^{\be}\kappa$ it is enough to show the case $s=1$. The proof for the case $s=1$ follows  easily from Lemma \ref{lem:hyoka:omega}. 
\end{proof}
\begin{lem}
\label{lem:daiji}
Denote ${\bar\varep }=\sqrt{3}/{\bar e}\,{\bar c}^{1/2}=4\sqrt{6}/{\bar e}$. There is $C>0$ such that
\[
\frac{1}{\kappa \lambda_1}\leq {\bar \varep}^2(1+CM^{-4})\kappa,\qquad \frac{1}{\sigma^2 \kappa}\leq \kappa.
\]
\end{lem}
\begin{proof}In view of Propositions \ref{pro:ext:matome} and \ref{pro:Skon} one sees
\[
\lambda_1\geq (1/{\bar \varep}^{2})(1-CM^{-4})\min{\big\{t^2, \omega^2\big\}}.
\]
Denote $c={\bar \varep}^2(1-CM^{-4})^{-1}$. If $\omega^2\geq t^2$ and hence $\lambda_1\geq t^2/c$ then $1/\lambda_1\leq c/t^2$ which shows that
\[
\frac{1}{\kappa \lambda_1}\leq \frac{c}{\kappa t^2}=\frac{c\,t \,\omega}{(t+\omega)t^2}=\frac{c\,\omega}{(t+\omega)t}\leq \frac{c(t+\omega)}{t\omega}=c\kappa.
\]
If $t^2\geq \omega^2$ and hence $\lambda_1\geq \omega^2/c$ then $1/\lambda_1\leq c/\omega^2$ and hence
\[
\frac{1}{\kappa \lambda_1}\leq \frac{c}{\kappa \omega^2}=\frac{c\, t\, \omega}{(t+\omega)\omega^2}=\frac{c\, t}{(t+\omega)\omega}\leq \frac{c(t+\omega)}{t\,\omega}=c\kappa
\]
thus the first assertion. To show the second assertion it suffices to note $\sigma \geq t$ and then $\sigma^2(t+\omega)^2\geq t^2(t+\omega)^2\geq t^2\omega^2$.
\end{proof}
%

\section{Lower bounds of $\op{\lambda_j}$}
\label{sec:bound:lam}

\subsection{Some preliminary lemmas}

Introduce a metric ${\bar g}=\lr{\xi}|dx|^2+\lr{\xi}^{-1}|d\xi|^2$ independent of $M$ so that $g=M^{-1}\,{\bar g}$. We start with
\begin{lem}
\label{lem:gyaku}
Let $m$ be an admissible weight for $g$ and $p\in S(m, g)$ satisfy
$p\geq c\,m$ with some constant $c>0$. Then $p^{-1}\in S(m^{-1},g)$ and there exist $k, {\tilde k}\in S(M^{-1},g)$ such that
\begin{align*}
p\#p^{-1}\#(1+k)=1,\;\;(1+k)\#p\#p^{-1}=1,\;\;p^{-1}\#(1+k)\#p=1,\\
p^{-1}\#p\#(1+{\tilde k})=1,\;\;(1+{\tilde k})\#p^{-1}\#p=1,\;\;p\#(1+{\tilde k})\#p^{-1}=1.
\end{align*}
\end{lem}
\begin{proof}In this proof every constant is independent of $\gamma\geq 1$ and $M$. It is clear that $p^{-1}\in S(m^{-1},g)$. Write $p\#p^{-1}=1-r$ where $r\in S(M^{-1},g)$.  Since 
 \[
 |r|^{(l)}_{S(1,{\bar g})}=\sup_{|\al+\be|\leq l, (x,\xi)\in\R^{2d}}\big|\lr{\xi}^{(|\be|-|\al|)/2}\dif_x^{\al}\dif_{\xi}^{\be}r\big|\leq C_lM^{-1}
 \]
 from the $L^2$-boundedness theorem (see \cite[Theorem 18.6.3]{Hobook}) we have $\|{\rm op}(r)\|\leq CM^{-1}$. Therefore for large $M$ there exists the inverse $(1-\op{r})^{-1}$ which is given by $1+\sum_{\ell=1}^{\infty}r^{\#\ell}\in S(1,{\bar g})$.  (see \cite{Be}, \cite{Ler}).  Denote $k=\sum_{\ell=1}^{\infty}r^{\#\ell}\in S(1, {\bar g})$ and prove $k\in S(M^{-1}, g)$. It can be seen from the proof (e.g.  \cite{Ler}) that for any $l\in \N$ one can find $C_l>0$, independent of $\gamma$,  such that
\[
|k|^{(l)}_{S(1,{\bar g})}\leq C_l
\]
because $|k|^{(l)}_{S(1,{\bar g})}$ depends only on $l$,  $ |r|^{(l')}_{S(1,{\bar g})}$ with some $l'=l'(l)$ and structure constants of ${\bar g}$ which is independent of $\gamma$. Note that $k$ satisfies $(1-r)\#(1+k)=1$, that is
\begin{equation}
\label{eq:knosiki}
k=r+r\#k.
\end{equation}
Since $r\in S(M^{-1},g)$ it follows from \eqref{eq:knosiki} that  $
\big|k\big|^{(l)}_{S(1,{\bar g})}\leq C_lM^{-1}$.
Assume that 
\begin{equation}
\label{eq:kino}
\sup\big|\lr{\xi}^{(|\be|-|\al|)/2}\dif_x^{\alpha}\dif_{\xi}^{\beta}k\big|\leq C_{\al,\be,\nu}M^{-1-l/2},\quad |\alpha+\beta|\geq l
\end{equation}
for $0\leq l\leq \nu$.
Let $|\alpha+\beta|\geq \nu+1$ and note that
\[
\dif_x^{\alpha}\dif_{\xi}^{\beta}k=\dif_x^{\alpha}\dif_{\xi}^{\beta}r+\sum C_{\cdots}\big(\dif_x^{\alpha''}\dif_{\xi}^{\beta''}r\big)\#\big(\dif_x^{\alpha'}\dif_{\xi}^{\beta'}k\big)
\]
where $\alpha'+\alpha''=\alpha$ and $\beta'+\beta''=\beta$. From the assumption \eqref{eq:kino} we have $\dif_x^{\al'}\dif_{\xi}^{\be'}k\in S(M^{-1-|\al'+\be'|/2}\lr{\xi}^{(|\al'|-|\be'|)/2}, {\bar g})$ if $|\al'+\be'|\leq \nu$ and $\dif_x^{\al'}\dif_{\xi}^{\be'}k\in S(M^{-1-\nu/2}\lr{\xi}^{(|\al'|-|\be'|)/2}, {\bar g})$ if $|\al'+\be'|\geq \nu+1$. Since $r\in S(M^{-1}, g)$ one has
\[
\big(\dif_x^{\alpha''}\dif_{\xi}^{\beta''}r\big)\#\big(\dif_x^{\alpha'}\dif_{\xi}^{\beta'}k\big)\in S(M^{-1-(\nu+2)/2}\lr{\xi}^{(|\al|-|\be|)/2}, {\bar g})
\]
which implies that \eqref{eq:kino} holds for $0\leq l\leq \nu+1$ and hence for all $\nu$ by induction on $\nu$.
This  proves that $k\in S(M^{-1},g)$. The proof of the assertions for ${\tilde k}$ is similar.
\end{proof}
Here recall \cite[Lemmas 3.1.6, 3.1.7]{Ni:book}.
\begin{lem}
\label{lem:Ni:hon:1} Let $q\in S(1, g)$ satisfy $q\geq c$ with a constant $c$ independent of $M$. Then there is $C>0$ such that 
\[
\big(\op{q}u,u) \geq (c-CM^{-1/2})\|u\|^2.
\]
\end{lem}
\begin{proof}One can assume $c=0$. We see $q(x,\xi)+M^{-1/2}$ is an admissible weight for ${\bar g}$ and $(q+M^{-1/2})^{1/2}\in S((q+M^{-1/2})^{1/2}, {\bar g})$. Moreover $\dif_x^{\al}\dif_{\xi}^{\be}(q+M^{-1/2})^{1/2}\in S(M^{-1/2}\lr{\xi}^{(|\al|-|\be|)/2}, {\bar g})$ for $|\al+\be|=1$. Therefore
\[
q+M^{-1/2}=(q+M^{-1/2})^{1/2}\#(q+M^{-1/2})^{1/2}+r,\quad r\in S(M^{-1}, {\bar g}
)
\]
which proves the assertion.
\end{proof}
\begin{lem}
\label{lem:Ni:hon:2} Let $q\in S(1, g)$ then there is $C>0$ such that 
\[
\|\op{q}u\| \leq \big(\sup{|q|}+CM^{-1/2}\big)\|u\|.
\]
\end{lem}
\begin{lem}
\label{lem:kihon:fu}Let $m>0$ be an admissible weight for $g$ and  $m\in S(m, g)$.  Then there is $C>0$ such that 
\[
(\op{m}u, u)\geq (1-CM^{-2})\|\op{\sqrt{m}}u\|^2.
\]
If $q\in S(m, g)$ then there is $C>0$ such that
\[
\big|(\op{q}u,u)\big|\leq \big(\sup{\big(|q|/m\big)}+CM^{-1/2}\big)\|\op{\sqrt{m}\,}u\|^2.
\]
\end{lem}
\begin{proof}
First note that $m^{\pm 1/2}$ are admissible weights and $m^{\pm 1/2}\in S(m^{\pm1/2},g)$. Write 
\[
{\tilde q}=(1+k)\#m^{-1/2}\#q\#m^{-1/2}\#(1+{\tilde k})\in S(1,g)
\]
where $m^{1/2}\#(1+k)\#m^{-1/2}=1$ and $m^{-1/2}\#(1+{\tilde k})\#m^{1/2}=1$ such that $
m^{1/2}\#{\tilde q}\#m^{1/2}=q$. 
Since $k$, ${\tilde k}\in S(M^{-1},g)$ one sees that ${\tilde q}=q m^{-1}+r$ with $r\in S(M^{-1},g)$. Thanks to Lemma \ref{lem:Ni:hon:2} we have $
\|\op{qm^{-1}}v\|\leq (\sup{\big(|q|/m\big)}+CM^{-1/2})\|v\|$ 
hence $\big|(\op{q}u,u)\big|$ is bounded by
\begin{align*}
 \big|(\op{qm^{-1}}\op{m^{1/2}}u,\op{m^{1/2}}u)|+ CM^{-1}\|\op{m^{1/2}}u\|^2
\end{align*}
which proves the second assertion. The first assertion follows from the second since $m=m^{1/2}\#m^{1/2}+r$ with $r\in S(M^{-2}m, g)$.
\end{proof}
\begin{lem}
\label{lem:m:1:2}Let $m_i>0$ be two admissible weights for $g$ and assume that $m_i\in S(m_i, g)$ and $m_2\leq C\, m_1$ with  $C>0$. Then there is $C'>0$ such that
\[
\big\|\op{m_2}u\big\|\leq C'\big\|\op{m_1}u\big\|.
\]
\end{lem}
\begin{proof}
Write  ${\tilde m}_2=m_2\#m_1^{-1}\#(1+k)\in S(1, g)$ such that $m_2={\tilde m}_2\#m_1$ with $k\in S(M^{-1},g)$.  Then  $
\|\op{m_2}u\|=\|\op{{\tilde m}_2}\op{m_1}u\|\leq C'\|\op{m_1}u\|$ 
proves the assertion.
\end{proof}
%
%
\subsection{Lower bounds of $\op{\lambda_j}$}

\begin{lem}
\label{lem:lam:1}
There exist $C>0$ and $M_0$ such that 
\begin{align*}
{\mathsf{Re}}\,\big(\op{\lambda_j\#\kappa}u,u\big)\geq 
(1-CM^{-2})\big\|
\op{\kappa^{1/2}\lambda_j^{1/2}}u\big\|^2,\quad M\geq M_0.
\end{align*}
\end{lem}
\begin{proof}Since $\lambda_j\#\kappa=\kappa\lambda_j+r_{j1}+r_{j2}$ 
where $r_{j1}$ is pure imaginary and $r_{j2}\in S(M^{-2}\kappa\lambda_j, g)$ the assertion follows from Lemma \ref{lem:kihon:fu}. 
\end{proof}
\begin{lem}
\label{lem:lam:1b}There exist $c>$ and $M_0$ such that 
\begin{align*}
\big(\op{\lambda_1}u,u\big)\geq c\,\big\|
\op{\lambda_1^{1/2}}u\big\|^2
+cM^2\big\|\lr{D}^{-1}u\big\|^2,\quad M\geq M_0.
\end{align*}
\end{lem}
\begin{proof} From Propositions \ref{pro:ext:matome} and \ref{pro:Skon} it follows that $\lambda_1\geq c'\,M\sigma \lr{\xi}^{-1}$ with some $c'>0$. 
Write 
\[
{\lambda}_1-c\, M\sigma\lr{\xi}^{-1}=\lambda_1/2+(\lambda_1/2-c\, M\sigma\lr{\xi}^{-1})
\]
where $c>0$ is chosen so that ${\tilde \lambda}_1=\lambda_1/2-c\, M\sigma\lr{\xi}^{-1}\geq c_1M\sigma\lr{\xi}^{-1}$ with $c_1>0$. Note that ${\tilde \lambda}_1\in {\mathcal C}(\sigma^2)$ since $M\sigma\lr{\xi}^{-1}\in {\mathcal C}(\sigma^2)$. Thanks to Lemmas \ref{lem:lam:ipan} and \ref{lem:lam:Sg} it follows that ${\tilde \lambda}_1\in S({\tilde \lambda}_1,g)$ and ${\tilde \lambda}_1$ is  an admissible weight for $g$. Thus we have $
(\op{{\tilde \lambda}_1}u,u)\geq (1-CM^{-2})\|\op{{\tilde \lambda}_1^{1/2}}u\|^2\geq 0$  if $M\geq \sqrt{C}$ by Lemma \ref{lem:kihon:fu}. Since $M^2\lr{\xi}^{-2}\leq M\sigma\lr{\xi}^{-1}$  it follows from Lemma \ref{lem:m:1:2} that
\[
M\|\lr{D}^{-1}u\|^2\leq C\|\op{\sigma^{1/2}\lr{\xi}^{-1/2}}u\|^2.
\]
Therefore the proof follows from Lemma \ref{lem:kihon:fu}.
\end{proof}
Similar arguments prove the following lemma.
\begin{lem}
\label{lem:lam:2b}There exist $c>0$ and $M_0$ such that 
\begin{gather*}
\big(\op{\lambda_2}u,u\big)\geq c\,\big\|
\op{\lambda_2^{1/2}}u\big\|^2
+c\, M\big\|\lr{D}^{-1/2}u\big\|^2,\quad M\geq M_0,\\
\big(\op{\lambda_3}u,u\big)\geq c\,\big\|
u\big\|^2, \quad M\geq M_0.
\end{gather*}
\end{lem}
Summarize what we have proved in
\begin{prop}
\label{pro:Lam:sita}There exist $c>0$, $C>0$ and $M_0$ such that
\begin{gather*}
{\mathsf{Re}}\,(\op{\varLambda\#\kappa}W,W)\geq (1-CM^{-2})\|\op{\kappa^{1/2}\varLambda^{1/2}}W\|^2,\\
{\mathsf{Re}}\,(\op{\varLambda}W,W)\geq c\,(\|\op{\varLambda^{1/2}}W\|^2+\|\op{{\mathcal D}}
W\|^2)
\end{gather*}
for $M\geq M_0$ where ${\mathcal D}={\rm diag}\big(M\lr{\xi}^{-1}, M^{1/2}\lr{\xi}^{-1/2}, 1\big)$.
\end{prop}
%

\section{System with diagonal symmetrizer}
\label{sec:henkan:jiko}

Diagonalizing the B\'ezout matrix introduced in Section \ref{sec:Bezout} we reduce 
the system \eqref{eq:redE} to a system with a diagonal symmetrizer.
\begin{lem}
\label{lem:CtoSg}
Let $p\in {\mathcal C}(\sigma^k)$ then $\dif_x^{\al}\dif_{\xi}^{\be}p\in S(\sigma^{k-|\al+\be|/2}\lr{\xi}^{-|\be|},g)$.  
\end{lem}
\begin{proof} The proof is clear from
\begin{gather*}
\big|\dif_x^{\al'}\dif_{\xi}^{\be'}(\dif_x^{\al}\dif_{\xi}^{\be}p)\big|\precsim \sigma^{k-|\al'+\be'+\al+\be|/2}\lr{\xi}^{-|\be'+\be|}\\
\precsim \sigma^{k-|\al+\be|/2}\lr{\xi}^{-|\be|}\sigma^{-|\al'+\be'|/2}\lr{\xi}^{-|\al'+\be'|/2}\lr{\xi}^{(|\al'|-|\be'|)/2}
\end{gather*}
for $\sigma\geq \rho\geq M\lr{\xi}^{-1}$.
\end{proof}
\begin{lem}
\label{lem:CtoC}Let $p\in {\mathcal C}(\sigma^k)$ and $q\in {\mathcal C}(\sigma^{\ell})$. Then
\[
p\#p-p^2\in S(\sigma^{2k-2}\lr{\xi}^{-2},g),\quad p\#q-pq\in S(\sigma^{k+\ell-1}\lr{\xi}^{-1},g).
\]
\end{lem}
\begin{proof}The assertions  follows from Lemma \ref{lem:CtoSg} and the Weyl calculus of pseudodifferential operators.
\end{proof}
In what follows to simplify notations sometimes we abbreviate $S(m, g)$ to $S(m)$ where $m$ is admissible for $g$. Since $a\in {\mathcal C}(\sigma)$, $b\in {\mathcal C}(\sigma^{3/2})$ one sees $A\#[\xi]=A(t, x,\xi)[\xi]+R$ with $R$ whose first row is $(0, S(\sigma^{1/2}), S(\sigma))$ for $\dif_{\xi}^{\be}[\xi]\in S(1, g)$ by  \eqref{eq:xi:kakudai}. Moving $R$ to $B$ we denote $L=D_t-\op{{\tilde A}}-\op{B}$ where
\begin{equation}
\label{eq:B:teigi}
{\tilde A}=\begin{bmatrix}0&a&b\\
1&0 &0\\
0&1&0
\end{bmatrix}[\xi],\quad B=\begin{bmatrix}b_1&b_2+d_M+S(\sigma^{1/2})&b_3+S(\sigma)\\
0&0 &0\\
0&0&0
\end{bmatrix}
\end{equation}
and transform $L$ to another system using $T$ introduced in Section \ref{sec:defT}. Note that $T^{-1}\#T=I-R$ with $R\in S(M^{-1},g)$. Thanks to Lemma \ref{lem:gyaku} there is $K\in S(M^{-1},g)$ such that $(I-R)\#(I+K)=I=(I+K)\#(I-R)$ and hence
\[
T^{-1}\#T\#(I+K)=I,\quad (I+K)\#T^{-1}\#T=I,\quad T\#(I+K)\#T^{-1}=I.
\]
Therefore one can write
\begin{equation}
\label{eq:L:tiL}
L\,\op{T}=\op{T}\,{\tilde L}
\end{equation}
where $
{\tilde L}=D_t-\op{(I+K)\#T^{-1}\#({\tilde A}+B)\#T}+\op{(I+K)\#T^{-1}\#(D_tT)}$.
\begin{lem}
\label{lem:K:seimitu}Notations being as above. Then $K\in S(M^{-1}\lr{\xi}^{-1},g)$.
\end{lem}
\begin{proof}Write $T=(t_{i j})$ then $T^{-1}\#T=(\sum_{k=1}^3t_{k i}\#t_{kj})$. Denote
\[
\sum_{k=1}^3t_{ki}\#t_{k j}=\delta_{ij}+r_{ij}.
\]
Taking Lemma \ref{lem:T:seimitu} into account, we see $r_{i i}\in S(\sigma^{-1}\lr{\xi}^{-2},g)\subset S(M^{-1}\lr{\xi}^{-1},g)$ and $r_{i j}\in S(\sigma^{1/2}\lr{\xi}^{-1},g)\subset S(M^{-2}\lr{\xi}^{-1},g)$ for $i\neq j$ thanks to  Lemma \ref{lem:CtoC} hence  $R\in S(M^{-1}\lr{\xi}^{-1},g)$. Since $K\in S(M^{-1},g)$ satisfies $K=R+R\#K$ we conclude the assertion.
\end{proof}
Therefore $K\#T^{-1}\#({\tilde A}+B)\#T\in S(M^{-1}, g)$ is clear.   Hence 
\[
{\tilde L}=D_t-\op{T^{-1}\#({\tilde A}+B)\#T-T^{-1}\#(D_tT)}+\op{S(M^{-1},g)}.
\]
In view of Lemmas \ref{lem:T:seimitu} and \ref{lem:diftT:seimitu}  it follows from Lemma \ref{lem:CtoC} that
\begin{equation}
\label{eq:TdtT}
\begin{split}
&T^{-1}\#(\dif_tT)=T^{-1}\dif_t T\\
&+\begin{bmatrix}S(\sigma^{-1}\lr{\xi}^{-1})&S(\sigma^{-1/2}\lr{\xi}^{-1})&S(\lr{\xi}^{-1})\\
S(\sigma^{-1/2}\lr{\xi}^{-1})&S(\sigma^{-1}\lr{\xi}^{-1})&S(\sigma^{-1/2}\lr{\xi}^{-1})\\
S(\lr{\xi}^{-1})&S(\sigma^{-1/2}\lr{\xi}^{-1})&S(\lr{\xi}^{-1})\\
\end{bmatrix}
\end{split}
\end{equation}
hence $T^{-1}\#(\dif_tT)=T^{-1}\dif_tT+S(M^{-1},g)$ because $\sigma\geq M\lr{\xi}^{-1}$.

Turn to study $T^{-1}\#{\tilde A}\#T$. Noting that $\dif_x^{\al}\dif_{\xi}^{\be}a\in S(\sigma^{1/2}\lr{\xi}^{-|\be|},g)$ and $\dif_x^{\al}\dif_{\xi}^{\be}b\in S(\sigma\lr{\xi}^{-|\be|},g)$ for $|\al+\be|=1$ and $\dif_{\xi}^{\be}[\xi]\in S(1,g)$, $|\be|=1$ we have
\[
T^{-1}\#{\tilde A}=T^{-1}{\tilde A}+R,\quad R=
\begin{bmatrix}S(1)&S(M^{-2})&S(M^{-6})\\
S(M^{-2})&S(1)&S(M^{-8})\\
S(M^{-8})&S(M^{-2})&S(M^{-6})
\end{bmatrix}.
\]
Therefore $T^{-1}\#{\tilde A}\#T=(T^{-1}{\tilde A})\#T+R_1$ with
\[
R_1=R\#T=\begin{bmatrix}S(M^{-4})&S(M^{-2})&S(1)\\
S(M^{-2})&S(1)&S(M^{-2})\\
S(M^{-4})&S(M^{-2})&S(M^{-8})
\end{bmatrix}.
\]
Note that
\[
T^{-1}{\tilde A}=\begin{bmatrix}C(\sigma^{1/2})&1+C(\sigma)&C(\sigma^{5/2})\\
-1+C(\sigma)&C(\sigma^{1/2})&C(\sigma^3)\\
C(\sigma^{5/2})&C(\sigma)&C(\sigma^{3/2})
\end{bmatrix}[\xi]
\]
and hence
\[
\lr{\xi}^{|\be|}\dif_x^{\al}\dif_{\xi}^{\be}\big(T^{-1}{\tilde A}\big)=\begin{bmatrix}S(1)&S(1)&S(M^{-8})\\
S(1)&S(1)&S(M^{-10})\\
S(M^{-8})&S(M^{-2})&S(M^{-4})
\end{bmatrix},\quad |\al+\be|=1.
\]
Then thanks to \eqref{eq:difYT:seimitu} one sees
\[
(T^{-1}{\tilde A})\#T=T^{-1}{\tilde A}T+R_2,\quad R_2=\begin{bmatrix}S(1)&S(M^{-2})&S(M^{-2})\\
S(1)&S(M^{-2})&S(M^{-2})\\
S(M^{-2})&S(M^{-4})&S(M^{-6})
\end{bmatrix}.
\]
Thus we obtain $T^{-1}\#{\tilde A}\#T=T^{-1}{\tilde A}T+R_1+R_2$ 
where 
\[
R_1+R_2=\begin{bmatrix}S(1)&S(M^{-2})&S(M^{-2})\\
S(1)&S(M^{-2})&S(M^{-2})\\
S(M^{-2})&S(M^{-4})&S(M^{-6})
\end{bmatrix}.\]
Recall $B$ given by \eqref{eq:B:teigi}. Since $d_M\in S(M, g)$ one sees by Lemma \ref{lem:T:seimitu} that
\[
T^{-1}\#B=\begin{bmatrix}S(\sigma)&S(M\sigma)&S(\sigma)\\
S(\sigma^{3/2})&S(M\sigma^{3/2})&S(\sigma^{3/2})\\
b_1+S(\sigma)&b_2+d_M+S(\sigma^{1/2})&b_3+S(\sigma)
\end{bmatrix}
\]
because $\sigma\leq CM^{-4}$. Thus we conclude that $T^{-1}\#B\#T$ is written
\begin{equation}
\label{eq:T:B:T}
\begin{bmatrix}S(\sigma)&S(M\sigma)&S(\sigma)\\
S(\sigma^{3/2})&S(M\sigma^{3/2})&S(\sigma^{3/2})\\
b_3+S(M\sigma^{1/2})&-b_2-d_M+S(\sigma^{1/2})&b_1+S(\sigma)
\end{bmatrix}.
\end{equation}
Noting  $b_3(t, x,\xi)- {\bar b}_3
\in S(M^{-2}, g)$ we can  summarize what we have proved in
\begin{prop}
\label{pro:T:henkan} One can write $
L\cdot \op{T}=\op{T}\cdot {\tilde L}$ 
where
\begin{align*}
&{\tilde L}=D_t-\op{{\mathcal A}+{\mathcal B}},\quad {\mathcal A}=(T^{-1}AT)[\xi]\,,\quad {\mathcal B}={\mathcal B}_1-T^{-1}D_tT\\[3pt]
& {\mathcal B}_1=T^{-1}\#B\#T=\begin{bmatrix}S(1)&S(1)&S(1)\\
S(1)&S(1)&S(1)\\
{\bar b}_3+S(M^{-1})&-2M{\bar e}+S(M^{-1})&S(1)
\end{bmatrix}.
\end{align*}
\end{prop}
%

\section{Weighted energy estimates}
\label{sec:mitibiku:ene}

\subsection{Energy form}

Let $w=t\phi (t, x,\xi)$ and consider the energy with scalar weight $\op{w^{-n}}$; 
\[
{\mathcal E}(V)=e^{-\theta t}\big(\op{\varLambda}\op{w^{-n}}V,\op{w^{-n}}V\big)
\]
where $\theta>0$ is a large positive parameter and $n$ is fixed such that
\begin{equation}
\label{eq:n:atai}
n>(4\sqrt{2}\,)\,|3\,{\bar b}_3+i{\bar e}|/{\bar e}+C^*+2+8(1+3\sqrt{2})
\end{equation}
where $C^*$ is given by   \eqref{eq:pert:Dis}. It is clear from \eqref{eq:P:sub} that \eqref{eq:n:atai} follows from
\begin{equation}
\label{eq:n:P:sub}
n>12\sqrt{2}\,\,\frac{|P_{sub}(0, 0, 0, {\bar\xi})|}{{\bar e}}+{\bar C^*},\quad {\bar C^*}=C^*+10+32\sqrt{2}.
\end{equation}
Note that $\dif_t \phi=\omega^{-1}\phi$ and hence
\[
\dif_t w^{-n}=-n\Big(\frac{1}{t}+\frac{1}{\omega}\Big)w^{-n}=-n\kappa\,w^{-n}.
\]
Recall that $V$ satisfies
\begin{equation}
\label{eq:eq:V}
\dif_tV=\op{i{\mathcal A}+i{\mathcal B}}V+F,\quad {\mathcal B}={\mathcal B}_1-T^{-1}D_tT.
\end{equation}
Noting that $\Lambda$ is real and diagonal hence $\op{\varLambda}^*=\op{\varLambda}$ one has 
\begin{equation}
\label{eq:en:id}
\begin{split}
\frac{d}{dt}{\mathcal E}=-\theta e^{-\theta t}\big(\op{\varLambda}\op{w^{-n}}V,\op{w^{-n}}V\big)\\
-2n{\mathsf{Re}}\,e^{-\theta t}\big(\op{\varLambda}\op{\kappa w^{-n}}V,\op{w^{-n}}V)\\
+e^{-\theta t}\big(\op{\dif_t\Lambda}\op{w^{-n}}V,\op{w^{-n}}V\big)\\
+2{\mathsf{Re}}\,e^{-\theta t}\big(\op{\varLambda}\op{w^{-n}}(\op{i{\mathcal A}+i{\mathcal B}}V+F),\op{w^{-n}}V).
\end{split}
\end{equation}
Consider $\op{\phi^{-n}}\op{\varLambda}\op{\kappa \phi^{-n}}=
\op{\phi^{-n}\#\varLambda\#(\kappa \phi^{-n})}$. Since $\kappa$ and $\phi^{-n}$ are admissible weights for $g$ one has $
\kappa\#\phi^{-n}=\kappa\phi^{-n}-r$ with $r\in S(M^{-1}\kappa\phi^{-n},g)$. 
Let ${\tilde r}=r\#\phi^n\#(1+k)\in S(M^{-1}\kappa, g)$ such that $r={\tilde r}\#\phi^{-n}$ and hence $\kappa \phi^{-n}=(\kappa+{\tilde r})\#\phi^{-n}$. 
Thus we have
\begin{align*}
{\mathsf{Re}}\,\big(\op{\varLambda}\op{\kappa w^{-n}}V,\op{w^{-n}}V\big)\geq {\mathsf{Re}}\,\big(\op{\varLambda\#\kappa}\op{w^{-n}}V,\op{w^{-n}}V\big)\\
-\big|\big(\op{\varLambda\#{\tilde r}}\op{w^{-n}}V,\op{w^{-n}}V)\big|.
\end{align*}
Since $\lambda_j\#{\tilde r}\in S(M^{-1}\kappa\lambda_j,g)$  thanks to Lemma \ref{lem:kihon:fu} the second term on the right-hand side is bounded by $
CM^{-1}\|\op{\kappa^{1/2}\varLambda^{1/2}}\op{w^{-n}}V\|$. 
Applying Proposition \ref{pro:Lam:sita},  denoting $W_j=\op{w^{-n}}V_j$, one can conclude that
\begin{gather*}
{\mathsf{Re}}\,\big(\op{\varLambda}\op{\kappa w^{-n}}V,\op{w^{-n}}V\big)\geq 
(1-CM^{-1})\|\op{\kappa^{1/2}\varLambda^{1/2}}W\|^2,\\
{\mathsf{Re}}\,\big(\op{\varLambda}\op{ w^{-n}}V,\op{w^{-n}}V\big)
\geq c\big(\|\op{\varLambda^{1/2}}W\|^2+
\|\op{{\mathcal D}}W\|^2\big)
\end{gather*}
for $M\geq M_0$. 
\begin{definition}\rm
\label{dfn:calE:1:2}To simplify notations we denote
\begin{gather*}
{\mathcal E}_1(V)=\|\op{\kappa^{1/2}\varLambda^{1/2}}\op{w^{-n}}V\|^2
=t^{-2n}\|\op{\kappa^{1/2}\varLambda^{1/2}}\op{\phi^{-n}}V\|^2,\\
{\mathcal E}_2(V)=\|\op{\varLambda^{1/2}}\op{w^{-n}}V\|^2
=t^{-2n}\|\op{\varLambda^{1/2}}\op{\phi^{-n}}V\|^2.
\end{gather*}
\end{definition}
Now we summarize
\begin{lem}
\label{lem:calE:j}One can find $C>0$, $c>0$ and $M_0$ such that
\begin{align*}
n\,{\mathsf{Re}}\,\big(\op{\varLambda}\op{\kappa w^{-n}}V,\op{w^{-n}}V\big)
+\theta\,{\mathsf{Re}}\,\big(\op{\varLambda}\op{ w^{-n}}V,\op{w^{-n}}V\big)\\
\geq n(1-CM^{-1}){\mathcal E}_1(V)+c\,\theta {\mathcal E}_2(V),\quad M\geq M_0.
\end{align*}
\end{lem}
%

\subsection{Term  $(\op{\varLambda}\op{w^{-n}}\op{{\mathcal B}}V, \op{w^{-n}}V)$}
\label{subsec:S:1:g}

Recall that $\lambda_i\in S(\lambda_i,g)$ and $\lambda_1\leq C\sigma \lambda_2\leq C\sigma^2 \lambda_3$ with some $C>0$. We show
\begin{lem}
\label{lem:LamB:hyoka}Let $W=\op{\phi^{-n}}V$. Then we have
\begin{gather*}
\big|(\op{\lambda_i}\op{b}W_j,W_i)\big|\leq CM^{-2}{\mathcal E}_1(V)+CM^2{\mathcal E}_2(V),\;\;b\in S(\sigma^{-1/2}, g),\;j\geq i,\\
|(\op{\lambda_3}\op{b}W_2,W_3)|
\leq CM^{-2}{\mathcal E}_1(V)+CM^{2+2l}{\mathcal E}_2(V),\;\;b\in S(M^l, g),\\
|(\op{\lambda_3}\op{b}W_1,W_3)|
\leq \big(\sqrt{3}\,{\bar \varep}\|\op{b}\|+CM^{-1/2}\big)
{\mathcal E}_1(V),\;\;b\in S(1, g),\\
|(\op{\lambda_2}\op{b}W_1,W_2)|
\leq ({\bar \varep}\|\op{\lambda_2^{1/2}b}\|+CM^{-1/2}\big){\mathcal E}_1(V),\;\;b\in S(\sigma^{-1/2}, g).
\end{gather*}
\end{lem}
\begin{proof}
 Let $b\in S(\sigma^{-1/2}, g)$. Noting  $\lambda_i^{1/2}\lambda_j^{-1/2}\in S(\sigma^{(j-i)/2}, g)$ one can write
 \begin{align*}
 r=(1+k)\#(\kappa^{-1/2}\lambda_i^{-1/2})\#(\lambda_i\#b)\#\lambda_j^{-1/2}\#(1+{\tilde k})\in S(\sigma^{(j-i)/2}, g),\quad j\geq i
\end{align*}
for $\sigma\kappa\geq 1$, such that $(\kappa^{1/2}\lambda_i^{1/2})\#r\#\lambda_j^{1/2}=\lambda_i\#b$. Then we have
\[
|(\op{\lambda_i}\op{b}W_j,W_i)|\leq M^{-2}\|\op{\kappa^{1/2}\varLambda^{1/2}}W\|^2+CM^2\|\op{\varLambda^{1/2}}W\|^2
\]
for $j\geq i$. Let $b\in S(M^l, g)$ and denote 
\[
r=(1+k)\#(\kappa^{-1/2}\lambda_2^{-1/2})\#(\lambda_3\#b)\#\lambda_3^{-1/2}\#(1+{\tilde k})
\]
such that $(\kappa^{1/2}\lambda_2^{1/2})\#r\#\lambda_3^{1/2}=\lambda_3\#b$. Since $r\in S(\kappa^{-1/2}\lambda_3^{1/2}\lambda_2^{-1/2}, g)\subset S(1, g)$ in view of Lemma \ref{lem:daiji} then $\big|(\op{\lambda_3}\op{b}W_2,W_3)\big|$ is bounded by
\begin{align*}
 CM^{-2}\|\op{\kappa^{1/2}\varLambda^{1/2}}W\|^2+CM^{2+2l}\|\op{\varLambda^{1/2}}W\|^2.
\end{align*}
We check $(\op{\lambda_3}\op{b}W_1,W_3)$ for $b\in S(1, g)$.  Noting $\kappa^{-1}\lambda_1^{-1/2}\in S(1, g)$ by Lemma \ref{lem:daiji} write
\[
r=(1+k)\#(\kappa^{-1/2}\lambda_1^{-1/2})\#(\lambda_3\#b)\#(\kappa^{-1/2}\lambda_3^{-1/2})\#(1+{\tilde k})\in S(1, g)
\]
such that $(\kappa^{1/2}\lambda_1^{1/2})\#r\#(\kappa^{1/2}\lambda_3^{1/2})=\lambda_3\#b$.
Since $k,\; {\tilde k}\in S(M^{-1}, g)$ it is easy to see that $
r=(\lambda_3^{1/2}\lambda_1^{-1/2}\kappa^{-1})\#b+{\tilde r}$ 
with ${\tilde r}\in S(M^{-1/2}, g)$. By Proposition \ref{pro:Skon} and Lemma \ref{lem:daiji} one sees that $
\big|\lambda_3^{1/2}\lambda_1^{-1/2}\kappa^{-1}\big|\leq \sqrt{3}\,{\bar \varep}+CM^{-4}$ 
hence 
\begin{align*}
|(\op{\lambda_3}\op{b}W_1,W_3)|=|(\op{r}\op{\kappa^{1/2}\lambda_1^{1/2}}W_1, \op{\kappa^{1/2}\lambda_3^{1/2}}W_3)|\\
\leq \big(\sqrt{3}\,{\bar \varep}\|\op{b}\|+CM^{-1/2}\big)\|\op{\kappa^{1/2}\varLambda^{1/2}}W\|^2.
\end{align*}
Next consider $(\op{\lambda_2}\op{b}W_1,W_2)$ for $b\in S(\sigma^{-1/2}, g)=S(\lambda_2^{-1/2}, g)$. Denote
\[
r=(1+k)\#(\kappa^{-1/2}\lambda_2^{-1/2})\#(\lambda_2\#b)\#(\lambda_1^{-1/2}\kappa^{-1/2})\#(1+{\tilde k})\in S(1, g)
\]
such that $(\kappa^{1/2}\lambda_2^{1/2})\#r\#(\lambda_1^{1/2}\kappa^{1/2})=\lambda_2\#b$. Write $
r=(\kappa^{-1}\lambda_1^{-1/2})\#(\lambda_2^{1/2}b)+{\tilde r}$ 
with ${\tilde r}\in S(M^{-1}, g)$. Thus repeating the same arguments as above one conclude the last assertion.
\end{proof}
In particular this lemma implies
\begin{cor}
\label{cor:LamB:1}Let $B=(b_{ij})\in S(1,g)$. Then with $W=\op{w^{-n}}V$ 
\begin{align*}
\big|(\op{\varLambda}\op{B}W,W)\big|\leq (\sqrt{3}\,{\bar \varep}\|\op{b_{31}}\|+CM^{-1/2}){\mathcal E}_1(V)
+C{\mathcal E}_2(V).
\end{align*}
\end{cor}
From Proposition \ref{pro:T:henkan} it results
$\phi^{-n}\#{\mathcal B}_1-{\mathcal B}_1\#\phi^{-n}\in S(M^{-1}\phi^{-n}, g)$ then one concludes by Corollary \ref{cor:LamB:1} that
\begin{equation}
\label{eq:calB:comm}
\big|(\op{\varLambda}[\op{w^{-n}}, \op{{\mathcal B}_1}]V, W)\big|\leq CM^{-1}{\mathcal E}_1(V)+C{\mathcal E}_2(V)
\end{equation}
where $W=\op{w^{-n}}V$ again. Write $T^{-1}\dif_tT=({\tilde t}_{i j})$ and recall \eqref{eq:dT:-1:T} and note that ${\tilde t}_{1 2}=-{\tilde t}_{2 1} \in {\mathcal C}(\sigma^{-1/2})$ and ${\tilde t}_{3 1}\in S(1, g)$. Then thanks to Lemma \ref{lem:dif:Phi:seimitu} one sees that $\lambda_j\#(\phi^{-n}\#{\tilde t}_{j1}-{\tilde t}_{j1}\#\phi^{-n})\#\phi^n$ is in
\begin{align*} 
S(\sigma^{1-j/2}\omega^{-1}\rho^{1/2}\lr{\xi}^{-1}, g)
\subset S(M^{-1}\sqrt{\kappa\lambda_1}\,\sqrt{\kappa\lambda_j}, g),\quad j=2, 3
\end{align*}
because $C\lambda_1\geq M\rho\lr{\xi}^{-1}$, $C\lambda_2 \geq M\lr{\xi}^{-1}$ and $\omega^{-1}\leq \kappa$. Therefore repeating similar arguments one concludes
\begin{equation}
\label{eq:TdT:comm}
\big|(\op{\varLambda}[\op{w^{-n}}, \op{T^{-1}\dif_tT}]V, W)\big|\leq CM^{-1}{\mathcal E}_1(V).
\end{equation}
Recalling ${\mathcal B}={\mathcal B}_1-T^{-1}D_tT$ it follows from \eqref{eq:calB:comm} and \eqref{eq:TdT:comm} that
\begin{equation}
\label{eq:BTdT:comm}
\big|(\op{\varLambda}[\op{w^{-n}}, \op{{\mathcal B}}]V, W)\big|\leq CM^{-1}{\mathcal E}_1(V)+C{\mathcal E}_2(V).
\end{equation}
With ${\mathcal B}=(q_{i j})$ we see that $q_{i j}\in S(\sigma^{-1/2}, g)$ for $j\geq i$ and 
\[
q_{21}=i\dif_t(3b/2a_M)+S(1),\;  q_{31}={\bar b}_3+i{\bar e}/3+S(M^{-1}),\;q_{3 2}=-2M {\bar e}+S(M^{-1})
\]
by Proposition \ref{pro:T:henkan}.  Applying Lemma \ref{lem:LamB:hyoka} we have from \eqref{eq:c:bar:1}, recalling Proposition \ref{pro:Skon} and  ${\bar\varep}=\sqrt{3}/{\bar e}{\bar c}^{1/2}$,  that
\begin{equation}
\label{eq:Lam:calB}
\begin{split} 
&\big|(\op{\varLambda}\op{{\mathcal B}}\op{w^{-n}}V,\op{w^{-n}}V)\big|\\
&\leq \big(|3{\bar b}_3+i{\bar e}|/{\bar e}{\bar c}^{1/2}+(6+\sqrt{2})/{\bar c}^{1/2}+CM^{-1/2}\big){\mathcal E}_1(V)+CM^4{\mathcal E}_2(V).
\end{split}
\end{equation}
Combining the estimates \eqref{eq:Lam:calB}  and \eqref{eq:BTdT:comm} we obtain
\begin{lem}
\label{lem:Lam:calB:T}The term $\big|(\op{\varLambda}\op{w^{-n}}\op{{\cal B}}V,\op{w^{-n}}V)\big|$ is bounded by the right-hand side of \eqref{eq:Lam:calB}.
\end{lem}
%

\subsection{Term $(\op{\varLambda}\op{w^{-n}}\op{i{\mathcal A}}V,\op{w^{-n}}V)$}

Note that $\phi^{-n}\#([\xi]r)-([\xi]r)\#\phi^{-n}\in S(\phi^{-n}\sigma^{s-1/2}\omega^{-1}\rho^{1/2}, g)$ for $r\in {\mathcal C}(\sigma^s)$ by Lemma \ref{lem:dif:Phi:seimitu}. Recalling Corollary \ref{cor:TAT:seimitu} and ${\mathcal A}=A^T[\xi]$ then denoting $\phi^{-n}\#{\mathcal A}-{\mathcal A}\#\phi^{-n}=(r_{i j})$ we see that $
r_{i j}\in S(\phi^{-n}\omega^{-1}\rho^{1/2}, g)\subset S(M^{-2}\kappa \phi^{-n}, g)$  for $j\geq i$ 
because $\omega^{-1}\leq \kappa$. Writing $\lambda_i\#r_{i j}=\lambda_i\#{\tilde r}_{i j}\#\phi^{-n}$ with ${\tilde r}_{i j}\in S(M^{-2}\kappa, g)$ one obtains $
\big|(\op{\lambda_i}\op{r_{i j}}V_j, W_i)\big|\leq CM^{-2}\|\op{\kappa^{1/2}\varLambda^{1/2}}W\|^2$ for $ j\geq i$ since $\lambda_i\#{\tilde r}_{i j}\in S(M^{-2}\kappa \lambda_i, g)$. From Lemma \ref{lem:nami:a}   one has ${\tilde a}_{2 1}=\lambda_1\, {\mathcal C}(\sigma^{-1})$ hence thanks to 
Lemmas \ref{lem:dif:Phi:seimitu} and \ref{lem:lam:Sg} 
\begin{align*}
\phi^{-n}\#({\tilde a}_{2 1}[\xi])-\phi^{-n}{\tilde a}_{2 1}[\xi]\in S(\sigma^{-1/2}\lambda_1^{1/2}\omega^{-1}\rho^{1/2}\phi^{-n}, g)\subset 
S(\lambda_1^{1/2}\kappa\phi^{-n}, g)
\end{align*}
for $\omega^{-1}\leq \kappa$ again. Thus  we have
\begin{align*}
\big|(\op{\lambda_2}\op{r_{2  1}}V_1, W_2)\big|
\leq CM^{-2}\|\op{\kappa^{1/2}\varLambda^{1/2}}W\|^2
\end{align*}
since $\lambda_2^{1/2}\leq CM^{-2}$. Similarly from ${\tilde a}_{3 1}=\lambda_1\,{\mathcal C}(\sigma^{1/2})$, ${\tilde a}_{3 2}=\lambda_2\,{\mathcal C}(1)$  and Lemma \ref{lem:lam:Sg} it follows that $
r_{3 j}\in S(\sigma^{2-j} \lambda_j^{1/2}\omega^{-1}\rho^{1/2}\phi^{-n}, g)\subset S(M^{-2}\lambda_j^{1/2}\kappa\phi^{-n}, g)$ for $j=1, 2$. 
Here we have used $
\dif_x^{\al}\dif_{\xi}^{\be}\lambda_2\in 
S(\lambda_2^{1/2}\lr{\xi}^{-|\be|}, g)$ for $|\al+\be|=1$ which follows from $\lambda_2\in {\mathcal C}(\sigma)$ easily. Then one obtains
\begin{align*}
\big|(\op{\lambda_3}\op{r_{3  j}}V_j, W_3)\big|
\leq CM^{-2}\|\op{\kappa^{1/2}\varLambda^{1/2}}W\|^2,\quad j=1, 2.
\end{align*}
Therefore $(\op{\varLambda}\op{w^{-n}}\op{{\mathcal A}}V,\op{w^{-n}}V)-(\op{\varLambda}\op{{\mathcal A}}W,W)$ is bounded by constant times $M^{-2}{\mathcal E}_1(V)$. 

Next study $\varLambda\#{\mathcal A}-\varLambda{\mathcal A}=(q_{i j})$. From Corollary \ref{cor:TAT:seimitu} and  Lemma \ref{lem:lam:Sg} we have
\begin{gather*}
\lambda_1\#({\tilde a}_{1 j}[\xi])-\lambda_1{\tilde a}_{1 j}[\xi]\in S(\sigma^{1/2}\lambda_1^{1/2}, g), \;
\lambda_2\#({\tilde a}_{2 j}[\xi])-\lambda_2{\tilde a}_{2 j}[\xi]\in S(\lambda_2^{1/2}, g)
\end{gather*}
for $j\geq 1$ and $j\geq 2$ respectively. Then noting $C\lambda_1^{1/2}\kappa\geq 1$ and $C\lambda_2\kappa\geq 1$
\[
\big|(\op{q_{i j}}W_j, W_i)\big|\leq CM^{-2}\|\op{\kappa^{1/2}\varLambda^{1/2}}W\|^2
+CM^2\|\op{\varLambda^{1/2}}W\|^2,\;\;j\geq i.
\]
Repeating similar arguments one has $
\lambda_i\#({\tilde a}_{i j}[\xi])-\lambda_i{\tilde a}_{i j}[\xi]\in  S(M^{-2}\kappa\lambda_i^{1/2}\lambda_j^{1/2}, g)$ and hence 
$|(\op{q_{i j}}W_i,W_j)|
\leq CM^{-2}\|\op{\varLambda^{1/2}\kappa^{1/2}}W\|^2$ for $i>j$. 
 Thus we conclude that 
\begin{equation}
\label{eq:LaA:ALa}
\begin{split}
\big|(\op{\varLambda}\op{w^{-n}}\op{{\mathcal A}}V,\op{w^{-n}}V)-(\op{\varLambda{\mathcal A}}W,W)\big|\\
\leq CM^{-2}{\mathcal E}_1(V)+CM^2{\mathcal E}_2(V).
\end{split}
\end{equation}
Since $\varLambda{\mathcal A}={\mathcal A}^*\varLambda$ 
we have
\begin{lem}
\label{lem:Lam:calA}One can find $C>0$ such that
\[
\big|{\mathsf{Re}}\,\big(\op{\varLambda}\op{w^{-n}}\op{i{\mathcal A}}V,\op{w^{-n}}V)\big|\leq CM^{-2}{\mathcal E}_1(V)+CM^2{\mathcal E}_2(V).
\]
\end{lem}
%

\subsection{Term $(\op{\dif_t\varLambda}\op{w^{-n}}V, \op{w^{-n}}V)$}

Start with 
\begin{lem}
\label{lem:dif:t:Lam}We have $\dif_t\lambda_j\in S(\kappa\lambda_j, g)$, $j=1,2$.
\end{lem}
\begin{proof}Note that Lemma \ref{lem:pert:Dis} with $\ep=\sqrt{2}M\lr{\xi}^{-1}$ implies 
\[
|\dif_t \Delta_M|\leq C^*\big(1/t+1/\omega\big)\Delta_M= C^*\kappa \Delta_M.
\]
Recalling $\dif_t \lambda_1=-\dif_t q(\lambda_1)/\dif_{\lambda}q(\lambda_1)$ it follows from \eqref{eq:dmod} and \eqref{eq:difq:tt} that
\[
|\dif_t\lambda_1|\leq (1+CM^{-2})\big(|\dif_ta_M/a_M|\lambda_1+|\dif_t\Delta_M|/6a_M\big).
\]
Since $(1+CM^{-2})\lambda_1\geq \Delta_M/6a_M$ by Proposition \ref{pro:Skon} and $1/a_M\leq \kappa/e$ by Lemma \ref{lem:daiji} one concludes 
\begin{equation}
\label{eq:dlam:1:lam}
|\dif_t\lambda_1|\leq (1+CM^{-2})(C^*+1)\kappa\lambda_1.
\end{equation}
Note that $
|\dif_x^{\al}\dif_{\xi}^{\be}\dif_t\lambda_1|\leq C\sigma^{1-|\al+\be|/2}\lr{\xi}^{-|\be|}\leq C\sigma^{1/2}\lr{\xi}^{-1/2}\lr{\xi}^{(|\al|-|\be|)/2}$ for $\dif_t\lambda_1\in {\mathcal C}(\sigma)$. 
From Lemma \ref{lem:daiji} and $C\lambda_1\geq M\sigma\lr{\xi}^{-1}$ it follows that
\[
\kappa\lambda_1\geq \kappa\sqrt{\lambda_1}M^{1/2}\sigma^{1/2}\lr{\xi}^{-1/2}/C\geq M^{1/2}\sigma^{1/2}\lr{\xi}^{-1/2}/C'
\]
which proves $|\dif_x^{\al}\dif_{\xi}^{\be}\dif_t\lambda_1|\leq CM^{-1/2}\kappa\lambda_1\lr{\xi}^{(|\al|-|\be|)/2}$ for $|\al+\be|=1$. For $|\al+\be|\geq 2$ it follows that
\begin{gather*}
|\dif_x^{\al}\dif_{\xi}^{\be}\dif_t\lambda_1|\precsim \sigma^{-(|\al+\be|-2)/2}\lr{\xi}^{-|\be|}
\precsim M\lr{\xi}^{-1}M^{-|\al+\be|/2}\lr{\xi}^{(|\al|-|\be|)/2}\\
\leq \sigma^{-1}M \sigma \lr{\xi}^{-1}M^{-|\al+\be|/2}\lr{\xi}^{(|\al|-|\be|)/2}\leq C\kappa\lambda_1M^{-|\al+\be|/2}\lr{\xi}^{(|\al|-|\be|)/2}
\end{gather*}
because $\kappa\sigma\geq 1$. Thus $\dif_t \lambda_1\in S(\kappa \lambda_1, g)$. On the other hand $\dif_t\lambda_j\in S(\kappa\lambda_j, g)$, $j=2, 3$ is clear since $\dif_t\lambda_j\in {\mathcal C}(1)\subset S(1, g)\subset S(\kappa\lambda_2, g)$ for $C\lambda_2\kappa\geq 1$. This completes the proof.
\end{proof}
Note that from \eqref{eq:dmod}, \eqref{eq:difq:tt} and $|\dif_t\Delta_M|\precsim a_M^2$ we see that
\begin{equation}
\label{eq:dlam:2:lam}
\big|\dif_t \lambda_2\big|\leq(1+CM^{-2}) |\dif_t a_M|\lambda_2/a_M+Ca_M\leq \big(1+CM^{-2}\big)\kappa\lambda_2
\end{equation}
for $C\kappa\lambda_2\geq 1$. Since $\big|(\op{\dif_t \lambda_3}W_3,W_3)\big|\leq C\|\op{\lambda_3}W_3\|^2$ is clear  applying Lemma \ref{lem:kihon:fu} one obtains from \eqref{eq:dlam:1:lam} and \eqref{eq:dlam:2:lam} that
\begin{lem}
\label{lem:dif:t:Lam:bis}We have
\begin{align*}
\big|(\op{\dif_t\varLambda}\op{w^{-n}}V, \op{w^{-n}}V)\big|
\leq (C^*+2+CM^{-1/2}){\mathcal E}_1(V)
+C{\mathcal E}_2(V).
\end{align*}
\end{lem}
%

\subsection{Conclusion}

In what follows we denote $\|u\|_s=\|\lr{D}^su\|$ and by $H^s=H^s(\R^d)$ the set of tempered distribution $u$ on $\R^d$ such that $\|u\|_s<+\infty$.
\begin{definition}
\label{dfn:calH}\rm Denote by ${\mathcal H}_{-n,s}(0,\delta)$ the set of $f$ such that
\[
t^{-n}\lr{D}^{s}f(t,\cdot)\in L^2((0,\delta)\times \R^d).
\]
\end{definition}
Consider the term ${\mathsf{Re}}\,\big(\op{\varLambda}\op{w^{-n}}F,\op{w^{-n}}V)$ where $F={^t}(F_1,F_2,F_3)$. Write $\varLambda=(\kappa^{1/2}\varLambda^{1/2})\#R\#(\kappa^{-1/2}\varLambda^{1/2})$ with $R\in S(1, g)$.  
Because of the choice of $n$ it follows from \eqref{eq:en:id} and Lemmas \ref{lem:calE:j}, \ref{lem:Lam:calB:T}, \ref{lem:Lam:calA}, \ref{lem:dif:t:Lam:bis} one can find $c_i>0$ and $M_0$, $\gamma_0$, $\theta_0$ such that
\begin{equation}
\label{eq:dif:ene:fu}
\begin{split}
\frac{d}{dt}{\mathcal E}(V)\leq -c_1e^{-\theta t}{\mathcal E}_1(V)-c_2\theta e^{-\theta t}{\mathcal E}_2(V)\\
+\big|{\mathsf{Re}}\,\big(\op{\varLambda}\op{w^{-n}}F,\op{w^{-n}}V)\big|
\end{split}
\end{equation}
for $M\geq M_0$, $\gamma\geq \gamma_0$ and $\theta\geq \theta_0$.  Thanks to Lemma \ref{lem:psi:t:sei} one has $
\kappa^{-1/2}\phi^{-n}\lambda_j^{1/2}\in S(M^{-n}\sqrt{t}\,\lr{\xi}^{n}, g)$ 
then we see easily
\[
\big|{\mathsf{Re}}\,\big(\op{\varLambda}\op{w^{-n}}F,\op{w^{-n}}V)\big|\leq CM^{-1}{\mathcal E}_1(V)+CM^{-2n+1}t^{2n+1}\|F\|_n^2.
\]
Since  $M\lr{\xi}^{-1}/C\leq \phi\leq CM^{-4}$ and $t^{-1/2}\leq \kappa^{1/2}\leq t^{-1/2}+\omega^{-1/2}\leq t^{-1/2}+M^{-1}\lr{\xi}$ and  $\lr{\xi}^{-3/2+j/2}\leq C\lambda_j^{1/2}$ for $1\leq j\leq 3$ then
\begin{equation}
\label{eq:phi-n:sita}
\begin{split}
M^{8n}t^{-1}\|V\|^2_{-1}/C\leq t^{2n}{\mathcal E}_1(V)
\leq CM^{-2n}(t^{-1}\|V\|^2_n+\|V\|^2_{n+1}),\\
M^{8n}\|V\|^2_{-1}/C\leq t^{2n}{\mathcal E}_2(V)
\leq CM^{-2n}\|V\|^2_{n}.
\end{split}
\end{equation}
Assume  $D_t^jV\in {\mathcal H}_{-n-1/2+j/2, n+1-j}(0, M^{-4})$, $j=0, 1$. From this one sees that $\lim_{t\to +0}\|V(t)\|_n$ exists which is $0$. Using this we see $\lim_{t\to +0}t^{-n}\|V(t)\|_n=0$. Noting that ${\mathcal E}(V)\leq CM^{-n}t^{-2n}\|V(t)\|_n^2$ and integrating \eqref{eq:dif:ene:fu} in $t$ 
we obtain
\begin{prop}
\label{pro:fu:matome}There exist $c_i>0$, $C>0$ and $M_0$, $\gamma_0$, $\theta_0$  such that for any $V$ with $D_t^jV\in {\mathcal H}_{-n-1/2+j/2, n+1-j}(0, M^{-4})$, $j=0, 1$, one has
\begin{gather*}
c_1t^{-2n}e^{-\theta t}\|V(t)\|_{-1}^2
+c_2\int_0^te^{-\theta s}s^{-2n-1}\|V(s)\|_{-1}^2ds\\
+c_3\theta \int_0^te^{-\theta s}s^{-2n}\|V(s)\|_{-1}^2ds
\leq CM^{1-10n}\int_0^te^{-\theta s}s^{-2n+1}\|{\tilde L}V(s)\|_n^2ds
\end{gather*}
for $0\leq t\leq M^{-4}$, $M\geq M_0$, $\gamma\geq \gamma_0$, $\theta\geq \theta_0$.
\end{prop}
\begin{cor}
\label{cor:tokusyu}For any $V$ with $D_t^jV\in {\mathcal H}_{-n-1/2+j/2, n+1-j}(0, M^{-4})$, $j=0, 1$
\begin{gather*}
\int_0^ts^{-2n-1}\|V(s)\|_{-1}^2ds
\leq C\int_0^ts^{-2n+1}\|{\tilde L}V(s)\|_n^2ds,\;\;0\leq t\leq M^{-4}.
\end{gather*}
\end{cor}
Consider the adjoint operator ${\hat P^*}$ of ${\hat P}$. Noting $a_M\in {\mathcal C}(\sigma)$, $b\in {\mathcal C}(\sigma^{3/2})$ and \eqref{eq:xi:kakudai}, \eqref{eq:dM:seimitu} we see that
\begin{align*}
{\hat P^*}=D_t^3-a_M(t, x, D)[D]^2D_t-b(t, x, D)\,[D]^3\\
+b_1D_t^2+\big({\tilde b}_2+d_M)[D]D_t+{\tilde b}_3[D]^2
+{\tilde c}_1D_t+{\tilde c}_2[D]
\end{align*}
with ${\tilde b}_j\in S(1,g)$ and ${\tilde c}_j\in S(M^2, g)$ hence ${\tilde c}_j[D]^{-1}\in S(M^{-3}, g)$ where it is not difficult to check that $
{\tilde b}_3-\big(b_3+i e\big)\in S(M^{-3},g)$. Denote by ${\tilde L}^*$  the corresponding first order system (which is not the adjoint of ${\tilde L}$).
Since the power $n$ of the weight $\phi^{-n}$ depends only on $a$, $b$ and $b_3$ (see \eqref{eq:n:atai}) then we can choose the same $n$ for ${\hat P^*}$ as for ${\hat P}$. Then employing the weighted energy
\[
{\mathcal E}^*(V)=e^{\theta t}\big(\op{\varLambda}\op{t^n\phi^n}V,\op{t^n\phi^n}V\big)
\]
and repeating the same arguments for ${\mathcal E}(V)$ and making the integration
\[
-\int^{\delta}_t\frac{d}{dt}{\mathcal E}^*\,dt,\quad 0<t<\delta=M^{-4}
\]
we have
\begin{prop}
\label{pro:fu:matome:ad}There exist $c_i>0$, $C>0$ and $M_0$, $\gamma_0$, $\theta_0$  such that for any $V$ with $D_t^jV\in {\mathcal H}_{n-1/2+j/2, 1-j}(0, M^{-4})$, $j=0, 1$, one has
\begin{gather*}
c_1t^{2n}e^{\theta t}\|V(t)\|_{-n-1}^2
+c_2\int_t^{\delta}e^{\theta \tau}\tau^{2n-1}\|V(\tau)\|_{-n-1}^2d\tau\\
+c_3\theta \int_t^{\delta}e^{-\theta \tau}\tau^{2n}\|V(\tau)\|_{-n-1}^2d\tau
\leq CM^{-10n}\delta^{2n}e^{\theta \delta}\|V(\delta)\|^2\\
+CM^{1-10n}\int_t^{\delta}e^{\theta \tau}\tau^{2n+1}\|{\tilde L}^*V(\tau)\|^2d\tau,\quad 0\leq t\leq \delta
\end{gather*}
for $M\geq M_0$, $\gamma\geq \gamma_0$, $\theta\geq \theta_0$  where ${\tilde L}^*V=\op{T}{^t}({\hat P^*}u,0,0)$ and $\op{T}V={^t}(D_t^2u, [D]D_tu, [D]^2u)$. 
\end{prop}
\begin{remark}\rm
\label{rem:hatP:sonzai:n}It is clear from the proof  that for any  $n'\geq n$, 
Propositions \ref{pro:fu:matome} and \ref{pro:fu:matome:ad} hold.
\end{remark}
%

\section{Preliminary existence result}
\label{sec:pre:sonzai}

Let $s\in \R$ and we estimate $\olr{D}^sV$. In what follows we fix $M$ and $\gamma$ (actually it is enough to choose $\gamma=M^5$, see \eqref{eq:seigen}) such that Propositions \ref{pro:fu:matome} and \ref{pro:fu:matome:ad} hold, therefore from now on we can assume $\lr{D}=\olr{D}=[D]$ without restrictions, while $\theta$ remains to be free.   From \eqref{eq:eq:V} one has
\[
\dif_t(\olr{D}^{s}V)=\big(\op{i{\mathcal A}+i{\mathcal B}}+i[\olr{D}^s,\op{{\mathcal A}+{\mathcal B}}]\olr{D}^{-s}\big)\olr{D}^sV+\olr{D}^sF.
\]
\begin{lem}
\label{lem:calA:D:s} For any $s\in\R$ there is $C>0$ such that
\[
\big|([\olr{D}^s,\op{{\mathcal A}}]V, \op{\varLambda}\olr{D}^sV)\big|\leq C{\mathcal E}_2(\olr{D}^sV).
\]
\end{lem}
\begin{proof} Denoting $T^{-1}AT=({\tilde a}_{i j})$ 
thanks to Lemmas \ref{lem:nami:a} and \ref{lem:lam:Sg}  we see that
\begin{equation}
\label{eq:calA:D:s:1}
\begin{split}
(({\tilde a}_{3j}[\xi])\#\olr{\xi}^s-\olr{\xi}^s\#({\tilde a}_{3j}[\xi]))\#\olr{\xi}^{-s}\in S(\sigma^{2-j}\sqrt{\lambda_j}, g),\;\; j=1, 2,
\\
(({\tilde a}_{21}[\xi])\#\olr{\xi}^s-\lr{\xi}^s\#({\tilde a}_{21}[\xi]))\#\olr{\xi}^{-s}\in S(\sigma^{-1/2}\sqrt{\lambda_1}, g)
\end{split}
\end{equation}
where $S(\sigma^{-1/2}\sqrt{\lambda_1}, g)
=S(\lambda_2^{-1/2}\sqrt{\lambda_1}, g)$.
From  Corollary \ref{cor:TAT:seimitu} it is easy to see $(({\tilde a}_{i j}[\xi])\#\olr{\xi}^s-\olr{\xi}^s\#({\tilde a}_{i j}[\xi]))\#\olr{\xi}^{-s}\in S(1, g)$ for $j\geq i$.
 Then together with \eqref{eq:calA:D:s:1} the proof follows from a repetition of similar arguments.
\end{proof}
\begin{lem}
\label{lem:calB:D:s} For any $s\in\R$ and any $\ep>0$ there is $C>0$ such that
\[
\big|([\olr{D}^s,\op{{\mathcal B}}]V, \op{\varLambda}\olr{D}^sV)\big|\leq \ep\,{\mathcal E}_1(\olr{D}^sV)+C{\mathcal E}_2(\olr{D}^sV).
\]
\end{lem}
\begin{proof} Write ${\mathcal B}_1=({\tilde b}_{i j})$. Since ${\tilde b}_{i j}\in S(1, g)$ then $\lambda_i\#(\olr{\xi}^s\#{\tilde b}_{ij}-{\tilde b}_{i j}\#\olr{\xi}^s)\#\olr{\xi}^{-s}$ is in $S(\olr{\xi}^{-1/2}\lambda_i, g)$. Noting $C\lambda_1\geq \sigma\olr{\xi}^{-1}$ and $C\lambda_2\geq \sigma\geq \olr{\xi}^{-1}$ it is easy to see that $S(\olr{\xi}^{-1/2}\lambda_i, g)\subset S(\lambda_i^{1/2}\lambda_j^{1/2}, g)$ except for $(i, j)=(3, 1)$. For $(i, j)=(3, 1)$ recalling ${\tilde b}_{3 1}=b_3+S(\sigma^{1/2})$ by \eqref{eq:T:B:T} one sees
\begin{align*}
\lambda_3\#(\olr{\xi}^s\#{\tilde b}_{3 1}-{\tilde b}_{3 1}\#\olr{\xi}^s)\#\olr{\xi}^{-s}\in S(\sigma^{1/2}\olr{\xi}^{-1/2}, g)\subset S(\lambda_1^{1/2}\lambda_3^{1/2}, g)
\end{align*}
(recall that $M$ has been fixed). Therefore one obtains 
\begin{equation}
\label{eq:calB:D:s}
\big|([\olr{D}^s,\op{{\mathcal B}_1}]V, \op{\varLambda}\olr{D}^sV)\big|\leq C{\mathcal E}_2(\olr{D}^sV).
\end{equation}
Next consider $T^{-1}\dif_tT=({\tilde t}_{i j})$. Recalling ${\tilde t}_{2 1}\in {\mathcal C}(\sigma^{-1/2})$,  ${\tilde t}_{3 1}\in {\mathcal C}(1)$ and ${\tilde t}_{3 2}\in {\mathcal C}(\sigma^{1/2})$ and noting $C\kappa\lambda_1\geq \olr{\xi}^{-1}$ and $C\lambda_2\geq \sigma\geq \olr{\xi}^{-1}$ one has 
\[
\lambda_i\#(\olr{\xi}^s\#{\tilde t}_{i j}-{\tilde t}_{i j}\#\olr{\xi}^s)\#\olr{\xi}^{-s}\in S(\sqrt{\kappa\lambda_j}\,\sqrt{\lambda_i}, g),\;\;\quad i> j.
\]
Therefore we have
\begin{align*}
\big|([\olr{D}^s,\op{T^{-1}\dif_t T}]V, \op{\varLambda}\olr{D}^sV)\big|\leq C\sqrt{{\mathcal E}_1(\olr{D}^sV)}\sqrt{{\mathcal E}_2(\olr{D}^sV)}\\
\leq \ep\, {\mathcal E}_1(\olr{D}^sV)+C^2\ep^{-1}{\mathcal E}_2(\olr{D}^sV)
\end{align*}
which together with \eqref{eq:calB:D:s} proves the assertion.
\end{proof}
Choosing $\ep>0$ smaller than  $c_2$ in Proposition \ref{pro:fu:matome} and  $\theta$  large   we conclude
\begin{prop}
\label{pro:fu:matome:s}For any $s\in \R$ there exists  $C>0$ such that for any $V$ with $D_t^jV\in {\mathcal H}_{-n-1/2+j/2, n+s+1-j}(0, \delta)$, $j=0, 1$, one has 
\begin{gather*}
t^{-2n}\|V(t)\|_{s-1}^2+\int_0^t\tau^{-2n-1}\|V(\tau)\|_{s-1}^2
d\tau
\leq C\int_0^t\tau^{-2n+1}\|{\tilde L}V(\tau)\|^2_{n+s}d\tau
\end{gather*}
for $0\leq t\leq \delta$. 
\end{prop}
Since ${\tilde L}=\op{I+K}\op{T^{-1}}\cdot L\cdot \op{T}$ with $T$, $T^{-1}\in S(1, g)$ then $\|{\tilde L}V\|_{s}\leq C_s\|L\cdot\op{T}V\|_{s}$ and 
$\|\op{T}V\|_{s}\leq C_s\|V\|_{s}$ with some $C_s>0$. Thanks to   Lemma \ref{lem:diftT:seimitu} one has $\|\op{\dif_tT}V\|_s\leq C_st^{-1/2}\|V\|_s$ then  replacing $\op{T}V$ by $U$ one sees that for any $U$ with $D_t^jU\in {\mathcal H}_{n-1/2+j/2, n+s+1-j}(0, \delta)$, $j=0, 1$ it holds 
\begin{gather*}
t^{-2n}\|U(t)\|_{s-1}^2+
\int_0^t{\tau}^{-2n-1}\|U(\tau)\|_{s-1}^2d\tau
\leq C\int_0^t\tau^{-2n+1}\|LU(\tau)\|_{n+s}^2d\tau.
\end{gather*}
Since $U={^t}(D_t^2u,[D]D_tu,[D]^2u)$ and $LU={^t}({\hat P}u,0,0)$ we have
\begin{prop}
\label{pro:V:to:u} For any $s\in \R$ there is $C>0$  such that for any $u$ with $D_t^ju\in {\mathcal H}_{-n-1/2, n+s+3-j}(0,\delta)$, $0\leq j\leq 3$, one has
\begin{equation}
\label{eq:hat:P}
\begin{split}
t^{-2n}\sum_{j=0}^2\|D_t^ju(t)\|_{s+1-j}^2+
\sum_{j=0}^2\int_0^t{\tau}^{-2n-1}\|D_t^ju(\tau)\|_{s+1-j}^2d\tau\\
\leq C\int_0^t\tau^{-2n+1}\|{\hat P}u(\tau)\|_{n+s}^2d\tau,\quad 0\leq t\leq \delta.
\end{split}
\end{equation}
\end{prop}
Repeating the same arguments we conclude
\begin{prop}
\label{lem:tokusyu:s:ad}For any $s\in \R$ there is $C>0$  such that for any $u$ with $D_t^ju\in {\mathcal H}_{n-1/2+j/2, s+3-j}(0,\delta)$, $0\leq j\leq 3$, we have
\begin{equation}
\label{eq:hat:P:ad}
\begin{split}
t^{2n}\sum_{j=0}^2\|D_t^ju(t)\|_{s+1-j}^2+
\sum_{j=0}^2\int_t^{\delta}{\tau}^{2n-1}\|D_t^ju(\tau)\|_{s+1-j}^2d\tau\\
\leq C\big(\sum_{j=0}^2\|D_t^ju(\delta)\|_{n+s+2-j}^2
+\int_t^{\delta}\tau^{2n+1}\|{\hat P^*}u(\tau)\|_{n+s}^2d\tau\big),\;\;0<t\leq \delta.
\end{split}
\end{equation}
 \end{prop}
 Since  \eqref{eq:hat:P:ad} holds with ${\bar n}\geq n+3$  as noticed in Remark \ref{rem:hatP:sonzai:n} then, in the resulting \eqref{eq:hat:P:ad},   replacing $s$ by  $-3{\bar n}-s-1$  we have
 \[
 \int_0^{\delta}t^{2{\bar n}-1}\|u(t)\|_{-3{\bar n}-s}^2dt\leq C\int_0^{\delta}t^{2{\bar n}+1}\|{\hat P^*}u(t)\|_{-2{\bar n}-s-1}^2dt
 \]
for any $u\in C_0^{\infty}((0,\delta)\times \R^d)$. This implies
\begin{align*}
\big|\int_0^{\delta}(f,v)dt\big|\leq \big(\int_0^{\delta}t^{-2{\bar n}+1}\|f\|_{3{\bar n}+s}^2dt\big)^{1/2}\big(\int_0^{\delta}t^{2{\bar n}-1}\|v\|_{-3{\bar n}-s}^2dt\big)^{1/2}\\
\leq C\big(\int_0^{\delta}t^{-2{\bar n}+1}\|f\|_{3{\bar n}+s}^2dt\big)^{1/2}\big(\int_0^{\delta}t^{2{\bar n}+1}\|{\hat P^*}v\|_{-2{\bar n}-s-1}^2dt\big)^{1/2}
\end{align*}
for any $v\in C_0^{\infty}((0,\delta)\times \R^d)$ and $f\in {\mathcal H}_{-{\bar n}+1/2, 3{\bar n}+s}(0,\delta)$. Using the Hahn-Banach theorem to extend the anti-linear form in ${\hat P^*}v$;
\begin{equation}
\label{eq:keisiki}
{\hat P^*}v\mapsto \int_0^{\delta}(f, v)dt
\end{equation}
we conclude that there is some $u\in {\mathcal H}_{-{\bar n}-1/2,2{\bar n}+s+1 }(0,\delta)$  such that
\[
\int_0^{\delta}(f, v)dt=\int_0^{\delta}(u,{\hat P^*}v)dt,\quad v\in C_0^{\infty}((0,\delta)\times \R^d).
\]
This implies that ${\hat P}u=f$. Since $u\in {\mathcal H}_{0, 2{\bar n}+s+1}(0,\delta)$ and $f\in {\mathcal H}_{0, 3{\bar n}+s}(0,\delta)$  it follows from \cite[Theorem B.2.9]{Hobook} that $D_t^ju\in {\mathcal H}_{0, 2{\bar n}+s+1-j}(0,\delta)$ for $ j=0,1,2,\ldots$. Thus with $w=\olr{D}^{{\bar n}+s}u$ one has $D_t^jw\in L^2((0,\delta)\times \R^d)$ for $j=0,\ldots, {\bar n}+1$ hence $D_t^jw(0)$ exists in $L^2(\R^d)$ which is $0$ for $j=0,\ldots, {\bar n}$ for $w\in {\mathcal H}_{-{\bar n}+1/2, 0}(0,\delta)$. Thus one can write $w(t)=\int_0^t(t-\tau)^{{\bar n}}\dif_t^{{\bar n}+1}w(\tau)d\tau/{\bar n}!$. From this one concludes that $D_t^ju\in {\mathcal H}_{-{\bar n}+j-1/2, {\bar n}+s}(0,\delta)$ hence $D_t^ju\in {\mathcal H}_{-n-1/2, n+s+3-j}(0,\delta)$ for $0\leq j\leq 3$ because ${\bar n}\geq n+3$ thus \eqref{eq:hat:P} holds for this $u$. Now let $f\in {\mathcal H}_{-n+1/2, n+s}(0,\delta)$. Take a rapidly decreasing function $\rho(\xi)$ with $\rho(0)=1$ then $f_{\ep}=e^{-\ep/t}\rho(\ep D)f\in {\mathcal H}_{-{\bar n}+1/2, 2{\bar n}+s+1}(0,\delta)$ and $f_{\ep}\to f$ in ${\mathcal H}_{-n+1/2, n+s}(0,\delta)$. As just proved above there is $u_{\ep}$ satisfying ${\hat P}u_{\ep}=f_{\ep}$ and \eqref{eq:hat:P}. Therefore choosing a weakly convergent subsequence $\{u_{\ep'}\}$ one can conclude
\begin{them}
\label{thm:pre:sonzai:s} There exists $\delta>0$ such that for any $s\in \R$ and  any $f\in {\mathcal H}_{-n+1/2, n+s}(0,\delta)$ there exists a unique $u$ with $D_t^ju\in {\mathcal H}_{-n-1/2, 1+s-j}(0, \delta)$, $j=0,1,2$, satisfying ${\hat P}u=f$ and \eqref{eq:hat:P}.
\end{them}
Instead of \eqref{eq:keisiki} considering the anti-linear form in ${\hat P}v$;
\begin{gather*}
{\hat P}v\mapsto \int_0^{\delta}(f, v)dt +\sum_{j=0}^1\big( w_{2-j}, D_t^jv(\delta, \cdot)\big)
+\big(w_0, (D_t^2-[D]^2a(\delta, x,  D))v(\delta, \cdot)\big)
\end{gather*}
for $v\in C_0^{\infty}((0,\infty)\times \R^d)$ and repeating similar arguments adopting \eqref{eq:hat:P} we conclude
\begin{them}
\label{thm:pre:sonzai:ad} There exists $\delta>0$ such that for any $s\in \R$ and  any $f\in {\mathcal H}_{n+1/2, n+s}(0,\delta)$ and $w_j\in H^{n+s+2-j}$, $j=0,1,2$, there is a unique $u$ with  $D_t^ju\in {\mathcal H}_{n-1/2, 1+s-j}(0,\delta)$ satisfying ${\hat P^*}u=f$, $D_t^ju(\delta,\cdot)=w_j$, $j=0,1,2$ and \eqref{eq:hat:P:ad}.
\end{them}
Indeed we first see that there is $u\in {\mathcal H}_{n-1/2,1+s}(0,\delta)$ satisfying ${\hat P^*}u=f$ and $D_t^ju(\delta)=w_j$, $j=0,1,2$ (e.g. \cite[Chapter 23]{Hobook}). Since $f\in {\mathcal H}_{0, n+s}(\varep,\delta)$ and $u\in {\mathcal H}_{0, 1+s}(\varep,\delta)$ it follows from \cite[Theorem B.2.9]{Hobook}  that $D_t^ju\in {\mathcal H}_{0, 1+s-j}(\varep,\delta)$, $0\leq j\leq 2$ for any $\varep>0$.   Applying \eqref{eq:hat:P:ad} with $t=\varep$ we conclude $D_t^ju\in {\mathcal H}_{n-1/2, 1+s-j}(0,\delta)$, $j=1, 2$ since $\varep>0$ is arbitrary.

\section{Propagation of wave front set} 
\label{sec:kyokusho}

In Section \ref{sec:pre:sonzai} we have proved an existence result of the Cauchy problem for ${\hat P}$, which coincides with the original $P$ only in $W_M$. Following \cite{Ni1}, \cite{Iv2} (also \cite{Ni:book}) we show that the wave front set of $u(t,\cdot)$, obtained by Theorem \ref{thm:pre:sonzai:s}, propagates with finite speed. This fact enables us to solve the Cauchy problem for the original $P$.

\subsection{Estimate of wave front set}
\label{sec:WF:hyoka:1}

Let $\chi(x)\in C_0^{\infty}(\R^d)$ be equal to $1$ near $x=0$ and vanish in $|x|\geq 1$. Set
\begin{equation}
\label{eq:f:ep:teigi}
\begin{split}
d_{\ep}(x,\xi;y,\eta)=\big\{\chi(x-y)|x-y|^2+|\xi\olr{\xi}^{-1}-\eta\olr{\eta}^{-1}|^2+\ep^2\big\}^{1/2},\\
f_{\ep}(t,x,\xi;y,\eta)=t-T+\nu d_{\ep}(x,\xi; y,\eta),\quad T>0
\end{split}
\end{equation}
where $(y,\eta)\in \R^d\times(\R^d\setminus\{0\})$ and $\nu$ is a positive small parameter. Note that
\begin{equation}
\label{eq:f:d:ep}
\big|\dif_x^{\al}\dif_{\xi}^{\be}d_{\ep}\big|\leq C\olr{\xi}^{-|\be|},\quad |\al+\be|=1
\end{equation}
where $C$ is independent of $\ep>0$. Define $\Phi_{\ep}$ by
\begin{equation}
\label{eq:cut:f}
\Phi_{\ep}(t,x,\xi)=\begin{cases}\exp{(1/f_{\ep}(t, x,\xi))}\quad \text{if}\;\;f_{\ep}<0\\
0\quad \text{otherwise}
\end{cases}
\end{equation}
and note that $\Phi_{\ep}\in S(1, g_0)$ for any fixed $\ep>0$ where $g_0=|dx|^2+\olr{\xi}^{-2}|d\xi|^2$. From now on to simplify notations, we denote
\[
{\mathcal E}_1(\olr{D}^sV)+{\mathcal E}_2(\olr{D}^sV)=t^{-2n}{\mathcal N}_s(V),\quad {\mathcal N}_0(V)={\mathcal N}(V).
\]
%

\begin{lem}
\label{lem:arata}There is $\nu_0>0$ such that for any $0<\nu\leq \nu_0$ and any $\ep>0$ there is $C>0$ such that for any $V$ with $\int_0^{\delta}t^{-2n}{\mathcal N}_{-1/4}(V)dt<+\infty$ and  ${\tilde L}V\in {\mathcal H}_{-n+1/2, l}(0,\delta)$ with some $l$ we have
\begin{gather*}
{\mathcal E}_1(\op{\Phi_{\ep}}V)+\int_0^t\tau^{-2n}{\mathcal N}(\op{\Phi_{\ep}}V)d\tau
\leq C\int_0^t\tau^{-2n+1}\|\op{\Phi_{\ep}}{\tilde L}V\|_n^2d\tau\\
+C\int_0^t\tau^{-2n}{\mathcal  N}_{-1/4}(V)d\tau,\quad 0<t\leq \delta.
\end{gather*}
\end{lem}
\begin{proof} Denote $V^{\mu}=\olr{\mu D}^{-{\bar n}}V$ with small $\mu>0$ where we choose ${\bar n}=2n+\max\{-l,0\}+3$.
Assume ${\tilde L}V=F$ so that ${\tilde L}V^{\mu}=F^{\mu}+R^{\mu}V^{\mu}=G^{\mu}$ where $R^{\mu}=[\olr{\mu D}^{-\bar n}, \op{{\mathcal A}+{\mathcal B}}]\olr{\mu D}^{\bar n}$ and $F^{\mu}=\olr{\mu D}^{-{\bar n}}F$. 
 Note that $\Phi_{\ep 1}=f_{\ep}^{-1}\Phi_{\ep}\in S(1, g_0)$ for any fixed $\ep>0$ and $
\Phi_{\ep}-f_{\ep}\#\Phi_{\ep 1}\in S(\olr{\xi}^{-1}, g_0)$. 
Since $\dif_t \Phi_{\ep}=-\Phi_{\ep 1}/f_{\ep}$ one can write 
\begin{equation}
\label{eq:cut:siki}
\begin{split}
\dif_t(\op{\Phi_{\ep}}V^{\mu})=-\op{f_{\ep}^{-1}\Phi_{\ep 1}}V^{\mu}+(\op{i{\mathcal A}+i{\mathcal B}})\op{\Phi_{\ep}}V^{\mu}\\
+[\op{\Phi_{\ep}}, \op{i{\mathcal A}+i{\mathcal B}}]V^{\mu}+\op{\Phi_{\ep}}G^{\mu}.
\end{split}
\end{equation}
Since $\Phi_{\ep}\#{\mathcal B}_1-{\mathcal B}_1\#\Phi_{\ep}\in S(\olr{\xi}^{-1/2}, g)$ by Proposition \ref{pro:T:henkan} it is not difficult to see from the proof of Corollary \ref{cor:LamB:1} that
\begin{align*}
\big|\big(\op{\varLambda}\op{\phi^{-n}}[\op{\Phi_{\ep}},\op{{\mathcal B}_1}]V^{\mu}, \op{\phi^{-n}}\op{\Phi_{\ep}}V^{\mu}\big)\big|
\leq c(\ep){\mathcal N}_{-1/4}(V^{\mu})
\end{align*}
Denote $\Phi_{\ep}\#(T^{-1}\dif_tT)-(T^{-1}\dif_tT)\#\Phi_{\ep}=(\varphi_{i j})$ hence $\varphi_{21}\in S(\sigma^{-1}\olr{\xi}^{-1}, g)$ and $\varphi_{3 1}\in S(\sigma^{-1/2}\olr{\xi}^{-1}, g)$ from \eqref{eq:dT:-1:T}. Thus $\lambda_j\#\varphi_{j 1}\in S(\olr{\xi}^{-1/2}\sqrt{\kappa\lambda_1}\,\sqrt{\kappa\lambda_j}, g)$ for $j=2, 3$  because $C\lambda_1\geq \sigma\olr{\xi}^{-1}$, $C\lambda_2\geq \sigma$ and $\kappa\sigma\geq 1$. A repetition of similar arguments proving \eqref{eq:TdT:comm} shows that
\begin{align*}
\big|\big(\op{\varLambda}\op{\phi^{-n}}[\op{\Phi_{\ep}},\op{T^{-1}\dif_tT}]V^{\mu}, \op{\phi^{-n}}\op{\Phi_{\ep}}V^{\mu}\big)\big|
\leq c(\ep){\mathcal N}_{-1/4}(V^{\mu}).
\end{align*}
Note that $\Phi_{\ep}\#{\mathcal A}-{\mathcal A}\#\Phi_{\ep}$ can be written
\[
\sum_{|\al+\be|=1}\frac{(-1)^{|\al|}}
{(2i)^{|\al+\be|}\al!\be!}\Big(\dif_x^{\al}\dif_{\xi}^{\be}\Phi_{\ep}\dif_x^{\be}\dif_{\xi}^{\al}{\mathcal A}-\dif_x^{\be}\dif_{\xi}^{\al}\Phi_{\ep}\dif_x^{\al}\dif_{\xi}^{\be}{\mathcal A}\Big)+R_{\ep}=H_{\ep}+R_{\ep}
\]
where it follows from Corollary \ref{cor:TAT:seimitu} that $R_{\ep}\in S(\olr{\xi}^{-1/2}, g)$ for $\sigma\geq \olr{\xi}^{-1}$. It is not difficult to see from the proof of Corollary \ref{cor:LamB:1} that
\begin{align*}
\big|\big(\op{\varLambda}\op{\phi^{-n}}\op{R_{\ep}}V^{\mu}, \op{\phi^{-n}}\op{\Phi_{\ep}}V^{\mu}\big)\big|
\leq c(\ep){\mathcal N}_{-1/4}(V^{\mu}).
\end{align*}
Note that  $H_{\ep}\in S(1, g)$ because $\dif_x^{\al}\dif_{\xi}^{\be}{\mathcal A}\in S(\olr{\xi}^{1-|\be|}, g)$ for $|\al+\be|=1$. Write $
\Phi_{\ep}=f_{\ep}\#\Phi_{\ep 1}+r_{\ep}$ with $r_{\ep}\in S(\olr{\xi}^{-1}, g_0)
$. Noting $\phi^{-n}\#f_{\ep}-f_{\ep}\#\phi^{-n}\in S(\omega^{-1}\rho^{1/2}\olr{\xi}^{-1}\phi^{-n}, g)\subset S(\phi^{-n}\olr{\xi}^{-1/2}, g)$ and $
f_{\ep}\#\lambda_j-\lambda_j\#f_{\ep}\in S(\lambda_j\olr{\xi}^{-1/2}, g)$ we see that
\begin{gather*}
\big|\big(\op{\varLambda}\op{\phi^{-n}}\op{iH_{\ep}}V^{\mu}, \op{\phi^{-n}}\op{\Phi_{\ep}}V^{\mu}\big)\\
-\big(\op{\varLambda}\op{\phi^{-n}}\op{f_{\ep}}\op{iH_{\ep}}V^{\mu}, \op{\phi^{-n}}\op{\Phi_{\ep 1}}V^{\mu}\big)\big|
\end{gather*}
 is bounded by $c(\ep){\mathcal N}(\olr{D}^{-1/4}V^{\mu})$. Here look at $iH_{\ep}$ more carefully;
\[
 iH_{\ep}=\Big(\sum_{|\al+\be|=1}\dif_x^{\al}\dif_{\xi}^{\be}({\tilde a}_{i j}[\xi])(\dif_x^{\be}\dif_{\xi}^{\al}f_{\ep})\frac{1}{f_{\ep}}\Phi_{1 \ep}\Big)=\big(h^{\ep}_{i j}\big)\frac{1}{f_{\ep}}\Phi_{1 \ep}.
\]
 Taking $h^{\ep}_{i j}\in S(1, g)$ and $f^{-1}_{\ep}\Phi_{\ep 1}$, $\Phi_{\ep 1}\in S(1, g_0)$  into account one can write $
 f_{\ep}\#(iH_{\ep})=\big(h^{\ep}_{i j}\big)\#\Phi_{\ep 1}
+R_{\ep}$ with $R_{\ep}\in S(\olr{\xi}^{-1/2}, g)$ hence denoting ${\tilde H_{\ep}}=\big(h^{\ep}_{i j}\big)$
\begin{gather*}
\big|\big(\op{\varLambda}\op{\phi^{-n}}\op{f_{\ep}}\op{iH_{\ep}}V^{\mu}, \op{\phi^{-n}}\op{\Phi_{\ep 1}}V^{\mu})\\
-(\op{\varLambda}\op{\phi^{-n}}\op{{\tilde H_{\ep}}}\op{\Phi_{\ep 1}}V^{\mu}, \op{\phi^{-n}}\op{\Phi_{\ep 1}}V^{\mu}\big)
\big|
\end{gather*}
is bounded by $c(\ep){\mathcal N}_{-1/4}(V^{\mu})$. From Corollary \ref{cor:TAT:seimitu} we see that $
h^{\ep}_{i j}\in {\mathcal C}(1)$ for $j\geq i$, $h^{\ep}_{2 1}, h^{\ep}_{3 2}\in {\mathcal C}(\sigma^{1/2})$ and $h^{\ep}_{3 1}\in {\mathcal C}(\sigma)$ 
then in view of Lemma \ref{lem:dif:Phi:seimitu} one has 
$\lambda_i\#\big(\phi^{-n}\#h^{\ep}_{i j}-h^{\ep}_{i j}\#\phi^{-n}\big)\in S(\kappa \lambda_i\olr{\xi}^{-1}\phi^{-n}, g)$ for $ j\geq i$ and 
$\lambda_i\#\big(\phi^{-n}\#h^{\ep}_{i j}-h^{\ep}_{i j}\#\phi^{-n}\big)\in S(\kappa\lambda_i\lambda_j^{1/2}\olr{\xi}^{-1/2}\phi^{-n}, g)$  for $i>j$. 
From this we see that 
\begin{align*}
\big|\big(\op{\varLambda}\op{\phi^{-n}}\op{{\tilde H}_{\ep}}\op{\Phi_{\ep 1}}V^{\mu}, \op{\phi^{-n}}\op{\Phi_{\ep 1}}V^{\mu})\\
-(\op{\varLambda}\op{{\tilde H_{\ep}}}\op{\phi^{-n}}\op{\Phi_{\ep 1}}V^{\mu}, \op{\phi^{-n}}\op{\Phi_{\ep 1}}V^{\mu}\big)
\big|
\end{align*}
is bounded by $c(\ep){\mathcal N}_{-1/4}(V^{\mu})$. 
\begin{lem}
\label{lem:h:bunkai}One can write 
\[
h^{\ep}_{ ij}=\sum_{|\al+\be|=1}k^{\ep}_{i j \al \be}\#l_{i j \al \be}+r^{\ep}_{i j}
\]
where $k^{\ep}_{i j \al \be}\in S(1, g_0)$ such that $|k^{\ep}_{i j \al \be}|\leq C\nu$ with some $C>0$ independent of $\nu$ and $\ep$ for any $1\leq i, j\leq 3$. As for $l_{i j \al \be}$ and $r^{\ep}_{i j}$ one has $
l_{i j \al \be}\in S(1, g)$ and $r^{\ep}_{i j}\in S(\sigma^{-1/2}\olr{\xi}^{-1}, g)$ for  $j\geq i$ and $l_{i j \al \be}\in S(\sigma^{(i-3)/2}\sqrt{\lambda_j}, g)$, $r^{\ep}_{i j}\in S(\sigma^{(i-3)/2}\sqrt{\lambda_j}\olr{\xi}^{-1/2}, g)$ for $i>j$.
\end{lem}
\begin{proof} Set  $k^{\ep}_{i j \al \be}=\olr{\xi}^{|\al|}\dif_x^{\be}\dif_{\xi}^{\al}f_{\ep}$ and $l_{i j \al \be}=\olr{\xi}^{-|\al|}\dif_x^{\al}\dif_{\xi}^{\be}({\tilde a}_{i j}[\xi])$ 
then the assertion for $k^{\ep}_{i j \al \be}$ is clear from \eqref{eq:f:d:ep}. The assertions for $l_{i j \al \be}$ follow from Corollary \ref{cor:TAT:seimitu} and  Lemmas \ref{lem:nami:a}, \ref{lem:lam:Sg}. Note that $
\dif_x^{\mu}\dif_{\xi}^{\nu}l_{i j \al \be}\in S(\sigma^{-1/2}\olr{\xi}^{-|\nu|}, g)$, $|\mu+\nu|=1$  for $j\geq i$ and $\dif_x^{\mu}\dif_{\xi}^{\nu}l_{2 1 \al \be}$, $\dif_x^{\mu}\dif_{\xi}^{\nu}l_{3 2 \al \be}\in S(\olr{\xi}^{-|\nu|}, g)$  and $\dif_x^{\mu}\dif_{\xi}^{\nu}l_{3 1 \al \be}\in S(\sigma^{1/2}\olr{\xi}^{-|\nu|}, g)$ 
for $|\mu+\nu|=1$ which follows from ${\tilde a}_{2 1}$, ${\tilde a}_{32}\in {\mathcal C}(\sigma)$ and ${\tilde a}_{3 1}\in {\mathcal C}(\sigma^{5/2})$. Then remarking that $\sigma\geq \lr{\xi}^{-1}$ and $C\lambda_1\geq \sigma\lr{\xi}^{-1}$ the assertions for $r^{\ep}_{i j}$ are checked immediately.
\end{proof}
With $R^{\ep}=(r^{\ep}_{i j})$ and $W=\op{\phi^{-n}}\op{\Phi_{\ep 1}}V^{\mu}$, recalling $\lambda_1\leq C\sigma\lambda_2\leq C'\sigma^2\lambda_3$, it is easy to see $
\big|(\op{R^{\ep}}W, \op{\varLambda}W)\big|\leq c(\ep)\|\op{\varLambda^{1/2}}\olr{D}^{-1/4}W\|^2$. 
Turn to $|(\op{h^{\ep}_{i j}}W_j,  \op{\lambda_i}W_i)|$.  Thanks to Lemma \ref{lem:h:bunkai} this is bounded by
\begin{align*}
C\|\op{\lambda_j^{1/2}}W_j\|\|\op{k^{\ep}_{i j \al \be}}\op{\lambda_i^{1/2}}W_i\|
\end{align*}
with $C$ independent of $\ep$ because $\lambda_i^{1/2}\#l_{i j \al \be}\in S(\lambda_j^{1/2}, g)$ in view of Lemma  \ref{lem:h:bunkai}. On the other hand, taking Lemma \ref{lem:h:bunkai} into account, it follows from the sharp G\aa rding inequality (e.g. \cite[Theorem 18.1.14]{Hobook}) that
\begin{align*}
\|\op{k^{\ep}_{i j \al \be}}\op{\lambda_i^{1/2}}W_i\|\leq C\nu\|\op{\lambda_i^{1/2}}W_i\|+C(\nu,\ep)\|\op{\lambda_i^{1/2}}\olr{D}^{-1/2}W_i\|.
\end{align*}
Therefore applying the above obtained estimates one can find  $C>0$ independent of $\ep$ and $\nu$ such that 
\begin{gather*}
\big|{\mathsf{Re}}\big(\op{\varLambda}\op{{\tilde H_{\ep}}}\op{\phi^{-n}}\op{\Phi_{\ep 1}}V^{\mu}, \op{\phi^{-n}}\op{\Phi_{\ep 1}}V^{\mu})\big|\\
\leq C\nu\|\op{\varLambda^{1/2}}\op{\phi^{-n}}\op{\Phi_{\ep 1}}V^{\mu}\|^2
+C'(\nu,\ep)\|\op{\varLambda^{1/2}}\op{\phi^{-n}}\olr{D}^{-1/4}V^{\mu}\|^2.
\end{gather*}
Since it follows from the same reasoning that
\begin{gather*}
\big|\big(\op{\varLambda}\op{\phi^{-n}}\op{f_{\ep}^{-1}\Phi_{\ep}}V^{\mu}, \op{\phi^{-n}}\op{\Phi_{\ep}}V^{\mu}\big)\big|\\
-\big(\op{\varLambda}\op{\phi^{-n}}\op{\Phi_{\ep 1}}V^{\mu}, \op{\phi^{-n}}\op{\Phi_{\ep 1}}V^{\mu}\big)\big|
\leq  c(\ep){\mathcal N}_{-1/4}(V^{\mu}).
\end{gather*}
We conclude finally that $-{\mathsf{Im}}\big(\op{\varLambda}\op{\phi^{-n}}{\tilde L}(\op{\Phi_{\ep}}V^{\mu}), \op{\phi^{-n}}\op{\Phi_{\ep}}V^{\mu}\big)$ is bounded by
\begin{equation}
\label{eq:cut:1}
\begin{split}
-(1-C\nu)\|\op{\varLambda^{1/2}}\op{\phi^{-n}}\op{\Phi_{\ep 1}}V^{\mu}\|^2
+c(\nu, \ep){\mathcal N}_{-1/4}(V^{\mu})\\
+{\mathsf{Re}}\big(\op{\varLambda}\op{\phi^{-n}}\op{\Phi_{\ep}}G^{\mu}, \op{\phi^{-n}}\op{\Phi_{\ep}}V^{\mu}\big).
\end{split}
\end{equation}
We fix $\nu_0$ such that $1-C\nu_0\geq 0$.  Since $|\dif_{\xi}^{\al}\olr{\mu \xi}^{-\bar n}|\leq C_{\al}\olr{\mu \xi}^{-\bar n}\olr{\xi}^{-|\al|}$ with $C_{\al}$ independent of $\mu>0$ we see  that
\[
\big|(\op{\Lambda}\op{\phi^{-n}}\op{\Phi_{\ep}}R^{\mu}V^{\mu}, \op{\phi^{-n}}\op{\Phi_{\ep}}V^{\mu})\big|\leq C{\mathcal E}_1(\op{\Phi_{\ep}}V^{\mu}).
\]
Therefore $\big|(\op{\Lambda}\op{w^{-n}}\op{\Phi_{\ep}}G^{\mu}, \op{w^{-n}}\op{\Phi_{\ep}}V^{\mu})\big|$ is bounded by 
\[
\ep_1{\mathcal E}_1(\op{\Phi_{\ep}}V^{\mu})+C_{\ep_1}t^{-2n+1}\|\op{\Phi_{\ep}}F^{\mu}\|_n^2+C{\mathcal E}_1(\op{\Phi_{\ep}}V^{\mu})
\]
for any $\ep_1>0$. Note that $D_t^jV^{\mu}\in {\mathcal H}_{0, 2n+7/4-j}(0,\delta)$, $j=0,1,\ldots$ hence $D_t^jV(0)$ exists in $H^{n+3/4}$ which is $0$ for $j=0,1, \ldots, n$ thus  $\lim_{t\to +0}t^{-n}\|V^{\mu}(t)\|_n=0$ for $\mu>0$. Applying \eqref{eq:dif:ene:fu}  to $\op{\Phi_{\ep}}V^{\mu}$ instead of $V$, choosing $\ep_1<c_1$ and then letting $\mu\to 0$  one concludes the proof.
\end{proof}
Applying $\olr{D}^s$ to \eqref{eq:cut:siki} and repeating similar arguments 
one obtains
\begin{prop}
\label{pro:fu:cut} For any $s\in \R$, any $0<\nu\leq \nu_0$ and any $\ep>0$ one can find $C>0$ such that for any $V$ with $\int_0^{\delta}t^{-2n}{\mathcal N}_{s-1/4}(V)dt<+\infty$ and  ${\tilde L}V\in {\mathcal H}_{-n+1/2, l}(0,\delta)$ with some $l$ we have
\begin{gather*}
{\mathcal E}_1(\olr{D}^s\op{\Phi_{\ep}}V)+\int_0^t\tau^{-2n}{\mathcal N}_s(\op{\Phi_{\ep}}V)d\tau\\
\leq C\big(\int_0^t\tau^{-2n+1}\|\op{\Phi_{\ep}}{\tilde L}V\|_{n+s}^2d\tau
+\int_0^t\tau^{-2n}{\mathcal  N}_{s-1/4}(V)d\tau\big),\;\;0\leq t\leq \delta.
\end{gather*}
\end{prop}
%

\subsection{Wave front set propagates with finite speed}
\label{sec:WF:hyoka:2}

\begin{lem}
\label{lem:yugen:denpa}
Assume $V\in {\mathcal H}_{-n-1/2, l_1+1}(0,\delta)$ and ${\tilde L}V\in {\mathcal H}_{-n+1/2, l_2}(0,\delta)$ and that $\op{\Phi_{\ep_0}}{\tilde L}V\in {\mathcal H}_{-n+1/2, n+s_0}(0,\delta)$ 
with some $l_1, l_2, s_0\in\R$, $\ep_0>0$. Then  for every $\ep>\ep_0$ we have $
\op{\Phi_{\ep}}V\in {\mathcal H}_{-n-1/2, s}(0,\delta)$ 
for all $s\leq s_0-5/4$. Moreover 
\begin{gather*}
\int_0^t \tau^{-2n-1}\|\op{\Phi_{\ep}}V(\tau)\|_s^2d\tau\leq C\int_0^t\big(\tau^{-2n-1}\|V(\tau)\|_{l_1+1}^2
+\tau^{-2n+1}\|{\tilde L}V(\tau)\|_{l_2}^2\big)d\tau\\
+C\int_0^t\tau^{-2n+1}\|\op{\Phi_{\ep_0}}{\tilde L}V(\tau)\|_{n+s_0}^2d\tau,\quad 0< t\leq \delta. 
\end{gather*}
\end{lem}
\begin{proof}We may assume $l_1\leq s_0$ otherwise nothing to be proved. Let $J$ be the largest integer such that $l_1+J/4\leq s_0$. Take $\ep_j>0$ such that  $\ep_0<\ep_1<\cdots<\ep_J=\ep$. We write $\Phi_{\ep_j}=\Phi_j$ and $f_j=f_{\ep_j}$ in this proof.  Inductively we show that
\begin{equation}
\label{eq:kinou:j}
\begin{split}
\int_0^t\tau^{-2n}{\mathcal N}_{l_1+j/4}(\op{\Phi_j}V)d\tau\leq C\int_0^t\tau^{-2n-1}\|V(\tau)\|_{l_1+1}^2d\tau\\
+C\int_0^t\tau^{-2n+1}\big\{\|{\tilde L}V(\tau)\|_{l_2}^2+\|\op{\Phi_{0}}{\tilde L}V(\tau)\|_{l_1+n+j/4}^2\big\}d\tau.
\end{split}
\end{equation}
Choose $\psi_j(x,\xi)\in S(1, g_0)$ so that ${\rm supp}\,{\psi_j}\subset \{f_j<0\}$ and $\{f_{j+1}<0\}\subset \{\psi_j=1\}$. Noting that $
\op{\Phi_{j+1}}{\tilde L}\,\op{\psi_j}=\op{\Phi_{j+1}\#\psi_j}{\tilde L}+\op{\Phi_{j+1}}[\tilde L,\op{\psi_j}]$ 
we apply Proposition \ref{pro:fu:cut} with $s=l_1+(j+1)/4$, $\Phi=\Phi_{j+1}$ and $V=\op{\psi_j}V$. Since $\Phi_{j+1}\#\psi_j-\Phi_{j+1}\in S^{-\infty}$ then $\|\op{\Phi_{j+1}}{\tilde L}\,\op{\psi_j}V\|_{l_1+(j+1)/4+n}^2$ is bounded by $
c\|\op{\Phi_{j+1}}{\tilde L}V\|_{l_1+(j+1)/4+n}^2+C(j)\|V\|_{l_1+1}^2$ 
and hence by
\begin{equation}
\label{eq:kino:uhen}
C(j)\big\{\|\op{\Phi_{0}}{\tilde L}V\|_{l_1+(j+1)/4+n}^2+\big\{\|{\tilde L}V\|_{l_2}^2+\|V\|_{l_1+1}^2\big\}
\end{equation}
because $\Phi_{j+1}-k_j\#\Phi_{0}\in S^{-\infty}$ with some $k_j\in S(1, g_0)$. Since $\psi_j-{\tilde k}_j\#\Phi_j\in S^{-\infty}$ with some ${\tilde k}_j\in S(1,g_0)$ it follows that
\[
{\mathcal N}_{l_1+j/4}(\op{\Phi_{j+1}}\op{\psi_j}V)\leq C{\mathcal N}_{l_1+j/4}(\op{\Phi_j}V)+C\|V\|_{l_1+1}^2.
\]
Consider ${\mathcal N}_{l_1+(j+1)/4}(\op{\Phi_{j+1}}\op{\psi_j}V)$. Noting that $\Phi_{j+1}\#\psi_j-\Phi_{j+1}\in S^{-\infty}$ the same reasoning shows that
\begin{equation}
\label{eq:j:to:j+1}
{\mathcal N}_{l_1+(j+1)/4}(\op{\Phi_{j+1}}V)
\leq C{\mathcal N}_{l_1+(j+1)/4}(\op{\Phi_{j+1}}\op{\psi_j}V)+C\|V\|_{l_1+1}^2.
\end{equation}
Multiply \eqref{eq:j:to:j+1} and \eqref{eq:kino:uhen} by $t^{-2n}$ and $t^{-2n+1}$ respectively and integrate it from $0$ to $t$ we conclude from Proposition  \ref{pro:fu:cut} that \eqref{eq:kinou:j} holds for $j+1$ and hence for $j=J$.  Since $l_1+J/4\leq s_0$, $l_1+J/4>s_0-1/4$ and $\|V\|_{s-1}/C\leq {\mathcal N}_s(V)$  the assertion follows.
\end{proof}
Let $\Gamma_i$ ($i=1, 2, 3 $) be open conic sets in $\R^d\times (\R^d\setminus\{0\})$ with relatively compact basis such that $\Gamma_1\Subset \Gamma_2\Subset \Gamma_3$. Take $h_i(x,\xi)\in S(1, g_0)$ with ${\rm supp}\,h_1\subset \Gamma_1$ and ${\rm supp}\,h_2\subset \Gamma_3\setminus \Gamma_2$. Consider a solution $V$ with  $V\in {\mathcal H}_{-n-1/2, l}(0,\delta)$  to the equation 
\[
{\tilde L}V=\op{h_1}F,\quad F\in {\mathcal H}_{-n+1/2, s}(0,\delta).
\]
\begin{prop}
\label{pro:yugen:denpa:V}Notations being as above. There exists $\delta'=\delta'(\Gamma_i)>0$ such that for any $r\in \R$ there is $C>0$ such that
\begin{equation}
\label{eq:til:kara:L}
\begin{split}
&\int_0^t\tau^{-2n-1}\|\op{h_2}V(\tau)\|_r^2d\tau\\
&\leq C\int_0^t\big\{\tau^{-2n+1}\|F(\tau)\|_s^2
+\tau^{-2n-1}\|V(\tau)\|_{l}^2\big\}d\tau,\quad  0< t\leq \delta'.
\end{split}
\end{equation}
\end{prop}
\begin{proof}
Let $f_{\ep}=t-\nu_0{\hat \tau}+\nu_0d_{\ep}(x,\xi; y,\eta)$ with a small ${\hat\tau}>0$. It is clear that there is ${\hat \ep}>0$ such that $
\{t\geq 0\}\cap\{f_{\hat \ep}\leq 0\}\cap\big(\R\times{\rm supp}\,h_1\big)=\emptyset$ 
for any $(y,\eta)\notin \Gamma_2$. Take ${\hat \ep}<{\tilde \ep}<{\hat\tau}$. It is also clear that one can find a finite number of $(y_i,\eta_i)\in \Gamma_3\setminus \Gamma_2$, $i=1,\ldots, M$ such that with $\delta'=\nu_0({\hat\tau}-{\tilde \ep})/2$ 
\begin{gather*}
\Gamma_3\setminus \Gamma_2\Subset \Big(\bigcup_{i=1}^M\{f_{\tilde \ep}(\delta', x,\xi; y_i,\eta_i)\leq 0\}\Big),\\
\{t\geq 0\}\cap\{f_{\tilde\ep}(t, x,\xi; y_i,\eta_i)\leq 0\}\cap\big(\R\times {\rm supp}\,h_1\big)=\emptyset.
\end{gather*}
Now $\Phi_{i\ep}$  is defined by \eqref{eq:cut:f} with $f_{ \ep}(t, x,\xi; y_i,\eta_i)$. Then since $\sum \Phi_{i{\tilde \ep}}>0$ on $[0,\delta']\times {\rm supp}\,h_2$ there is $k\in S(1, g_0)$ such that $h_2-k\sum\Phi_{i{\tilde\ep}}\in S^{-\infty}$. Noting that $\op{\Phi_{i{\hat \ep}}}\op{h_1}F\in {\mathcal H}_{-n+1/2, r}(0,\delta)$ for any $r\in \R$ we apply Lemma \ref{lem:yugen:denpa} with $\Phi_{\ep_0}=\Phi_{{\hat \ep}}$, $\Phi_{\ep}=\Phi_{i{\tilde \ep}}$ and $s_0=r+5/4$ to obtain
\begin{align*}
\int_0^t \tau^{-2n-1}\|\op{\Phi_{i{\tilde\ep}}}V(\tau)\|_r^2d\tau\leq C\int_0^t\tau^{-2n-1}\|V(\tau)\|_{l}^2d\tau
\\
+\int_0^t\tau^{-2n+1}\big(\|\op{\Phi_{i{\hat\ep}}}\op{h_1}F(\tau)\|_{2n+r+5/4}^2+\|F(\tau)\|_s\big)d\tau
\end{align*}
for $\|{\tilde L}V(\tau)\|_s\leq C\|F(\tau)\|_s$. Since $\Phi_{i{\hat\ep}}\#h_1\in S^{-\infty}$ summing up the above estimates over $i=1,\ldots, M$ one concludes the desired assertion.
\end{proof}
\begin{lem}
\label{lem:yugen:denpa:U}The same assertion as Proposition \ref{pro:yugen:denpa:V} holds for $L$.
\end{lem}
\begin{proof}Assume that $U\in {\mathcal H}_{-n-1/2, l}(0,\delta)$ satisfies $
LU=\op{h_1}F$ where $F\in {\mathcal H}_{-n+1/2, s}(0,\delta)$. 
Choose ${\tilde \Gamma}_i$ such that $\Gamma_1\Subset {\tilde \Gamma}_1\Subset {\tilde \Gamma}_2\Subset \Gamma_2\Subset\Gamma_3\Subset {\tilde \Gamma}_3$ and ${\tilde h}_i\in S(1, g_0)$ such that ${\rm supp}\,{\tilde h}_1\subset {\tilde \Gamma}_1$, ${\rm supp}\,{\tilde h}_2\subset {\tilde \Gamma}_3\setminus {\tilde\Gamma}_2$ and ${\tilde h}_i=1$ on the support of $h_i$. Recall that $L\,\op{T}=\op{T}\,{\tilde L}$ hence ${\tilde L}V=(I+\op{K})\op{T^{-1}}\op{h_1}F$ with $U=\op{T}V$. Since there is ${\tilde T}\in S(1,g)$ such that $(I+K)\#T^{-1}\#h_1-{\tilde h}_1{\tilde T}\in S^{-\infty}$ it follows from Proposition \ref{pro:yugen:denpa:V} (or rather its proof) that \eqref{eq:til:kara:L} holds with ${\tilde h}_2$ in place of $h_2$. Similarly since there is ${\tilde T}\in S(1, g)$ such that $h_2\#T-{\tilde h}_2{\tilde T}\in S^{-\infty}$ repeating the same arguments we conclude the assertion.
\end{proof}
Returning to ${\hat P}$  we have 
\begin{prop}
\label{pro:yugen:denpa:hatP} Notations being as above. Then there exists $\delta'=\delta'(\Gamma_i)>0$ such that for any $s$, $r\in \R$ there is $C$ such that for any $u$  with $D_t^ju\in {\mathcal H}_{-n-1/2, l+2-j}(0,\delta')$, $j=0,1,2$ with some $l$ satisfying  ${\hat P}u=\op{h_1}f$ 
where $f\in {\mathcal H}_{-n+1/2, s}(0,\delta')$ one has 
\begin{gather*}
\sum_{j=0}^2\int_0^t\tau^{-2n-1}\|\op{h_2}D_t^ju(\tau)\|_{r+2-j}^2d\tau\\
\leq C\big(\int_0^t\tau^{-2n+1}\|f(\tau)\|_s^2d\tau
+\sum_{j=0}^2\int_0^t\tau^{-2n-1}\|D_t^ju(\tau)\|_{l+2-j}^2d\tau\big),\;\;0< t\leq \delta'.
\end{gather*}
\end{prop}
Thanks to Theorem \ref{thm:pre:sonzai:s} for any $f\in {\mathcal H}_{-n+1/2,n+s}(0,\delta)$ there is a  unique solution $u\in {\mathcal H}_{-n-1/2, s+1}(0,\delta)$ to ${\hat P}u=f$ satisfying \eqref{eq:hat:P}. Denote this map  by
\[
{\hat G}: {\mathcal H}_{-n+1/2,n+s}(0,\delta)\ni f\mapsto u\in {\mathcal H}_{-n-1/2, s+1}(0,\delta).
\]
From Proposition \ref{pro:yugen:denpa:hatP} and Theorem \ref{thm:pre:sonzai:s} we conclude
\begin{prop}
\label{pro:yugen:denpa}Notations being as above and let $\Gamma_i$ $(i=1, 2, 3 )$ be open conic sets in $\R^d\times (\R^d\setminus\{0\})$ with relatively compact basis such that $\Gamma_1\Subset \Gamma_2\Subset \Gamma_3$ and  $h_i(x,\xi)\in S(1, g_0)$ with ${\rm supp}\,h_1\subset \Gamma_1$ and ${\rm supp}\,h_2\subset \Gamma_3\setminus \Gamma_2$. Then there exists $\delta'=\delta'(\Gamma_i)>0$ such that for any $r$, $s$ one can find $C>0$ such that
\begin{gather*}
\sum_{j=0}^2\int_0^t\tau^{-2n-1}\|\op{h_2}D_t^j{\hat G}\,\op{h_1}f(\tau)\|_{r-j}^2d\tau
\leq C\int_0^t\tau^{-2n+1}\|f(\tau)\|_{s}^2d\tau
\end{gather*}
for $0< t\leq \delta'$ and for any $f\in {\mathcal H}_{-n+1/2, s}(0,\delta')$.
\end{prop}
Denote by ${\mathcal H}_{n,s}^{*}(0,\delta]$ the set of all $f$ with $t^{n}\langle{D}\rangle^{s}f(t,\cdot)\in L^2((0,\infty)\times \R^d)$ such that $f=0$ for  $t\geq \delta$. Thanks to Theorem \ref{thm:pre:sonzai:ad} for any $f\in {\mathcal H}^{*}_{n+1/2,n+s}(0,\delta]$ there is a unique solution $u\in {\mathcal H}_{n-1/2, s+1}(0,\delta)$ to ${\hat P^*}u=f$ with $D_t^ju(\delta)=0$, $j=0,1,2$ satisfying \eqref{eq:hat:P:ad} hence $u\in {\mathcal H}^*_{n-1/2, s+1}(0,\delta]$. Denote this map  by
\[
{\hat G^*}: {\mathcal H}^{*}_{n+1/2,n+s}(0,\delta]\ni f\mapsto u\in {\mathcal H}^{*}_{n-1/2, s+1}(0,\delta].
\]
Repeating similar arguments proving Proposition \ref{pro:yugen:denpa} one obtains
\begin{prop}
\label{pro:yugen:denpa:ad} Notations being as in Proposition \ref{pro:yugen:denpa}. There exists $\delta'=\delta'(\Gamma_i)>0$ such that for any $r$, $s$ one can find $C>0$ such that
\begin{gather*}
\sum_{j=0}^2\int_t^{\delta'}\tau^{2n-1}\|\op{h_2}D_t^j{\hat G^*}\,\op{h_1}f(\tau)\|_{r-j}^2d\tau
\leq C\int_t^{\delta'}\tau^{2n+1}\|f(\tau)\|_s^2d\tau
\end{gather*}
for $0< t\leq \delta'$ and $f\in {\mathcal H}^{*}_{n+1/2, s}(0,\delta']$.
\end{prop}
\begin{remark}
\label{rem:hatP:yugen:n}\rm As already remarked in Remark \ref{rem:hatP:sonzai:n}, it is clear from the proof  that Theorems \ref{thm:pre:sonzai:s}, \ref{thm:pre:sonzai:ad} and Propositions \ref{pro:yugen:denpa} and \ref{pro:yugen:denpa:ad} hold for any  $n'$ greater than $n$.
\end{remark}
%

\subsection{Remarks on propagation of singularities}

In this section  we give a more precise picture of the propagation of wave front set $WF(u)$ of $u$ applying the same arguments in Sections \ref{sec:WF:hyoka:1} and \ref{sec:WF:hyoka:2}. Denote $X=[0,\delta)\times U$ and $\overset{\circ}{X}=(0,\delta)\times U$ and by $\Sigma$, $\Sigma_1$ the set of characteristics and simple characteristics of $p$ respectively. As explained in the Introduction every characteristic in $T^*{\overset{\circ}{X}}\setminus 0$ is at most double and double characteristic is effectively hyperbolic. Let ${\mathcal U}$ be an open conical set in $\subset T^*\Xmaru\setminus 0$. According to \cite{Mel} we say a continuous curve $\gamma(t): (0, a]\to {\mathcal U}\cap \Sigma$, parametrized by $t$, a generalized bicharacteristic if $\gamma^{-1}(\Sigma\setminus\Sigma_1)=\{t_i\}$ is discrete in $(0, a]$ and $\gamma$ is a parametrized smooth bicharacteristic of $p$ on $(t_i, t_{i+1})$. In the present case a generalized bicharacteristic is described rather easily. Let $\rho\in {\mathcal U}$ be a double characteristic. Then from \cite{Ni:book} one can find a conical open set ${\mathcal V}\ni \rho$ and a smooth function $\psi(x,\xi)$, homogeneous of degree $0$ in $\xi$,  such that the double characteristics of $p$ in ${\mathcal V}$ are contained in $\{t=\psi\}$ and  
there are exactly two smooth bicharacteristics enter $\rho$  transversally to $t=\psi$ in the direction of decreasing $t$ (also exactly two in the direction of increasing $t$) 
(e.g.\cite{KoN}). Therefore $\gamma$ consists of segments of smooth bicharacteristic of $p$, the only end points of theses segments lying in $\Sigma$ and a transition to one of two segments takes place there. Let $\gamma_i$, $i=1, 2$ be two bicharacteristic segments entering to $\rho$ in the direction of decreasing $t$ (or  two segments in the direction of increasing $t$) then from \cite[Theorem 2.1]{Ni1:bis}, \cite[Theorem 1.7]{Mel} we have $\rho\not\in WF(u)$ if $\gamma_i\not\in WF(u)$, $i=1, 2$ and $\rho\not\in WF(Pu)$. 
\begin{lem}
\label{lem:bich:lim}If $\gamma(t)$ is a generalized bicharacteristic then $\lim_{t\to+0}\gamma(t)$ exists.
\end{lem}
\begin{proof}Write $p=\prod_{j=1}^m(\tau-\tau_j(t, x, \xi))$ then  thanks to \cite{Bron} for any $0<\delta_1<\delta$ and $U'\Subset U$ there is $L>0$ such that
\[
|\nabla_x\tau_j(t, x,\xi)|/|\xi|,\; |\nabla_{\xi}\tau_j(t, x,\xi)|\leq L,\; (t, x, \xi)\in [0,\delta_1]\times U'\times \R^d,\;1\leq j\leq m.
\]
In each $(t_i, t_{i+1})$ it is clear that $\gamma(t)=(t, x(t), \tau(t), \xi(t))$ is a bicharacteristic of some $\tau-\tau_j(t, x, \xi)$ and hence $
dx/d t=-\nabla_{\xi}\tau_j(t, x, \xi)$ and $d\xi/d t=\nabla_{x}\tau_j(t, x, \xi)$. This shows that $x(t)$ and $\xi(t)$ are uniformly Lipschitz continuous in $(0, a]$ with the Lipschitz constant $L$, though $\tau(t)$ is not Lipschitz continuous in $(0, a]$ in general. Then $\lim_{t\to +0}x(t)$ and $\lim_{t\to+0}\xi(t)$ exist. Since $\tau(t)=\tau_k(t, x(t), \xi(t))$ for some $k$ and $\tau_k(t, x, \xi)$ are continuous in $X\times \R^d$ then $\lim_{t\to +0}\tau(t)$ also exists.
\end{proof}
Denote 
\[
K^{-}(\rho)=\bigcup_{\gamma}\gamma(t) 
\]
where $\gamma$ varies over all generalized bicharacteristics such that $\gamma(a)=\rho$, extended to $t=0$ according to Lemma \ref{lem:bich:lim}. Note that for any $\varep>0$ the set $K^{-}(\rho)\cap\{t=\varep\}$ consists of finite number of characteristic points. Thanks to the propagation results near double  effectively hyperbolic characteristics mentioned above, if 
\begin{equation}
\label{eq:jyubun}
WF(u)\cap K^{-1}(\rho)\cap\{t=\varep\}=\emptyset
\end{equation}
with some $\varep>0$ then we have $\rho\not\in WF(u)$.
\begin{them}
\label{thm:pro:WF}Notations being as above and let $K_0(\rho)=\pi\big(K^{-}(\rho)\cap\{t=0\}\big)$ where $\pi:(t, x,\tau,\xi)\to (x,\xi)$ is the projection. Let $Pu\in C^{\infty}(X)$ and $D_t^ju(0, \cdot)=u_j$. If $K_0(\rho)\cap (\cup_{j=0}^mWF(u_j))=\emptyset$ then $\rho\not\in WF(u)$.
\end{them}
\noindent
We give a sketch of the proof. Since characteristics are at most triple we may assume that $m=3$.  It suffices to show that \eqref{eq:jyubun} holds with some $\varep>0$. In \eqref{eq:f:ep:teigi} we take $T=2\nu$. Since $K_0(\rho)$ is compact one can find $0<\nu<\nu_0$ and finitely many $(y_i, \eta_i)$, $i=1,\ldots, k$  such that $K_0(\rho)\subset \cup_{i=1}^k\{f_0(0, x, \xi, y_i, \eta_i)<0\}$ and $\{f_0(0, x, \xi, y_i, \eta_i)\leq 0\}\cap \big(\cup_{j=0}^2 WF(u_j)\big)=\emptyset$. Let $\Phi_{i,\ep}$  be the symbol defined by \eqref{eq:cut:f} with $f_{\ep}(t, x, \xi, y_i, \eta_i)$ then we have
\begin{equation}
\label{eq:cut:syoki}
\op{\Phi_{i,\ep}(0)}u_j\in H^s,\quad 0\leq j\leq 2,\;\;1\leq i\leq k,\;\;\forall s\in\R
\end{equation}
for enough small $\ep>0$. Using the same reduction as in the proof of Theorem \ref{thm:main:bis} it suffices to study $Pu=f$ where $D_t^ju\in {\mathcal H}_{-n-1/2, l_1-j}(0,\delta)$, $0\leq j\leq 2$  and $f\in  {\mathcal H}_{-n-1/2, l_2}(0,\delta)$ with some $l_j\in\R$ such that 
\[
 \op{\Phi_{i,\ep}}f\in {\mathcal H}_{-n-1/2, s}(0,\delta), \quad 1\leq i\leq k,\;\;\forall s\in\R
 \]
 for enough small $\ep>0$ which follows from \eqref{eq:cut:syoki}. Repeating the same arguments as in Sections \ref{sec:WF:hyoka:1} and \ref{sec:WF:hyoka:2} we conclude that $
 \op{\Phi_{i, 0}}D_t^ju\in L^2((0,\delta'), H^s)$ for any $1\leq i\leq k$, $0\leq j\leq 2$ and any $s\in\R$ with some $\delta'>0$. Finally, using the equation we have
 \[
  \op{\Phi_{i, 0}}D_t^ju\in L^2((0,\delta'), H^s),\quad 1\leq i\leq k, \;\;\forall s\in\R,\;\;\forall j\in\N.
  \]
Since $\Phi_{i,0}(t, y_i, \eta_i)<0$ for $t<\nu$ this proves \eqref{eq:jyubun} with some $\varep>0$.
\qed
 %
 
\section{Proof of Theorem \ref{thm:main}}
\label{sec:syuteiri}

Applying Proposition \ref{pro:yugen:denpa} we prove Theorem \ref{thm:main} following \cite{Ni1}, \cite{Ni:book}. 

\subsection{Solution operator with finite speed of propagation}

Consider
\begin{equation}
\label{eq:P:kata}
P=D_t^m+\sum_{j=1}^{m}a_{j}(t, x, D)\langle{D}\rangle^jD_t^{m-j}
\end{equation}
which is a differential operator in $t$ with coefficients $a_{j}\in S^0$. We say that $G$ is a solution operator  for $P$ with finite propagation speed of wave front set (which we abbreviate to ``solution operator with finite speed of propagation" from now on) with loss of $(n, l)$ derivatives if $G$ satisfies the following conditions:
\begin{enumerate}
\item[(i)] There exists $\delta>0$ such that for any $s\in\R$ there is $C>0$ such that for $f\in {\mathcal H}_{-n+1/2, s}(0,\delta)$ we have  $P G f=f$ and 
\begin{gather*}
\sum_{j=0}^{m-1}\int_0^t{\tau}^{-2n-1}\|D_t^jGf(\tau)\|_{-l+s+m-j}^2d\tau
 \leq C\int_0^t\tau^{-2n+1}\|f(\tau)\|_s^2d\tau.
\end{gather*}
\item[(ii)] For  any $h_j(x,\xi)\in S(1, g_0)$, $j=1,2$ with ${\rm supp}\,h_2\Subset (\R^d\times\R^d)\setminus {\rm supp}\,h_1$ there exists $\delta'>0$ such that for any $r$, $s\in\R$ there is $C>0$ such that 
\begin{gather*}
\sum_{j=0}^{m-1}\int_0^t\tau^{-2n-1}\|\op{h_2}D_t^jG\,\op{h_1}f(\tau)\|_{r-j}^2d\tau
\leq C\int_0^t\tau^{-2n+1}\|f(\tau)\|_s^2d\tau
\end{gather*}
for $f\in {\mathcal H}_{-n+1/2, s}(0,\delta')$ and $0<t\leq \delta'$.
\end{enumerate}
Let $P_1$ and $P_2$ be two operators of the form \eqref{eq:P:kata}. We say
\[
P_1\equiv P_2 \quad\text{at}\;\;({\hat x}, {\hat \xi})
\]
if there exist $\delta'>0$ and a conic neighborhood $W$ of $({\hat x}, {\hat \xi})$ such that 
\begin{equation}
\label{eq:mod:teigi}
P_1-P_2=\sum_{j=1}^{m}R_j(t, x, D)\lrD^jD_t^{m-j}
\end{equation}
with $R_j\in S^0$ which are in $S^{-\infty}(W)$ uniformly in $0\leq t\leq \delta'$.
\begin{them}
\label{thm:sonzai:ippan} Assume that for any $({\hat x},\eta)$, $|\eta|=1$ one can find  $P_{\eta}$ of the form \eqref{eq:P:kata}  for which there is a solution operator with finite speed of propagation with loss of $(n, \ell(\eta))$ derivatives such that $P\equiv P_{\eta}$ at  $({\hat x},\eta)$. Then  there exist $\delta>0$, $\ell=\sup\ell(\eta)$ and a neighborhood $U$ of ${\hat x}$ such that for every $f\in {\mathcal H}_{-n+1/2,s+\ell}(0,\delta)$ there exists $u$ with $D_t^j u\in {\mathcal H}_{-n, s+m-j}(0,\delta)$, $0\leq j\leq m-1$, satisfying $
Pu=f$ in $(0,\delta)\times U$.
\end{them}
\begin{proof}By assumption $P_{\eta}$ has a solution operator $G_{\eta}$ with finite speed of propagation with loss of $(n, \ell(\eta))$ derivatives. There are finite open conic neighborhood $W_i$ of $({\hat x}, \eta_i)$ such that $\cup_iW_i\supset \Omega\times (\R^d\setminus \{0\})$, where $\Omega$ is a neighborhood of ${\hat x}$, and $P\equiv P_{\eta_i}$ at $({\hat x}, \eta)$ with $W=W_i$ in \eqref{eq:mod:teigi}. Now take another open conic covering $\{V_i\}$ of $\Omega\times (\R^d\setminus\{0\})$ with $V_i\Subset W_i$, and a partition of unity $\{\alpha_i(x,\xi)\}$ subordinate to $\{V_i\}$ so that $
\sum_i\alpha_i(x,\xi)=\alpha(x)$ 
where $\alpha(x)=1$  in a neighborhood of ${\hat x}$. Define
\[
G=\sum_iG_{\eta_i}\alpha_i
\]
then denoting $P-P_{\eta_i}=R_i$ we have
\[
P G f=\sum_iPG_{\eta_i}\alpha_if=\sum_iP_{\eta_i}G_i\alpha_if
+\sum_iR_iG_i\alpha_if
=\alpha(x)f-Rf
\]
where $R=\sum_i R_iG_{\eta_i}\alpha_i$. Then 
\[
\int_0^t\tau^{-2n-1}\|Rf(\tau)\|_{s+\ell}^2d\tau\leq C\int_0^t\tau^{-2n+1}\|f(\tau)\|_{s+\ell}^2d\tau
\]
for $0\leq t\leq \delta''$ with some $\delta''>0$ by the condition (ii) where $\ell=\max_i\ell(\eta_i)$. Choosing $0<\delta_1\leq \delta''$ such that $\delta_1^2C\leq 1/2$ one has
\[
\int_0^t\tau^{-2n+1}\|Rf(\tau)\|_{s+\ell}^2d\tau\leq \frac{1}{2}\int_0^t\tau^{-2n+1}\|f(\tau)\|_{s+\ell}^2d\tau,\quad 0< t\leq \delta_1
\]
for $f\in {\mathcal H}_{-n+1/2, s+\ell}(0,\delta)$. With $
S=\sum_{k=0}^{\infty}R^k$ 
one has $Sf\in {\mathcal H}_{-n+1/2, s+\ell}(0,\delta_1)$ and $\int_0^t\tau^{-2n+1}\|Sf(\tau)\|_{s+\ell}^2d\tau\leq 2\int_0^t\tau^{-2n+1}\|f(\tau)\|_{s+\ell}^2d\tau$. Let $\gamma(x)\in C_0^{\infty}(\R^d)$ be equal to $1$ near ${\hat x}$ such that ${\rm supp}\,\gamma\Subset \{\alpha=1\}$. Since $\gamma(\alpha-R)S=\gamma(I-R)S=\gamma$ it follows that $
\gamma(x)PGSf=\gamma(x)f$, that is
\[
P\big(GSf\big)=f\quad \text{on}\;\;\{\gamma(x)=1\}.
\]
With $u=G S f$ one has
\begin{gather*}
\sum_{j=0}^{m-1}\int_0^t\tau^{-2n-1}\|D_t^ju(\tau)\|_{s+m-j}^2d\tau\leq C\int_0^t\tau^{-2n+1}\|Sf(\tau)\|_{s+\ell}^2d\tau
\end{gather*}
which proves the assertion.
\end{proof}
We define a solution operator with finite speed of propagation for $P^*$ with obvious modifications.
\begin{them}
\label{thm:sonzai:ippan:ad} Assume that for every $({\hat x},\eta)$, $|\eta|=1$ one can find  $P^*_{\eta}$ of the form \eqref{eq:P:kata}  such that $P^*\equiv P^*_{\eta}$ at $({\hat x},\eta)$ for which solution operator with finite speed of propagation with loss of $(n, \ell(\eta))$ derivatives exists. Then  there exist $\delta>0$, $\ell=\sup\ell(\eta)$ and a neighborhood $U$ of ${\hat x}$ such that for every $f\in {\mathcal H}^{*}_{n+1/2,s+\ell}(0,\delta]$ there exists $u$ with $D_t^j u\in {\mathcal H}^{*}_{n-1/2, s+m-j}(0,\delta]$, $0\leq j\leq m-1$, satisfying $P^*u=f$ in $ (0,\delta)\times U$.
\end{them}
%

\subsection{Local existence  and uniqueness}

First consider a third order operator $P$ of the form \eqref{eq:takata}. To reduce $P$ to the case $a_1(t, x, D)=0$ we apply a Fourier integral operator, which is actually the solution operator $S(t',t)$ of the Cauchy problem 
\[
D_tu+a_1(t, x, D)u=0,\quad u(t', x)=\phi(x)
\]
such that $S(t', t)\phi=u(t)$ then it is clear that $S(t,0)(D_t+a_1)S(0, t)=D_t$.  Then ${\tilde P}=S(t,0)PS(0,t)$ has the form \eqref{eq:takata} with $a_1=0$ (e.g. \cite{Ego}, \cite{MT}). Assume that ${\tilde P}$ has a solution operator with finite speed of propagation  ${\tilde G}$ with loss of $(n,\ell)$ derivatives. Then one can show that $G=S(0,t){\tilde G} S(t,0)$ is a solution operator for $P$ with finite speed of propagation with loss of $(n,\ell)$ derivatives. 

Let $|\eta|=1$ be given. Assume that $p$ has a triple characteristic root ${\bar\tau}$ at $(0,0,\eta)$ and $(0,0,{\bar\tau},\eta)$ is effectively hyperbolic. Theorem \ref{thm:pre:sonzai:s} and Proposition \ref{pro:yugen:denpa} imply that ${\hat P}$, which coincides with the original $P$ in $W_M$, given by \eqref{eq:conic:nbd},  has a solution operator with finite speed of propagation with loss of $(n, n+2)$ derivatives. 

Next assume that $p$ has a double characteristic root ${\bar\tau}$ at $(0,0,\eta)$ such that $(0,0,{\bar\tau},\eta)$ is effectively hyperbolic characteristic if it is a singular point. Note that  one can write
\[
p(t, x,\tau,\xi)=\big(\tau+b(t, x,\xi)\big)\big(\tau^2+a_1(t, x,\xi)\tau+a_2(t, x,\xi)\big)=p_1p_2
\]
in a conic neighborhood of $(0,0,\eta)$ where $p_1(0,0,{\bar\tau},\eta)\neq 0$. There exist ${\hat P}_i$ such that $
P\equiv {\hat P}_1\cdot {\hat P}_2$ at $(0, \eta)$ 
where  the principal symbol of ${\hat P}_i$ coincides with $p_i$ in a conic neighborhood of $(0,0,\eta)$. If ${\hat P}_i$ has a solution operator with finite speed of propagation $G_i$ with loss of $(n, \ell_i)$ derivatives then one can see that $G_2G_1$ is a solution operator with finite speed of propagation for ${\hat P}_1\cdot{\hat P}_2$ with loss of $(n, \ell_1+\ell_2)$ derivatives. Consider the case that $(0,0,{\bar\tau},\eta)$ is a singular point. Then 
$F_p(0,0,{\bar\tau},\eta)=c\,F_{p_2}(0,0,{\bar\tau},\eta)$ 
with some $c\neq 0$ and $(0,0,{\bar\tau},\eta)$ is effectively hyperbolic characteristic of $p_2$. Write $p_2$ as 
\begin{equation}
\label{eq:niji:tokusei}
p_2(t, x,\tau,\xi)=\tau^2-a(t, x,\xi)|\xi|^2
\end{equation}
such that ${\bar \tau}=0$ is a double characteristic and $(0, 0, 0, \eta)$ is a singular point. To apply earlier results \cite{Ni2, Ni:book} on operators with double effectively hyperbolic characteristics  we need some modifications because $a( t, x,\xi)$ is assumed to be nonnegative only in $t\geq 0$ side in the present case. One can improve \cite[Lemma 1.2.2]{Ni:book} to
\begin{prop}
\label{pro:niji:jiko}Assume that $a(t, x, \xi)$ is smooth in some conic neighborhood of $(0, 0, \eta)$, homogeneous of degree $0$ in $\xi$,  and nonnegative in $t\geq 0$ and $(0, 0, 0, \eta)$ is a effectively hyperbolic singular point of $p_2=0$. Then there exist a smooth function $\psi(x,\xi)$ in a conic neighborhood $V$ of $(0,\eta)$ and constants $0<\kappa<1$, $c>0$ such that
\begin{equation}
\label{eq:seigen:niji}
\{\psi, a\}^2\leq 4\,\kappa\,a,\quad a(t, x, \xi)\geq c\,\min{\big\{t^2, (t-\psi(x,\xi))^2\big\}}
\end{equation}
for $(x,\xi)\in V$, $t\geq 0$ where $\psi(x,\xi)$ satisfies $\big|\dif_x^{\al}\dif_{\xi}^{\be}\psi\big|\precsim \langle{\xi}\rangle^{-|\be|}$.  
\end{prop}
Indeed that the same time function given in \cite{Ni2:bis} under the assumption $a(t, x, \xi)\geq 0$ in a full neighborhood in $t$,  denoted by $Y(t, x,\xi)$ there, satisfies \eqref{eq:seigen:niji} can be proved applying \cite[Theorem 1.1]{Ni2:bis}. 
Then repeating similar arguments as in  \cite{Ni2, Ni:book} 
 we conclude that there is a solution operator  with finite speed of propagation for ${\hat P}_2$. Since ${\hat P}_1$ is a first order operator with real principal symbol $p_1$ it is easy to see that ${\hat P}_1$ has a solution operator with finite speed of propagation. Therefore $P$ has a solution operator  with finite speed of propagation. Turn to the case that $(0,0,0,\eta)$ is not a singular point. It is easy to see that $(0,0,0,\eta)$ is not a singular point implies $\dif_t a(0,0,\eta)>0$, which is the case that ${\hat P}_2$ is a hyperbolic operator of principal type and some detailed discussion is found in \cite[Chapter 23.4]{Hobook}. It is easily proved that ${\hat P}_2$ has a solution operator with finite speed of propagation, because it suffices to employ the weight $t^{-n}$ ($\phi^{-n}$ is now absent) in order to obtain weighted energy estimates.

Turn to the general case. Let $|\eta|=1$ be arbitrarily fixed. Write $
p(0,0,\tau,{\eta})=\prod_{j=1}^r\big(\tau-\tau_j)^{m_j}$ 
where $\sum m_j=m$ and $\tau_j$ are real and  distinct from each other, where $m_j\leq 3$ which follows from the assumption. There exist $T>0$ and a conic neighborhood $U$ of $(0,\eta)$ such that one can write
\begin{align*}
&p(t, x,\tau,\xi)=\prod_{j=1}^rp^{(j)}(t, x,\tau,\xi),\\
&p^{(j)}(t, x,\tau,\xi)= \tau^{m_j}+a_{j,1}(t, x,\xi)\tau^{m_j-1}+\cdots+ a_{j,m_j}(t, x,\xi)
\end{align*}
for $(t, x,\xi)\in (-T,T)\times U$ where $a_{j, k}(t, x,\xi)$ are real valued, homogeneous of degree $k$ in $\xi$ and $p^{(j)}(0,0,\tau,\eta)=(\tau-\tau_j)^{m_j}$ and $p^{(j)}(t, x,\tau,\xi)=0$ has only real roots in $\tau$ for $(t, x,\xi)\in [0, T)\times U$.  If $(0,0,\tau_j,\eta)$ is a singular point of $p$, and necessarily $m_j\geq 2$, then $(0,0,\tau_j,\eta)$ is a singular point of $p^{(j)}$ and it is easy to see $
F_p(0,0,\tau_j,\eta)=c_j F_{p^{(j)}}(0,0,\tau_j,\eta)$ 
with some $c_j\neq 0$ and hence $F_{p^{(j)}}(0,0,\tau_j,\eta)$ has non-zero real eigenvalues if $F_{p}(0,0,\tau_j,\eta)$ does and vice versa. It is well known  that one can find $P^{(j)}$ such that
\[
P\equiv P^{(1)}P^{(2)}\cdots P^{(r)}\quad \text{at}\;\;(0, \eta) 
\]
where $P^{(j)}$ are operators of the form \eqref{eq:P:kata} with $m=m_j$ whose principal symbol coincides with $p^{(j)}$  in some conic neighborhood of $(0,0,\eta)$. Since  each $P^{(j)}$ has a solution operator with finite speed of propagation with loss of $(n_j, n_j+2)$ derivatives thanks to Theorem \ref{thm:pre:sonzai:s} and Proposition \ref{pro:yugen:denpa} hence so does $P$ with loss of $(n, r(n+2))$ derivatives with $n=\max_j n_j$ noting Remark \ref{rem:hatP:yugen:n}. Therefore Theorem \ref{thm:main} results from Theorem \ref{thm:sonzai:ippan}. 

\medskip

Repeating a parallel arguments to the existence proof for $P$ above we obtain
\begin{them}
\label{thm:main:ad}  Under the same assumption as in Theorem \ref{thm:main} there exist $\delta>0$, a neighborhood $U$  of the origin and $n> 0$, $\ell>0$ such that  for any $s\in\R$ and any $f\in {\mathcal H}^{*}_{n+1/2, s}(0,\delta]$ there exists $u$ with $D_t^ju\in {\mathcal H}^{*}_{n-1/2, -\ell+s+m-j}(0,\delta]$, $0\leq j\leq m-1$, satisfying $
P^*u=f$ in $(0,\delta)\times U$.
\end{them}
Now we prove a local uniqueness result for the Cauchy problem for $P$ applying Theorem \ref{thm:main:ad}. From the assumption one can find  a neighborhood $W$ of the origin of $\R^d$ and $T>0$ such that every multiple characteristic of $p$ on $(t, x, \xi)\in (0,T)\times W$ is at most double and double characteristic is effectively hyperbolic.  Thanks to \cite[Main Theorem]{KNW} there exists ${\hat c}>0$ such that  for any solution $v$ to $P^*v=f$ vanishing in $t\geq \delta'$ with $f\in C_0^{\infty}((0,\delta')\times \{|x|< \varep\})$ ($\delta'\leq  T$) one has
\[
{\rm supp}_x v(t, \cdot)\subset \{|x|\leq \varep+{\hat c}\,\delta'\},\quad 0< t\leq \delta'.
\]
Now assume that $u$  satisfies $Pu=0$ in $(0,\delta)\times U$ and $\dif_t^ku(0,x)=0$ for all $k$. Choose $\varep>0$ and $\delta'>0$ such that $\{|x|\leq \varep+{\hat c}\,\delta'\}\subset U$, $\delta'\leq \delta$. Then we see
\begin{align*}
0=\int_0^{\delta'}\big(Pu, v\big)dt=\int_0^{\delta'}\big(u, P^*v\big)dt=\int_0^{\delta'}\big(u, f\big)dt.
\end{align*}
Since $f\in C_0^{\infty}((0,\delta')\times \{|x|< \varep\})$ is arbitrary, we conclude that 
\[
u(t, x)=0,\quad (t, x)\in (0,\delta')\times \{|x|\leq \varep\}.
\]
\begin{them}
\label{thm:local:itii}Assume \eqref{eq:cond:hyp} and  that every  singular point $(0,0,\tau,\xi)$,  $|(\tau,\xi)|\neq0$ of $p=0$ is effectively hyperbolic. 
If $u(t,x)\in C^{\infty}([0,\delta)\times U)$ satisfies $Pu=0$ in $[0,\delta)\times U$ and $\dif_t^k u(0, x)=0$ for all $k$ then $u=0$ in a neighborhood of $(0,0)$.
\end{them}

\bigskip
\noindent
{\bf Acknowledgment.} This work was supported by 
JSPS KAKENHI Grant Number JP20K03679.

%

\end{document}